\documentclass{amsart}

\numberwithin{equation}{section}

\usepackage[english]{babel}
\usepackage[letterpaper,top=2cm,bottom=2cm,left=3cm,right=3cm,marginparwidth=1.75cm]{geometry}

\usepackage{amsmath,amssymb,amsthm,mathtools}
\usepackage{mathabx}
\usepackage{enumitem}
\usepackage{mathrsfs}
\usepackage{scalerel,stackengine}
\usepackage{booktabs}
\usepackage{tabularx}
\usepackage{xfrac}
\usepackage{indentfirst}
\usepackage{cases}
\usepackage{graphicx}
\usepackage[colorlinks=true, linkcolor=blue, citecolor=blue, urlcolor=blue]{hyperref}

\stackMath

\newcommand{\dd}{\,d}
\newcommand{\ang}[1]{\langle #1\rangle}
\newcommand{\eps}{\varepsilon}

\newcommand{\spar}{\scriptscriptstyle /\!/}
\newcommand\reallywidehat[1]{%
\savestack{\tmpbox}{\stretchto{%
  \scaleto{%
    \scalerel*[\widthof{\ensuremath{#1}}]{\kern-.6pt\bigwedge\kern-.6pt}%
    {\rule[-\textheight/2]{1ex}{\textheight}}%
  }{\textheight}%
}{0.5ex}}%
\stackon[1pt]{#1}{\tmpbox}%
}

\theoremstyle{plain}
\newtheorem{theorem}{Theorem}[section]

\newtheorem{prop}[theorem]{Proposition}
\newtheorem{lemma}[theorem]{Lemma}

\newtheorem{cor}[theorem]{Corollary}

\theoremstyle{definition}
\newtheorem{definition}[theorem]{Definition}

\newtheorem{remark}[theorem]{Remark}

\title[Sharp Thresholds for Wave Kinetic Equations]{Sharp Decay Thresholds in Weighted $L^\infty$ for Wave Kinetic Equations with Power-Law Dispersion}

\author{Xilu Zhu}
\address{Department of Mathematics, University of Michigan, Ann Arbor, MI, USA}
\email{xiluzhu@umich.edu}


\begin{document}

\begin{abstract}
We study four-wave kinetic equations in space dimension three with power-law dispersion $\omega(p)=|p|^a$ and collision kernels with high-frequency growth measured by $\beta$. In weighted $L^\infty$ spaces, we identify the sharp decay threshold
$$
s_c=4\beta+3-\frac a2.
$$
For $s>s_c$, we prove local well-posedness by establishing trilinear bounds for the full gain-loss collision operator. For $s<s_c$, we prove ill-posedness by constructing data concentrated near a high-low-low-high resonant configuration. This threshold captures the balance between the high-frequency strength of the kernel and the geometry of the resonant manifold. The proof also shows that gain-loss cancellations are essential in the most delicate regimes.
\end{abstract}

\maketitle

\tableofcontents

\section{Introduction}

\subsection{Introduction of the Problem}
\ \par
Kinetic equations are fundamental models in statistical mechanics. They describe the effective evolution of systems with a large number of interacting particles or waves, where the microscopic dynamics are too complicated to track directly. Classical examples include the Boltzmann equation for dilute gases and wave kinetic equations arising in weak turbulence theory. In the latter setting, many weakly nonlinear dispersive equations are expected, under suitable random phase assumptions and kinetic scaling limits, to lead to effective kinetic descriptions for the spectral statistics of the waves.

The rigorous derivation of kinetic equations is a major problem in mathematical physics. For particle systems, the derivation of the Boltzmann equation from microscopic dynamics goes back to the classical work of Lanford, and recent work of Deng, Hani, and Ma has obtained a long-time derivation from hard sphere dynamics \cite{Lanford1975,DengHaniMa2024}. For wave systems, important progress toward the derivation of wave kinetic equations was made by Buckmaster, Germain, Hani, and Shatah \cite{BuckmasterGermainHaniShatah2021}, by Collot and Germain \cite{CollotGermainHomogeneous,CollotGermainLonger} and by Deng and Hani \cite{DengHani2021, DengHani2023FullRange}. The strongest results in this direction are due to Deng and Hani, who first obtained a full derivation of the wave kinetic equation from the cubic NLS at the kinetic time scale \cite{DengHani2023Full}, and later extended the justification to arbitrarily long times covering the full lifespan of the limiting WKE \cite{DengHani2023LongTime}. We also mention related derivation results for other dispersive models, including KdV-type equations, stochastic KP equations, and inhomogeneous kinetic limits \cite{Ma2022,StaffilaniTran2021,Faou2020,AmpatzoglouCollotGermain2025,HaniShatahZhu2024,HannaniRosenzweigStaffilaniTran2022,VassilevWu2026FullFPUT,Vassilev2025OneDimensional,Wu2025FPUT}.

Once a kinetic equation is derived or proposed as an effective model, it is natural to study the equation itself. This analytic problem is independent of the derivation question: one wants to know in which spaces the collision operator is well defined, whether the Cauchy problem is locally well posed, and where such a theory breaks down. This is particularly delicate for wave kinetic equations, since the collision operator is supported on resonant sets and its mapping properties depend strongly on both the interaction kernel and the dispersion relation.

There has also been substantial progress on the analysis of wave kinetic equations themselves. For the classical four-wave kinetic equation with Schr{\"o}dinger dispersion, Escobedo and Vel{\'a}zquez studied weak solutions in the isotropic setting \cite{EscobedoVelazquez2015}, and Germain, Ionescu, and Tran established an optimal local well-posedness theory in weighted $L^2$ and $L^\infty$ spaces \cite{GermainIonescuTran2020}. In one dimension, Germain, La, and Zhang developed a local theory for the kinetic MMT model with general dispersion relation, using nonlinear smoothing effects of the collision operator \cite{GermainLaZhang2025}. More recently, Ampatzoglou and L{\'e}ger identified a sharp well/ill-posedness threshold for quasilinear kinetic wave equations \cite{AmpatzoglouLeger2025}, and also gave a unified treatment of the optimal local theory in weighted $L^r$ spaces, $2\le r\le\infty$ \cite{AmpatzoglouLegerJDE}. For kinetic equations arising from one-dimensional FPUT chains, Germain, La, and Menegaki proved local well-posedness and stability results in weighted $L^\infty$ spaces \cite{GermainLaMenegaki2024FPUT}; see also their recent review \cite{GermainLaMenegaki2026OscillatorChains}. We also mention the work of Pan and Wu on the gravity water wave kinetic equation, which involves a more complicated physical collision kernel \cite{PanWu2026}. These results show that the behavior of a kinetic wave equation depends sensitively on the collision kernel, the dispersion relation, and the geometry of the resonant set.

The goal of the present paper is to study a model family of four-wave kinetic equations which makes this dependence explicit. On the one hand, the collision kernel contains a high-frequency weight measured by a parameter $\beta$, which represents the strength of the interaction and can be viewed as a derivative loss. Larger $\beta$ makes the collision operator more singular. On the other hand, we consider the power-law dispersion relation
$$\omega(p)=|p|^a.$$
The exponent $a$ changes the geometry of the resonance manifold and therefore changes the size of the resonant integral. Our main result gives a sharp local well-posedness and ill-posedness threshold in weighted $L^\infty$ spaces, showing quantitatively how the kernel strength and the dispersive geometry balance each other.

\subsection{The Model}
\ \par
Let $1\le a\le 5$ and $0\le \beta\le 1$. The purpose of this paper is to investigate the kinetic theory for systems of waves governed by the quasilinear equations
\begin{align}
    i\partial_t u+|\nabla|^a u=|\nabla|^\beta\left(|\nabla|^\beta u\, \big||\nabla|^\beta u\big|^2\right)
\end{align}
for $u:(-T,T)\times \mathbb{T}^3\rightarrow \mathbb{C}$, $T>0$, where the operator $|\nabla|^a$ is defined on its Fourier side.\par
Then it is standard to derive (formally) the corresponding homogeneous kinetic wave equation see \cite{Nazarenko2011}:
\begin{align}
    &\partial_t f(t,p)=\mathcal{C}[f](t,p)\qquad\mbox{ on }\mathbb{R}_+\times\mathbb{R}^3\label{1.2}\\
    &\mathcal{C}[f](p)\triangleq |p|^{2\beta} \iiint_{\mathbb{R}^9} \left(|p_1|^{2\beta}|p_2|^{2\beta}|p_3|^{2\beta}\right)f\,f_1 f_2 f_3\left(\frac{1}{f}+\frac{1}{f_1}-\frac{1}{f_2}-\frac{1}{f_3}\right)\notag\\
    &\qquad\qquad\qquad\qquad\qquad\qquad\qquad\qquad\times\delta(p+p_1-p_2-p_3)\delta\left(|p|^a+|p_1|^a-|p_2|^a-|p_3|^a\right)\,dp_1 dp_2 dp_3,\notag
\end{align}
where we use the convention $f=f(p)$ and $f_i=f(p_i)$ ($i=1,2,3$).\par
We consider the initial-value problems of (\ref{1.2}):
\begin{equation}
\begin{aligned}
    \partial_t f&=\mathcal{C}[f]\triangleq\mathcal{T}_1(f,f,f)+\mathcal{T}_2(f,f,f)-2\mathcal{T}_3(f,f,f),\\
    f(0)&=f_0,
\end{aligned}\label{1.3}
\end{equation}
where
\begin{equation}
\begin{aligned}
\mathcal{T}_1(f,g,h)&\triangleq\sup_p\left|p\right|^{2\beta}\iiint_{\mathbb{R}^9}\left|p_1\right|^{2\beta}\left|p_2\right|^{2\beta}\left|p_3\right|^{2\beta}\,f(p_1)\,g(p_2)\,h(p_3)\\
&\qquad\qquad\qquad\qquad\qquad\times\delta(p+p_1-p_2-p_3) \delta\left(\left|p\right|^a+\left|p_1\right|^a-\left|p_2\right|^a-\left|p_3\right|^a\right)\,dp_1 dp_2 dp_3\\
\mathcal{T}_2(f,g,h)&\triangleq\sup_p\left|p\right|^{2\beta}f(p)\iiint_{\mathbb{R}^9}\left|p_1\right|^{2\beta}\left|p_2\right|^{2\beta}\left|p_3\right|^{2\beta}\,g(p_2)\,h(p_3)\\
&\qquad\qquad\qquad\qquad\qquad\times\delta(p+p_1-p_2-p_3) \delta\left(\left|p\right|^a+\left|p_1\right|^a-\left|p_2\right|^a-\left|p_3\right|^a\right)\,dp_1 dp_2 dp_3\\
\mathcal{T}_3(f,g,h)&\triangleq\sup_p\left|p\right|^{2\beta}f(p)\iiint_{\mathbb{R}^9}\left|p_1\right|^{2\beta}\left|p_2\right|^{2\beta}\left|p_3\right|^{2\beta}\,g(p_1)\,h(p_2)\\
&\qquad\qquad\qquad\qquad\qquad\times\delta(p+p_1-p_2-p_3) \delta\left(\left|p\right|^a+\left|p_1\right|^a-\left|p_2\right|^a-\left|p_3\right|^a\right)\,dp_1 dp_2 dp_3\\
\end{aligned}.\label{1.4}    
\end{equation}
We define the function spaces $L^\infty_s$ ($s\ge 0$) by
$$L^\infty_s\triangleq\left\{f\in C(\mathbb{R}^3):\left\|\ang{x}^s f(x)\right\|_{L^\infty}<\infty\right\},$$
where by convention we denote $\ang{x}\triangleq\left(1+|x|^2\right)^{1/2}$. For convenience, we set $\mathcal{T}(f,g,h)\triangleq\mathcal{T}_1(f,g,h)+\mathcal{T}_2(f,g,h)-2\mathcal{T}_3(f,g,h)$.

Let us now comment on the physical meaning of the exponent $a$ in the dispersion relation. Although the dispersion laws arising in concrete physical models may be more complicated than a pure power, the model
$$
\omega(p)=|p|^a
$$
provides a simple way to distinguish different dispersive regimes. For example, $a=1$ corresponds to a linear, acoustic-type dispersion, where the resonance geometry is highly degenerate. The exponent $a=3/2$ corresponds to the power law of the deep-water pure capillary-wave dispersion. The case $a=2$ is the classical Schrödinger dispersion and is one of the standard examples in four-wave kinetic theory. Super-quadratic powers $a>2$ can also arise as high-frequency models for more complicated radial dispersion laws. For instance, in the very low temperature regime of the Bohm--Pines dispersion law \cite{BohmPines1951}, one may consider polynomial approximations of the form
$$
\omega(p)=\lambda_0+\lambda_1|p|^2+\lambda_2|p|^4.
$$
When $\lambda_2>0$, this dispersion behaves like $|p|^4$ at high frequencies, and is therefore modeled by the case $a=4$ at large frequencies.

\subsection{Main Results}
\ \par
We denote the following two key assumptions by {\rm (H1)} and {\rm (H2)}:\begin{enumerate}[label=(H\arabic*),itemsep=6pt]
    \item $\displaystyle{s>4\beta+3-\frac{a}{2}}$ and $\displaystyle{0\le\beta\le\frac{a-1}{4}}$;
    \item $\displaystyle{s=4\beta+3-\frac{a}{2}}$ and  $\displaystyle{0<\beta<\frac{a-1}{4}}$.
\end{enumerate}
Then, a simplified version of the main result is:
\begin{theorem}
    If either (H1) or (H2) holds, then (\ref{1.3}) is locally well-posed in $L^\infty_s$. Otherwise, if neither (H1) nor (H2) holds, then(\ref{1.3}) is ill-posed in $L^\infty_s$ in the sense that no strong solution can be constructed regardless of the size of the initial datum.
\end{theorem}

For the more precise statements, see Theorem 4.2 and Theorem 5.1 below.

\begin{remark}
We expect that the local well-posedness and ill-posedness mechanisms discussed here can be extended to the inhomogeneous equation in spaces such as $L_x^\infty L_{s,p}^\infty$, where the transport flow is an isometry, although we do not pursue this extension in the present paper.
\end{remark}

\begin{remark}
Let us briefly explain why we work in weighted $L^\infty$ spaces. From the physical point of view, in wave turbulence theory the unknown $f(p)$ represents the spectral density of wave action; equivalently, it describes how the wave intensity is distributed among different Fourier modes. Thus the norm $\left\|f\right\|_{L_s^\infty}$ provides pointwise control on the high-frequency tail of the wave spectrum, namely on the possible polynomial decay of $f(p)$ as $|p|\to\infty$. This is a natural class to consider, since several physically important spectra, including Rayleigh--Jeans equilibrium distributions and Kolmogorov--Zakharov cascade spectra, exhibit polynomial behavior in the momentum variable.

From the mathematical point of view, the weighted $L^\infty$ framework isolates the decay mechanism of the collision operator. Since the main difficulty in the kinetic wave equation comes from resonant interactions between different frequency regimes, the weighted $L^\infty$ norm allows us to identify the amount of high-frequency decay needed to control these interactions. Moreover, the space $L_s^\infty$ imposes no regularity in the momentum variable. Hence the threshold obtained here should be viewed as a pure decay threshold, rather than a condition mixing decay with smoothness.

Finally, the weighted $L^\infty$ framework allows us to test the collision operator in a pointwise sense. More precisely, it asks whether a distribution satisfying
$$|f(p)|\lesssim \langle p\rangle^{-s}$$
continues to have the same type of pointwise decay after one nonlinear collision interaction. This is different from working in weighted $L^p$ spaces, where the norm gives an averaged control in the momentum variable. In $L_s^\infty$, every output frequency must be controlled, so the worst resonant configurations cannot be averaged out. Thus the estimates
$$\mathcal{T}:(L_s^\infty)^3\to L_s^\infty$$
directly measure whether the collision dynamics is closed in a class of spectra with prescribed pointwise polynomial decay. Our well-posedness result proves this closure above the critical threshold, while the ill-posedness construction shows that it fails below the threshold.
\end{remark}

\begin{remark}
Let us also discuss the meaning of the critical threshold
$$
s_c=4\beta+3-\frac a2.
$$
This exponent reflects the balance between the strength of the interaction coefficient and the geometry of the resonant set. The term $4\beta$ reflects the high-frequency growth of the collision kernel: larger $\beta$ makes the nonlinear interaction more singular and requires stronger decay. The term $3$ comes from the three-dimensional resonant integration. The negative contribution $-a/2$ reflects the effect of the dispersion relation $\omega(p)=|p|^a$. Changing $a$ changes the geometry of the resonance manifold
$$
p+p_1=p_2+p_3,\qquad |p|^a+|p_1|^a=|p_2|^a+|p_3|^a.
$$
In the estimates for the collision operator, this geometry affects both the size of the resonant set and the Jacobian factor coming from the resonance function. The resulting geometric gain depends on $a$, and it appears in the final threshold as the term $-a/2$.

Therefore, $s_c$ is understood as the sharp polynomial decay threshold in weighted $L^\infty$ spaces and it records the sharp balance between the high-frequency strength of the kernel, the dimension of the resonant integration, and the dispersive geometry of the resonance manifold.
\end{remark}

\begin{remark}

Let us compare our result with several related works. In the Schrödinger case $\omega(p)=|p|^2$ without additional kernel growth, namely $a=2$ and $\beta=0$, our threshold gives
$$s_c=2.$$
This agrees with the sharp weighted $L^\infty$ threshold obtained by Germain, Ionescu, and Tran \cite{GermainIonescuTran2020}. Thus our result extends the weighted $L^\infty$ theory from the classical Schrödinger dispersion to the family of power-law dispersions $\omega(p)=|p|^a$, with an additional interaction weight measured by $\beta$.

Our result is also consistent with the recent sharp well/ill-posedness threshold for quasilinear kinetic wave equations obtained by Ampatzoglou and L{\'e}ger \cite{AmpatzoglouLeger2025}. In our theorem, besides the decay threshold$$
s_c=4\beta+3-\frac a2,$$
there is a natural restriction on the strength of the interaction weight,
$$\beta\leq \frac{a-1}{4}.$$
In the Schrödinger case $a=2$, this condition becomes exactly $\beta\leq 1/4$, which is the critical threshold found in the quasilinear model. Thus our result recovers the same critical value of $\beta$ when $a=2$, and extends it to the power-law dispersion $\omega(p)=|p|^a$.
\end{remark}

\subsubsection{First Reduction}
\ \par
We will use contraction principle to prove the local well-posedness result in later chapters. So, the key point here is to establish the boundedness of the operator $\mathcal{T}$ (See Proposition 3.9). Therefore, for later use, we write
\begin{align}
    \mathcal{J}_1&\triangleq\sup_p\iiint_{\mathbb{R}^9}\frac{\langle p\rangle^{s}\left|p\right|^{2\beta}\left|p_1\right|^{2\beta}\left|p_2\right|^{2\beta}\left|p_3\right|^{2\beta}}{\langle p_1\rangle^{s}\langle p_2\rangle^{s}\langle p_3\rangle^{s}}\delta(p+p_1-p_2-p_3) \delta\left(\left|p\right|^a+\left|p_1\right|^a-\left|p_2\right|^a-\left|p_3\right|^a\right)\,dp_1 dp_2 dp_3\\
    &=\sup_{\substack{p\\\left|p\right|\lesssim 1}}\iiint_{\left|p_1\right|\lesssim 1}\left(\ldots\right)+\sup_{\substack{p\\\left|p\right|\lesssim 1}}\iiint_{\left|p_1\right|\gg 1}\left(\ldots\right)+\sup_{\substack{p\\\left|p\right|\gg 1}}\iiint_{\left|p_1\right|\lesssim 1}\left(\ldots\right)+\sup_{\substack{p\\\left|p\right|\gg 1}}\iiint_{\substack{\left|p_1\right|\le\left|p\right|\\\left|p_1\right|\gg 1}}\left(\ldots\right)+\sup_{\substack{p\\\left|p\right|\gg 1}}\iiint_{\substack{\left|p_1\right|>\left|p\right|\\\left|p_1\right|\gg 1}}\left(\ldots\right)
    \notag\\
    &\triangleq\mathcal{J}_{1,\mbox{low},\mbox{low}}+\mathcal{J}_{1,\mbox{low},\mbox{high}}+\mathcal{J}_{1,\mbox{high},\mbox{low}}+\mathcal{J}_{1,\mbox{high},\mbox{high},1}+\mathcal{J}_{1,\mbox{high},\mbox{high},2},\notag
\end{align}
\begin{align}
    \mathcal{J}_2&\triangleq\sup_p\iiint_{\mathbb{R}^9}\frac{\left|p\right|^{2\beta}\left|p_1\right|^{2\beta}\left|p_2\right|^{2\beta}\left|p_3\right|^{2\beta}}{\langle p_2\rangle^{s}\langle p_3\rangle^{s}}\delta(p+p_1-p_2-p_3) \delta\left(\left|p\right|^a+\left|p_1\right|^a-\left|p_2\right|^a-\left|p_3\right|^a\right)\,dp_1 dp_2 dp_3\\
    &=\sup_{\substack{p\\\left|p\right|\lesssim 1}}\iiint_{\left|p_1\right|\lesssim 1}\left(\ldots\right)+\sup_{\substack{p\\\left|p\right|\lesssim 1}}\iiint_{\left|p_1\right|\gg 1}\left(\ldots\right)+\sup_{\substack{p\\\left|p\right|\gg 1}}\iiint_{\left|p_1\right|\lesssim 1}\left(\ldots\right)+\sup_{\substack{p\\\left|p\right|\gg 1}}\iiint_{\substack{\left|p_1\right|\le\left|p\right|\\\left|p_1\right|\gg 1}}\left(\ldots\right)+\sup_{\substack{p\\\left|p\right|\gg 1}}\iiint_{\substack{\left|p_1\right|>\left|p\right|\\\left|p_1\right|\gg 1}}\left(\ldots\right)
    \notag\\
    &\triangleq\mathcal{J}_{2,\mbox{low},\mbox{low}}+\mathcal{J}_{2,\mbox{low},\mbox{high}}+\mathcal{J}_{2,\mbox{high},\mbox{low}}+\mathcal{J}_{2,\mbox{high},\mbox{high},1}+\mathcal{J}_{2,\mbox{high},\mbox{high},2},\notag
\end{align}
and
\begin{align}
    \mathcal{J}_3&\triangleq\sup_p\iiint_{\mathbb{R}^9}\frac{\left|p\right|^{2\beta}\left|p_1\right|^{2\beta}\left|p_2\right|^{2\beta}\left|p_3\right|^{2\beta}}{\langle p_1\rangle^{s}\langle p_2\rangle^{s}}\delta(p+p_1-p_2-p_3) \delta\left(\left|p\right|^a+\left|p_1\right|^a-\left|p_2\right|^a-\left|p_3\right|^a\right)\,dp_1 dp_2 dp_3\\
    &=\sup_{\substack{p\\\left|p\right|\lesssim 1}}\iiint_{\left|p_1\right|\lesssim 1}\left(\ldots\right)+\sup_{\substack{p\\\left|p\right|\lesssim 1}}\iiint_{\left|p_1\right|\gg 1}\left(\ldots\right)+\sup_{\substack{p\\\left|p\right|\gg 1}}\iiint_{\left|p_1\right|\lesssim 1}\left(\ldots\right)+\sup_{\substack{p\\\left|p\right|\gg 1}}\iiint_{\substack{\left|p_1\right|\le\left|p\right|\\\left|p_1\right|\gg 1}}\left(\ldots\right)+\sup_{\substack{p\\\left|p\right|\gg 1}}\iiint_{\substack{\left|p_1\right|>\left|p\right|\\\left|p_1\right|\gg 1}}\left(\ldots\right)
    \notag\\
    &\triangleq\mathcal{J}_{3,\mbox{low},\mbox{low}}+\mathcal{J}_{3,\mbox{low},\mbox{high}}+\mathcal{J}_{3,\mbox{high},\mbox{low}}+\mathcal{J}_{3,\mbox{high},\mbox{high},1}+\mathcal{J}_{3,\mbox{high},\mbox{high},2}.\notag
\end{align}
For $l=1,2,3$, We also set $\mathcal{J}_{l,\mbox{high},\mbox{high}}\triangleq\mathcal{J}_{l,\mbox{high},\mbox{high},1}+\mathcal{J}_{l,\mbox{high},\mbox{high},2}$. For later purpose, for $l=1,2,3$, we respectively define
$$\mathcal{J}_{l,\mbox{low},\mbox{high},\mbox{bad}}\triangleq \sup_{\substack{p\\\left|p\right|\lesssim 1}}\left[\iiint_{\substack{\left|p_1\right|\gg 1\\\left|p\right|\ll\left|p_1\right|\approx\left|p_3\right|\\\left|p_2\right|\ll\left|p_1\right|}}\left(\ldots\right)+\iiint_{\substack{\left|p_1\right|\gg 1\\\left|p\right|\ll\left|p_1\right|\approx\left|p_2\right|\\\left|p_3\right|\ll\left|p_1\right|}}\left(\ldots\right)\right]\triangleq \mathcal{J}_{l,\mbox{low},\mbox{high},\mbox{bad},1}+\mathcal{J}_{l,\mbox{low},\mbox{high},\mbox{bad},2}$$
and
$$\mathcal{J}_{l,\mbox{low},\mbox{high},\mbox{good}}\triangleq\mathcal{J}_{l,\mbox{low},\mbox{high}}-\mathcal{J}_{l,\mbox{low},\mbox{high},\mbox{bad}}.$$
Thus, proving the boundedness of the operator $\mathcal T_l$ (or $\mathcal{T}$) $:(L^\infty_s)^3\rightarrow L^\infty_s$ essentially reduces to controlling the corresponding $\mathcal J_l$-integrals (or $\mathcal J_1+\mathcal J_2-2\mathcal J_3$).

\subsubsection{Outline of the Proof}
\ \par
We first give an outline of the proof of local well-posedness. The main task is to prove suitable bounds for the trilinear operators associated with the nonlinear collision term. Once these estimates are established, the local theory follows from a standard contraction mapping argument.

The trilinear estimates are reduced to integral estimates on the resonant set. Since the collision integral contains delta constraint, one has to carefully integrate over the corresponding resonance manifold. In the all-low regime, this is straightforward. In the high-frequency regimes, especially in the all-high case, a direct use of the coarea formula through the Jacobian of the resonance function is less convenient, unless one makes a rather fine decomposition. For this reason, following \cite{GermainIonescuTran2020},  in Chapter 3 we introduce a parametrization of the resonant manifold. This reduces the all-high contribution to an explicit integral, which can then be estimated by elementary methods. We emphasize that the parametrization is used mainly for convenience and to present a different approach; we do not claim that the coarea method alone cannot handle these regimes.

The collision operator contains both gain and loss terms. In many regimes, these terms can be estimated separately. However, in several bad regions, separate estimates are not sufficient, and one has to exploit the cancellation between the gain and loss contributions. This cancellation is delicate in the weighted $L^\infty$ setting. Since the input functions have no regularity, we cannot simply use the mean value theorem on the inputs. Instead, the proof extracts cancellation from the structure of the kernel and the geometry of the resonant set. This is one of the main difficulties discussed in Section 1.4. The main cancellation regimes are summarized in the following table.

\begin{table}[htbp]
\centering
\caption{Overview of the cancellation structure}
\label{tab:cancellation-structure}
\renewcommand{\arraystretch}{1.3}
\begin{tabularx}{\textwidth}{c c c X}
\toprule
Range of $a$ & Regime & Cancellation? & Comment \\
\midrule
$2\le a\le 5$
& $c_\star\gg 1$
& Yes
& The only cancellation regime in this range. \\

$2\le a\le 5$
& $c_\star\lesssim 1$
& No
& Direct estimates suffice. \\
\midrule
$1\le a<2$
& $c_\star\gg 1$
& Yes
& Same cancellation as in the range $2\le a\le 5$. \\

$1\le a<2$
& $|p_1|\gg |p|$
& Yes
& Extra cancellation is needed in high-low-low-high type regimes. \\

$1\le a<2$
& Other regimes
& No
& Direct estimates suffice. \\
\bottomrule
\end{tabularx}

\vspace{0.4em}
\begin{minipage}{\textwidth}
\small
Here $c_\star$ is a geometric parameter used to distinguish one delicate resonant regime (See (\ref{3.8}) for its precise definition); when $1\le a<2$, there are additional cancellation regimes not captured solely by $c_\star$.
\end{minipage}

\end{table}

\begin{remark}
In the regime $2\le a\le 5$, the main obstruction in the direct estimates for the large $c_\star$ case (i.e. $c_\star\gg 1$) appears only near the upper end of the range, namely when $a>2+2\sqrt2$. For simplicity and uniformity, however, we use the cancellation structure throughout the whole region $2\le a\le5$ and $c_\star\gg1$. This avoids introducing an additional subdivision in the proof.
\end{remark}

Next, we outline the proof of ill-posedness. The first step is to identify a bad region where the nonlinear operator has the largest possible growth. A scaling analysis suggests that the high-low-low-high interaction is one of the worst regions. In fact, we consider the dyadic patch
$$|p|\sim |p_3|\sim 2^k,\qquad |p_1|\sim |p_2|\sim 2^{k_1},\qquad 1\ll k_1\ll k,$$
with $p_3=p-p_1+p_2$. Let
$$\Phi(p_1,p_2)=|p|^a+|p_1|^a-|p_2|^a-|p-p_1+p_2|^a.$$
Choosing a coordinate direction in which $p_3$  has size $\sim 2^k$, we have
$$\left|\partial_{p_1^{(1)}}\Phi\right|=\left|a|p_1|^{a-2}p_1^{(1)}+a|p_3|^{a-2}p_3^{(1)}\right|\sim 2^{(a-1)k},$$
where $p_1^{(1)}$ means the first coordinate of $p_1$ and $p_3^{(1)}$ means the first coordinate of $p_3$. Therefore the delta constraint contributes the Jacobian factor $2^{-(a-1)k}$, while the remaining five low-frequency variables contribute the volume factor $2^{5k_1}$. Hence
$$\int_{\mathrm{HL LH}}\delta(\Phi)\,dp_1dp_2\sim 2^{5k_1}2^{-(a-1)k}.$$
Combining this with the kernel weight and the input weights formally gives
\begin{align*}
\langle p\rangle^s \mathcal{T}_1(f,f,f)(p)
&\sim
\frac{\langle p\rangle^s
|p|^{2\beta}|p_1|^{2\beta}|p_2|^{2\beta}|p_3|^{2\beta}}
{\langle p_1\rangle^s\langle p_2\rangle^s\langle p_3\rangle^s}
\cdot
2^{5k_1}2^{-(a-1)k} \\[8pt]
&\sim
\frac{2^{sk}\cdot 2^{2\beta k}\cdot 2^{2\beta k_1}\cdot 2^{2\beta k_1}\cdot 2^{2\beta k}}
{2^{sk_1}\cdot 2^{sk_1}\cdot 2^{sk}}
\cdot
2^{5k_1}2^{-(a-1)k} \\[8pt]
&=
2^{(4\beta+1-a)k}2^{(4\beta+5-2s)k_1}.
\end{align*}
Letting $k_1$ approach $k$ gives the growth $\displaystyle{\,2^{(8\beta+6-a-2s)k}}$, which is unbounded when $\displaystyle{s<4\beta+3-\frac a2}$. This explains why the high-low-low-high region is a bad region. More precisely, for the operator $\mathcal{T}_1$, this region gives a bad contribution throughout the whole parameter range considered in this paper. There are other bad regions as well, but they do not provide a uniform obstruction in all regimes.

We therefore construct initial data whose support is concentrated near the high-low-low-high region. The supports are chosen so that the main contribution comes from $\mathcal{T}_1$, while the terms $\mathcal{T}_2$ and $\mathcal{T}_3$ do not cancel or dominate this lower bound. The key estimate is then a lower bound showing that $\mathcal{T}_1$ fails to satisfy the required weighted $L^\infty$ mapping estimate on this family of data. This failure of the trilinear estimate gives the desired ill-posedness result. The endpoint case requires a more delicate construction; this is also one of the main challenges discussed in Section 1.4.

\subsection{Challenges}
\ \par
The first difficulty appears in the proof of local well-posedness. When $a<2$, the curvature of the dispersion relation degenerates at high frequencies. As a result, in some worst frequency configurations, the individual bounds for $\mathcal T_1,\mathcal T_2,\mathcal T_3$ are no longer sufficient, and one has to exploit the cancellation structure of the nonlinear term.

%The main obstruction is that we work in weighted $L^\infty$ spaces. Thus the input functions have no regularity, and the usual mean value theorem cannot be applied directly to the functions in order to gain a small factor from the cancellation.The collision integral contains delta functions enforcing the momentum and energy constraints. Instead, we exploit the cancellation through the geometry of the resonance surfaces. 
The main obstruction comes from the fact that the cancellation is not a pointwise cancellation of smooth weight factors. Although the weights, such as $\langle p\rangle^{-s}$, are smooth and can be estimated by a mean-value argument, this only treats part of the integrand and does not capture the actual structure of the collision operator. The collision integral contains delta functions enforcing the momentum and energy constraints, so the cancellation has to be analyzed at the level of the full constrained integral. In particular, these delta functions determine the resonance surfaces, the angular variables, and the associated Jacobian factors, all of which enter into the cancellation mechanism. After fixing the radial variables and passing to spherical coordinates, the relevant slices of the resonance manifolds become nearby latitude circles on the sphere. This geometric observation makes it possible to extract a cancellation gain.

However, this geometric cancellation is not strong enough in a pointwise sense for arbitrary bounded test functions. Indeed, a test function may concentrate in a thin neighborhood of only one of the two latitude circles, in which case the difference of the two integrals may still be large. To overcome this obstruction, we use an averaged spherical difference estimate: instead of comparing the two circles pointwise, we control the angularly averaged difference uniformly for bounded input functions; see Lemma 6.3 for a precise statement.

The second difficulty occurs in the ill-posedness argument. Although the high-low-low-high configuration suggests the expected bad interaction at a heuristic level, turning this observation into a rigorous ill-posedness proof is far from immediate. The first issue is that one has to isolate the gain contribution. We do this by choosing the initial data so that its mass is concentrated on carefully selected frequency patches, especially in different directions. With this choice, the dominant contribution comes from the gain term, while the loss terms are either absent or much smaller in the main testing region. This separation is essential for obtaining a clean lower bound.

This is relatively straightforward for a single testing patch. However, at the critical threshold, a single patch only gives a bounded contribution of size $O(1)$ and is not enough to prove ill-posedness. We therefore superpose many high-low-low-high patches in order to recover a logarithmic divergence. This creates a new difficulty. Once many patches are superposed together, different patches may interact with each other through the loss terms or through mixed frequency configurations. Such interactions could destroy the lower bound coming from the gain term. To avoid this, we choose the output scales of the patches to be well separated. This ensures that the desired gain interactions dominate, while the cross interactions between different patches are negligible.

Finally, a further technical point is that the geometry of the resonance function changes with the parameter $a$. The regimes $a=1$, $1<a<2$, and $2\leq a\leq5$ have different leading-order behaviors, and the corresponding angular localization estimates are not identical. For this reason, several geometric estimates 

\subsection{Notations}
\ \par
We write $A\lesssim B$ if there exists a (large) constant $C>0$ such that $A\leq CB$. We write $A\ll B$ if there exists a sufficiently small constant $c>0$ such that $A\leq cB$. The notations $A\gtrsim B$ and $A\gg B$ are defined similarly. The implicit constants may depend on the fixed parameters $s,a$, and $\beta$.
We use the standard Japanese bracket notation
$$\langle x\rangle \triangleq (1+|x|^2)^{1/2}.$$
Throughout the proof, we set
$$\rho \coloneqq p+p_1,\qquad r\coloneqq |p_1|.$$
In a few places, we will slightly abuse the notation $r$. Such abuses will always be explicitly indicated. For later use, we also define
$$
\widetilde r\coloneqq \frac{r}{|p|},\qquad \overline r\coloneqq \frac{|p|}{r}.
$$

\subsection{Plan of the Article}
\ \par
The rest of the paper is organized as follows.

In Chapter 2, we prove the boundedness of the low-frequency part of the collision operator. The all-low regime is treated directly by the coarea formula, while the one-high-one-low regime requires the cancellation between the gain and loss terms.

In Chapter 3, we turn to the high-frequency part of the collision operator, which is the most complicated part of the proof. We divide the analysis into cancellation and non-cancellation regimes. In the non-cancellation regimes, we parametrize the resonant manifold and further split the proof according to the range of the exponent $a$.

In Chapter 4, we complete the proof of local well-posedness. Once the boundedness of the trilinear operators is established, the local theory follows from a standard contraction mapping argument.

In Chapter 5, we prove the ill-posedness result by making rigorous the heuristic mechanism described above.

Finally, in Chapter 6, we collect several auxiliary lemmas used throughout the paper.

\section{LWP Proof Part I -- small-$\left|p\right|$ regime}
In this chapter, we will prove the boundedness of operator $\mathcal{T}_l$ ($l=1,2,3$) or $\mathcal{T}$ when $\left|p\right|$ and/or $\left|p_1\right|$ are $\lesssim 1$. From now until Chapter 4, we assume that either (H1) or (H2) holds.

\subsection{All-Low Case}
\ \par
We begin with the all-low regime. This case is relatively straightforward since all frequencies remain bounded, so no high-frequency growth has to be controlled. The estimate can be obtained directly, and no cancellation between the gain and loss terms is needed.
\begin{prop}
    $\mathcal{J}_{1,\mbox{low},\mbox{low}},\ \mathcal{J}_{2,\mbox{low},\mbox{low}},\ \mathcal{J}_{3,\mbox{low},\mbox{low}}\lesssim 1$.
\end{prop}
\begin{proof}
We will divide into two cases: (i) $1< a\le 5$; (ii) $a=1$ due to the concavity of the resonance function.\par
\noindent\underline{Case 1. $1<a\le 5$}\par
Since $a<6$, we get $\displaystyle{s\ge4\beta+3-\frac{a}{2}}$, which implies that $\langle p\rangle^{s-2\beta},\langle p_1\rangle^{s-2\beta},\langle p_2\rangle^{s-2\beta},\langle p_3\rangle^{s-2\beta}\lesssim 1$. Therefore, it suffices to show that
$$\displaystyle{\sup_{\left|p\right|\lesssim 1}\int_{\mathbb{R}^3}\left(\int_{\mathcal{S}_{p,p_1}}\frac{d\mu(z)}{\left|\nabla_z\Phi(z)\right|}\right)dp_1\lesssim 1,}$$
where $z=p_2$ and $\mathcal{S}_{p,p_1}\triangleq\left\{z:\Phi(z)\triangleq\left|p+p_1-z\right|^a+\left|z\right|^a-\left|p\right|^a-\left|p_1\right|^a=0\right\}$.\par
Now, we perform dyadic decomposition: let $\max\left(\left|p\right|,\left|p_1\right|,\left|p_2\right|,\left|p_3\right|\right)\sim 2^k$, where $k\le k_0$, $k_0>0$ is a fixed number. Thanks to $\left|p\right|^a+\left|p_1\right|^a=\left|p_2\right|^a+\left|p_3\right|^a$, we must have $\max\left(\left|p\right|,\left|p_1\right|\right)\sim \max\left(\left|p_2\right|,\left|p_3\right|\right)\sim 2^k$. Without loss of generality, we may furthermore assume that $\left|p_2\right|\le\left|p_3\right| \Rightarrow \left|p_3\right|\sim 2^k$. Next we perform dyadic decomposition to $\left|p_2\right|$: let $\left|p_2\right|\sim 2^{k_2}$ ($k_2<k$). We also denote
\begin{align*}
    &\sup_{\left|p\right|\lesssim 1}\int_{\mathbb{R}^3}\left(\int_{\mathcal{S}_{p,p_1}}\frac{d\mu(z)}{\left|\nabla_z\Phi(z)\right|}\right)dp_1\\
    =&\sup_{\left|p\right|\lesssim 1}\int_{\mathbb{R}^3}\left(\int_{\substack{\mathcal{S}_{p,p_1}\\\left|p_2\right|\ll\left|p_3\right|}}\frac{d\mu(z)}{\left|\nabla_z\Phi(z)\right|}\right)dp_1+\sup_{\left|p\right|\lesssim 1}\int_{\mathbb{R}^3}\left(\int_{\substack{\mathcal{S}_{p,p_1}\\\left|p_2\right|\sim\left|p_3\right|}}\frac{d\mu(z)}{\left|\nabla_z\Phi(z)\right|}\right)dp_1\\
    \triangleq& \sup_{\left|p\right|\lesssim 1}\int_{\mathbb{R}^3}I_A\,dp_1+\sup_{\left|p\right|\lesssim 1}\int_{\mathbb{R}^3}I_B\,dp_1.
\end{align*}
$1^\circ$:  $\left|p_2\right|\ll\left|p_3\right|$\par
Since $\nabla_z\Phi(z)=\nabla_{p_2}\Phi(p_2)=-a\left|p_2\right|^{a-2}p_2+a\left|p_3\right|^{a-2}p_3$ and $a>1$, we see that $\left|\nabla_z\Phi(z)\right|\sim \left(2^k\right)^{a-1}$. In addition, due to $\left|p_2\right|\sim 2^{k_2}$, the area of $\mathcal{S}_{p,p_1}$ is $\lesssim 2^{2k_2}$. Therefore, by coarea formula, we get
$$I_A=\sum_{\substack{\left|p_2\right|\sim 2^{k_2}\\\left|p_2\right|\ll\left|p_3\right|}}\frac{2^{2k_2}}{\left(2^k\right)^{a-1}}\lesssim \frac{2^{2k}}{\left(2^k\right)^{a-1}}=\frac{1}{\left(2^k\right)^{a-3}}.$$\\
$2^\circ$: $\left|p_2\right|\sim\left|p_3\right|\sim 2^k$\par
In this case, we further decompose $\left|p_2-p_3\right|\sim 2^t$ (say $t<k+3$ since $\left|p_2-p_3\right|\lesssim 2^k$). Then by mean value theorem, it's easy to see that
$$\bigl|\left|p_3\right|^{a-2}p_3-\left|p_2\right|^{a-2}p_2\bigr|\sim \left(2^k\right)^{a-2}\left|p_2-p_3\right|\sim  \left(2^k\right)^{a-2}\cdot 2^t,$$
where we use the observation $\displaystyle{p_2=\frac{p+p_1}{2}+\frac{p_2-p_3}{2}=\mbox{a fixed vector}+\frac{p_2-p_3}{2}}$ in the last step. Therefore, again by coarea formula, we get
$$I_B=\sum_{\substack{t\\t<k+3}}\frac{2^{2t}}{2^t\cdot\left(2^k\right)^{a-2}}\lesssim \frac{1}{\left(2^k\right)^{a-2}}\sum_{\substack{t\\t<k+3}}2^t\sim\frac{1}{\left(2^k\right)^{a-3}}.$$\par
Finally, since the volume of $p_1$ is $\lesssim 2^{3k}$, we conclude 
$$\displaystyle{\sup_{\left|p\right|\lesssim 1}\int_{\mathbb{R}^3}\left(\int_{\mathcal{S}_{p,p_1}}\frac{d\mu(z)}{\left|\nabla_z\Phi(z)\right|}\right)dp_1\lesssim \sum_{\substack{k\\k\le k_0}}\frac{2^{3k}}{\left(2^k\right)^{a-3}}=\sum_{\substack{k\\k\le k_0}}\left(2^k\right)^{6-a}\lesssim 1.}$$
This finishes the proof of Case 1.\par
\noindent\underline{Case 2. $a=1$}\par
In this case, we must have $\beta=0$ and $\langle p\rangle^{s},\langle p_1\rangle^{s},\langle p_2\rangle^{s},\langle p_3\rangle^{s}\lesssim 1$. Denote $E\triangleq\left|p\right|+\left|p_1\right|$, $z\triangleq p_2$ and $u\triangleq\left|p_2\right|$ in this case. If necessary, we will represent $z=\left(z_1,z_2,z_3\right)$, where $z_1,z_2,z_3\in \mathbb{R}$. Then it suffices to show that
$$\int_{\left|p_1\right|\lesssim 1}\int \delta\left(\left|p\right|+\left|p_1\right|-\left|p_2\right|-\left|p_3\right|\right)\,dp_2\,dp_1\triangleq\int_{\left|p_1\right|\lesssim 1} I_{\mbox{in}}\left(E,\rho\right)\,dp_1\lesssim 1.$$
We first show that
\begin{align}
    I_{\mbox{in}}\left(E,\rho\right)\lesssim E^2. \label{2.1}
\end{align}\par
When $\rho=0$, we have $p_2=p+p_1-p_2$. Therefore, it's easy to see that
\begin{align*}
    I_{\mbox{in}}\left(E,0\right)=\int_{\mathbb{R}^3}\delta \left(E-2u\right)\,dp_2=4\pi\int^{+\infty}_0 u^2\delta\left(E-2u\right)\,du=4\pi\frac{\left(\frac{E}{2}\right)^2}{2}\lesssim E^2.
\end{align*}\par
When $\rho\neq 0$, to simplify our computation, we use the following spheroidal coordinates:
\begin{align}
    \mu\triangleq\frac{\left|p_2\right|+\left|\rho-p_2\right|}{\left|\rho\right|}\ \ \ \ \ \ \ \ \ \ \nu\triangleq\frac{\left|p_2\right|-\left|\rho-p_2\right|}{\left|\rho\right|}.\label{2.2}
\end{align}
Without loss of generality, we also assume that the vector $\rho$ has the same direction as $x-$axis. Then, we have the following observations: (i) $\mu\ge 1$ and $-1\le \nu\le 1$, (ii) $\displaystyle{z_1=\frac{\left|\rho\right|}{2}\left(\mu\nu+1\right)}$ and (iii) $\displaystyle{z_2^2+z_3^2=\frac{\left|\rho\right|^2}{4}\left(\mu^2-1\right)\left(1-\nu^2\right)}$. (i) is trivial. For (ii), this is because from (\ref{2.2}) we have
$$\mu\cdot\nu=\frac{\left|z\right|^2-\left|\rho-z\right|^2}{\left|\rho\right|^2}=\frac{z_1^2-\left(\left|\rho\right|-z_1\right)^2}{\left|\rho\right|^2}=\frac{\left|\rho\right|^2+2\left|\rho\right|\cdot z_1}{\left|\rho\right|^2}.$$
For (iii), this is because using (\ref{2.2}) and above result we see
$$z_2^2+z_3^2=\left|z\right|^2-z_1^2=\left|z\right|^2-\frac{\left|\rho\right|^2}{4}\left(\mu\nu+1\right)^2=\frac{\left|\rho\right|^2}{4}\left[\left(\mu+\nu\right)^2-\left(\mu\nu+1\right)^2\right]=\frac{\left|\rho\right|^2}{4}\left(\mu^2-1\right)\left(1-\nu^2\right).$$
Therefore, we can perform the change of variables as:
\begin{align*}
    \begin{cases}
    z_1&=\displaystyle{\frac{\left|\rho\right|}{2}\left(\mu\nu+1\right)}\\[8pt]
    z_2&=\displaystyle{\frac{\left|\rho\right|}{2}\sqrt{\left(\mu^2-1\right)\left(1-\nu^2\right)}\cos\theta}\\[8pt]
    z_3&=\displaystyle{\frac{\left|\rho\right|}{2}\sqrt{\left(\mu^2-1\right)\left(1-\nu^2\right)}\sin\theta}
\end{cases}.
\end{align*}\par
Next, we show that
$$dz=dz_1dz_2dz_3=\frac{\left|\rho\right|^3}{8}\left(\mu^2-\nu^2\right)d\mu d\nu d\theta.$$
To prove this, we introduce an auxiliary variable for convenience:
$$z_\perp\triangleq\sqrt{z_2^2+z_3^2}=\frac{\left|\rho\right|}{2}\sqrt{\left(\mu^2-1\right)\left(1-\nu^2\right)}.$$
Then, we may write
\begin{align*}
    \begin{cases}
    z_1&=\displaystyle{\frac{\left|\rho\right|}{2}\left(1+\mu\nu\right)}\\
    z_2&=\displaystyle{z_\perp\cos\theta}\\
    z_3&=\displaystyle{z_\perp\sin\theta}
\end{cases},
\end{align*}
which gives us $dz=dz_1\cdot dz_2\cdot dz_3=dz_1\cdot z_\perp\cdot dz_\perp\cdot d\theta$. Therefore, we only need to compute $\displaystyle{\left|\frac{\partial\left(z_1,z_\perp\right)}{\partial\left(\mu,\nu\right)}\right|}$, where
\begin{align*}
    \begin{cases}
        z_1&=\displaystyle{\frac{\left|\rho\right|}{2}\left(1+\mu\nu\right)}\\
        z_\perp&=\displaystyle{\frac{\left|\rho\right|}{2}\sqrt{\left(\mu^2-1\right)\left(1-\nu^2\right)}}
    \end{cases}.
\end{align*}
In fact, we can get
\begin{align*}
    \frac{\partial\left(z_1,z_\perp\right)}{\partial\left(\mu,\nu\right)}=\begin{bmatrix}
\displaystyle{\frac{\left|\rho\right|}{2}\nu} & \displaystyle{\frac{\left|\rho\right|}{2}\mu}\\[8pt]
\displaystyle{\frac{\left|\rho\right|}{2}\frac{\mu\left(1-\nu^2\right)}{\sqrt{\left(\mu^2-1\right)\left(1-\nu^2\right)}}} & \displaystyle{-\frac{\left|\rho\right|}{2}\frac{\nu\left(\mu^2-1\right)}{\sqrt{\left(\mu^2-1\right)\left(1-\nu^2\right)}}}
\end{bmatrix},
\end{align*}
which gives $\displaystyle{\left|\frac{\partial\left(z_1,z_\perp\right)}{\partial\left(\mu,\nu\right)}\right|=\frac{\left|\rho\right|^2}{4}\frac{\left(\mu^2-\nu^2\right)}{\sqrt{\left(\mu^2-1\right)\left(1-\nu^2\right)}}}$. To sum up, we obtain
$$dz_1dz_2dz_3=\frac{\left|\rho\right|}{2}\sqrt{\left(\mu^2-1\right)\left(1-\nu^2\right)}\,\,\frac{\left|\rho\right|^2}{4}\frac{\left(\mu^2-\nu^2\right)}{\sqrt{\left(\mu^2-1\right)\left(1-\nu^2\right)}} d\mu d\nu d\theta=\frac{\left|\rho\right|^3}{8}\left(\mu^2-\nu^2\right)d\mu d\nu d\theta.$$\par
Now, we are ready to compute $I_{\mbox{in}}\left(E,\rho\right)$. In fact, we get
\begin{align*}
    I_{\mbox{in}}\left(E,\rho\right)&=\int_{\mathbb{R}^3}\delta\left(E-\left|\rho\right|\mu\right)\,dp_2=\frac{1}{\left|\rho\right|}\int_{\mathbb{R}^3}\delta\left(\mu-\frac{E}{\left|\rho\right|}\right)\,dp_2\\[5pt]
    &=\frac{\left|\rho\right|^2}{8}\int^1_{-1}\int^{2\pi}_0\int^{+\infty}_1\delta\left(\mu-\frac{E}{\left|\rho\right|}\right)\left(\mu^2-\nu^2\right)d\mu d\nu d\theta\\[5pt]
    &\sim \left|\rho\right|^2\int^1_{-1}\int^{2\pi}_0\left(\frac{E^2}{\left|\rho\right|^2}-\nu^2\right) d\nu d\theta \lesssim \left|\rho\right|^2\cdot\frac{E^2}{\left|\rho\right|^2}=E^2.
\end{align*}\par
Finally, we perform the dyadic decomposition as before. Let  $\max\left(\left|p\right|,\left|p_1\right|,\left|p_2\right|,\left|p_3\right|\right)\sim 2^k$, where $k\le k_0$, $k_0>0$ is a fixed number. Once again we must have $\max\left(\left|p\right|,\left|p_1\right|\right)\sim \max\left(\left|p_2\right|,\left|p_3\right|\right)\sim 2^k$. Now, it's easy to compute
\begin{align*}
    \int_{\left|p_1\right|\lesssim 1} I_{\mbox{in}}\left(E,\rho\right)\,dp_1 \lesssim \int_{\left|p_1\right|\lesssim 1} E^2\,dp_1\lesssim\sum_{\substack{k\\k\le k_0}}2^{2k}\int_{\left|p_1\right|\lesssim 1}dp_1\lesssim \sum_{\substack{k\\k\le k_0}}2^{2k}\cdot 2^{3k}\lesssim 1.
\end{align*}
This finishes the proof of Case 2.
\end{proof}

\subsection{Low-High Case}
\ \par
Now, we turn to consider the low-high case. We first discuss the endpoint case $a=1$, and then turn to the case $1<a<2$. In both cases, the cancellation mechanism is needed and yields the same estimates. The proofs are similar in spirit; however, due to the different convexity properties of the resonance functions, the two cases must be treated separately.

\begin{prop}
    When $a=1$, we have 
    \begin{enumerate}[label=(\roman*),itemsep=3pt]
        \item $\mathcal{J}_{1,\mbox{low},\mbox{high}},\ \mathcal{J}_{2,\mbox{low},\mbox{high},\mbox{good}},\ \mathcal{J}_{3,\mbox{low},\mbox{high},\mbox{good}}\lesssim 1$;
        \item $\mathcal{J}_{1,\mbox{low},\mbox{high},\mbox{bad}}+\mathcal{J}_{2,\mbox{low},\mbox{high},\mbox{bad}}-2\mathcal{J}_{3,\mbox{low},\mbox{high},\mbox{bad}}\lesssim 1.$
    \end{enumerate}
   In particular, this implies $\mathcal{J}_{1,\mbox{low},\mbox{high}}+\mathcal{J}_{2,\mbox{low},\mbox{high}}-2\mathcal{J}_{3,\mbox{low},\mbox{high}}\lesssim 1$.
\end{prop}
\begin{proof}
In this low-high regime, since $\left|p\right|\lesssim 1$ and $1\ll r$, we must have $\left|\rho\right|\approx r\gg 1$. Moreover, we must have $\beta=0$ and $s>2.5$. We divide the proof into four steps.\par
\noindent\underline{{Step 1. We first prove $\mathcal{J}_{1,\mbox{low},\mbox{high}}\lesssim 1$.}}\par Since $\langle p \rangle\sim 1$, we have
\begin{align}
\mathcal{J}_{1,\mbox{low},\mbox{high}}&\lesssim\sup_{\substack{p\\\left|p\right|\lesssim 1}}\int_{\left|p_1\right|\gg 1}\frac{1}{\langle p_1\rangle^s}\int_{\mathbb{R}^3}\delta\left(\left|p\right|+\left|p_1\right|-\left|p_2\right|-\left|p+p_1-p_2\right|\right)\frac{dp_2}{\langle p_2\rangle^s \langle p+p_1-p_2\rangle^s}\,dp_1\label{2.3}\\
&\triangleq \sup_{\substack{p\\\left|p\right|\lesssim 1}}\int_{\left|p_1\right|\gg 1}\frac{1}{\langle p_1\rangle^s}I_{\mbox{in}}\,dp_1.\notag
\end{align}
Denote $u\triangleq\left|p_2\right|, v\triangleq\left|p_3\right|, E\triangleq\left|p\right|+\left|p_1\right|$ and let $\theta$ be the angle between $p_2$ and $\rho$. Writing in spherical coordinate and $\theta$ variable, we see that
\begin{align*}
    I_{\mbox{in}}=\int^{+\infty}_0\int_{\mathbb{S}^2}\delta\left(E-u-v\right)\frac{u^2}{\langle u\rangle^s \langle v\rangle^s}d\omega du\sim \int^{+\infty}_0\int^1_{-1}\delta\left(E-u-v(\mu)\right)\frac{u^2}{\langle u\rangle^s \langle v(\mu)\rangle^s}d\mu du,
\end{align*}
where we used $d\omega=\sin\theta\,d\theta d\phi=-d(\cos\theta) d\phi\triangleq-d\mu d\phi$ and $\int d\phi=2\pi$. Denote $\Phi(\mu)\triangleq E-u-v(\mu)$, then 
\begin{align}
\Phi^\prime(\mu)=-v^\prime(\mu)=-\frac{-2u\left|\rho\right|}{2\sqrt{u^2+\left|\rho\right|^2-2u\left|\rho\right|\mu}}=\frac{u\left|\rho\right|}{v},\label{2.4}
\end{align}
which implies 
$$\int^1_{-1}\delta\left(E-u-v(\mu)\right)\frac{u^2}{\langle u\rangle^s \langle v(\mu)\rangle^s}d\mu=\frac{u^2}{\langle u\rangle^s \langle v\rangle^s}\frac{v}{\left|\rho\right|u}=\frac{uv}{\left|\rho\right|\langle u\rangle^s \langle v\rangle^s}.$$
Moreover, since $\left|u-v\right|\le\left|\rho\right|$ and $u+v=E$, we must have $\displaystyle{\frac{E-\left|\rho\right|}{2}\le u\le\frac{E-\left|\rho\right|}{2}}$. Therefore, we get
\begin{align*}
    I_{\mbox{in}}&=\frac{1}{\left|\rho\right|}\int^{+\infty}_0 \frac{u^2}{\langle u\rangle^s \langle v\rangle^s}du\sim \frac{1}{\left|\rho\right|}\int^{\frac{E+\left|\rho\right|}{2}}_\frac{E-\left|\rho\right|}{2}  \frac{u(E-u)}{\langle u\rangle^s \langle E-u\rangle^s}du\sim \frac{1}{\left|\rho\right|}\int^{\frac{E}{2}}_\frac{E-\left|\rho\right|}{2}  \frac{u(E-u)}{\langle u\rangle^s \langle E-u\rangle^s}du\\
    &\sim \frac{1}{\left|\rho\right|}\int^{\frac{E}{2}}_\frac{E-\left|\rho\right|}{2}  \frac{u}{\langle u\rangle^s }du\lesssim \frac{1}{\left|\rho\right|}\int^{\frac{E}{2}}_0  \frac{u}{\langle u\rangle^s }du\sim \frac{1}{\left|\rho\right|^s},
\end{align*}
where in the second step we used the symmetry of the integral; in the third step we used $E-u\sim E\sim \left|\rho\right|$; in the last step we used $s>2.5$. Finally, by (\ref{2.3}) and noticing that $\left|\rho\right|\approx r$, we conclude that
$$\mathcal{J}_{1,\mbox{low},\mbox{high}}\lesssim\int_{r\gg 1}\frac{1}{r^s}\cdot\frac{1}{\left|\rho\right|^s}\cdot r^2\,dr=\int_{r\gg 1}\frac{dr}{r^{2s-2}}\lesssim 1.$$\par
\noindent\underline{{Step 2. We next prove $\mathcal{J}_{1,\mbox{low},\mbox{high},\mbox{bad},1}+\mathcal{J}_{2,\mbox{low},\mbox{high},\mbox{bad},1}-2\mathcal{J}_{3,\mbox{low},\mbox{high},\mbox{bad},1}\lesssim 1$.}}\par 
Recall that
\begin{align}
    &\mathcal{J}_{1,\mbox{low},\mbox{high},\mbox{bad},1}+\mathcal{J}_{2,\mbox{low},\mbox{high},\mbox{bad},1}-2\mathcal{J}_{3,\mbox{low},\mbox{high},\mbox{bad},1}\notag\\
    \lesssim &\sup_{\substack{p\\\left|p\right|\lesssim 1}}\iint_{\left|p_1\right|\gg 1}\left(\frac{\langle p\rangle^s}{\langle p_1\rangle^s\langle p_2\rangle^s\langle p_3\rangle^s}+\frac{1}{\langle p_2\rangle^s\langle p_3\rangle^s}-\frac{1}{\langle p_1\rangle^s\langle p_2\rangle^s}-\frac{1}{\langle p_1\rangle^s\langle p_3\rangle^s}\right)\notag\\
    &\qquad\qquad\qquad\qquad\qquad\qquad\qquad\qquad\times\delta\left(\left|p\right|+\left|p_1\right|-\left|p_2\right|-\left|p+p_1-p_2\right|\right)\,dp_1 dp_2\notag\\
    \lesssim &\sup_{\substack{p\\\left|p\right|\lesssim 1}}\iint_{\left|p_1\right|\gg 1}\left(\frac{\langle p\rangle^s}{\langle p_1\rangle^s\langle p_2\rangle^s\langle p_3\rangle^s}-\frac{1}{\langle p_1\rangle^s\langle p_3\rangle^s}\right)+\left(\frac{1}{\langle p_2\rangle^s\langle p_3\rangle^s}-\frac{1}{\langle p_1\rangle^s\langle p_2\rangle^s}\right)\label{badarea}\\
    &\qquad\qquad\qquad\qquad\qquad\qquad\qquad\qquad\times\delta\left(\left|p\right|+\left|p_1\right|-\left|p_2\right|-\left|p+p_1-p_2\right|\right)\,dp_1 dp_2\notag\\
    \triangleq& I_A+I_B,\notag
\end{align}
where in the first inequality we used change of variables in $\mathcal{J}_{3,\mbox{low},\mbox{high},\mbox{bad},1}$.\par
We first estimate $I_B$, where
\begin{align*}
I_B&=\sup_{\substack{p\\\left|p\right|\lesssim 1}}\int\frac{1}{\langle p_2\rangle^s}\int\left(\frac{1}{\langle p_3\rangle^s}-\frac{1}{\langle p_1\rangle^s}\right)\delta\left(\left|p\right|+\left|p_1\right|-\left|p_2\right|-\left|p+p_1-p_2\right|\right)\,dp_1\,dp_2\\
&\triangleq \sup_{\substack{p\\\left|p\right|\lesssim 1}}\int\frac{1}{\langle p_2\rangle^s} I_{\mbox{in},B}\,dp_2.
\end{align*}
It turns out that this $I_B$ is the worst term, and we need to exploit a delicate cancellation argument to control it. Set $z\triangleq p_2$, $p_3=p+p_1-z$ and $h\triangleq z-p=p_1-p_3$. Then we perform the dyadic decompositions: $|p_1|\sim |p_3|\sim 2^{k_1}$, $k_1\gg1$ and $|h|=|p_1-p_3|\sim 2^l$, $l\le k_1-C$, \vspace{0.3em} where $C>0$ is a fixed large constant. Since $|p|\leq 1$, we have $\ang{z}\sim \ang{h}\sim \ang{2^l}$. We write
\begin{align*}
    I_B=\sum_{\substack{k_1\\k_1\gg 1}}\sum_{\substack{l\\l\leq k_1-C}} I_{B,k_1,l}\qquad\mbox{and}\qquad I_{\mbox{in},B}=\sum_{\substack{k_1\\k_1\gg 1}}\sum_{\substack{l\\l\leq k_1-C}} I_{\mbox{in},B,k_1,l}.
\end{align*}
Now we focus on the inner integral $I_{\mbox{in},B,k_1,l}$. For the first term we set $Y=p_3$, so that $p_1=Y+h$. For the second term we set $Y=p_1$, so that $p_3=Y-h$. Then, denote
$$\begin{cases}
\Phi_+(Y,z)=\left|p\right|+\left|Y+h\right|-\left|z\right|-\left|Y\right|;\\
\Phi_-(Y,z)=\left|p\right|+\left|Y\right|-\left|z\right|-\left|Y-h\right|
\end{cases}$$
and we can write
$$I_{\mbox{in},B}=\int_{\left|p_1\right|\gg 1}\frac{1}{\ang{Y}^s}\left[\delta\left(\Phi_+(Y,z)-\Phi_-(Y,z)\right)\right]\,dY.$$
We furthermore denote $Y=r\sigma$, $h=m\omega$ and $A\triangleq\left|z\right|-\left|p\right|=\left|p+h\right|-\left|p\right|$, where $\sigma,\omega\in\mathbb{S}^2$. Then $r\sim 2^{k_1}$ and $m\sim 2^l$. (Note that $r$ does not necessarily represent $|p_1|$ here, so we somehow mix the notation here.)\par
Next, if $\Phi_+=0$, then we have $\left|Y+h\right|=\left|Y\right|+\left|z\right|-\left|p\right|$, which is equivalent to $\left|r\sigma+h\right|=r+A$. Taking square on both sides and we see that 
\begin{align}
\sigma\cdot\omega=\mu_+(r,z)\triangleq\frac{2rA+A^2-m^2}{2rm}.\label{2.5}
\end{align}
Similarly, if $\Phi_-=0$, then we can get 
\begin{align}
\sigma\cdot\omega=\mu_-(r,z)\triangleq\frac{2rA+m^2-A^2}{2rm}.\label{2.6}
\end{align}
Thus, the two resonant surfaces two actually two latitudes on $\mathbb{S}^2$ with axis $\omega$.\par
Moreover, denote $\displaystyle{\mu_0\triangleq\frac{\mu_++\mu_-}{2}}$ and we observe that these two latitudes are very close in the sense that $\displaystyle{\frac{\left|\mu_+-\mu_-\right|}{1-\left|\mu_0\right|}\lesssim \frac{m}{r}}$. In fact, by (\ref{2.5}) and (\ref{2.6}), we have $\displaystyle{\mu_0=\frac{A}{m}}$ and $\displaystyle{\mu_+-\mu_-=\frac{A^2-m^2}{rm}}$. Since $A=\left|p+h\right|-\left|p\right|$ and $m=h$, we know that $\left|A\right|\le m$ and $\mu_+-\mu_-<0$. Therefore, if $A>0$, then we obtain
$$\frac{\left|\mu_+-\mu_-\right|}{1-\left|\mu_0\right| }=\frac{\frac{m^2-A^2}{rm}}{1-\frac{A}{m}}=\frac{m+A}{r}\sim\frac{m}{r}.$$
If $A<0$, then we obtain
$$\frac{\left|\mu_+-\mu_-\right|}{1-\left|\mu_0\right| }=\ldots=\frac{m-A}{r}\lesssim\frac{m}{r}.$$\par
In addition, denote $\mu=\sigma\cdot\omega$ and we recall that 
\begin{align*}
    \Phi_+=&\left|p\right|+\left|r\sigma+h\right|-\left|z\right|-\left|r\sigma\right|\\
    =&\left|p\right|+\sqrt{r^2+m^2+2rm\mu}-\left|z\right|-\left|r\sigma\right|.
\end{align*}
Then, we obtain 
$$\partial_\mu\Phi_+=\frac{1}{2}\frac{2rm}{\sqrt{r^2+m^2+2rm\mu}}=\frac{rm}{\left|r\sigma+h\right|}\sim\frac{rm}{r}=m.$$
An analogous argument also yields that $\partial_\mu\Phi_-\sim m$. These give us that
$$\delta(\Phi_\pm)\sim\frac{1}{m}\delta(\mu-\mu_\pm)=\frac{1}{m}\delta(\sigma\cdot\omega-\mu_\pm).$$\par
Now, by Lemma 6.3, we can compute
\begin{align*}
    I_{\mbox{in},B,k_1,l}&\sim \frac{1}{m}\int \frac{r^2}{\ang{r}^s}\int_{\mathbb{S}^2}\left[\delta\left(\sigma\cdot\omega-\mu_+\right)-\delta\left(\sigma\cdot\omega-\mu_-\right)\right]d\sigma dr\\
    &\lesssim\frac{1}{2^l}\cdot 2^{\frac{1}{2}\left(l-k_1\right)}\cdot\int_{r\sim  2^{k_1}} \frac{r^2}{\ang{r}^s}dr=\frac{1}{2^{l/2+k_1/2}}\cdot\frac{1}{\left(2^{k_1}\right)^{s-3}}=\frac{1}{2^{l/2}\cdot\left(2^{k_1}\right)^{s-5/2}}.
\end{align*}
Note that
$$\int \frac{dp_2}{\ang{p_2}^s}=\int \frac{dh}{\ang{h}^s}=\int_{m\sim 2^l}\frac{m^2 }{\ang{m}^s}dm=\begin{cases}
    \displaystyle{\frac{1}{\left(2^l\right)^{s-3}}}\qquad&\mbox{, if }l\ge 0\vspace{0.3em}\\
    2^{3l}\qquad\ \ \ \ \ \ \ &\mbox{, if }l<0
\end{cases}.$$
Therefore, we get
$$I_{B,k_1,l}\lesssim\frac{1}{2^{l/2}\cdot \left(2^{k_1}\right)^{s-5/2}}\sup_p\int \frac{dp_2}{\ang{p_2}^s}\lesssim\begin{cases}
    \displaystyle{\frac{1}{\left(2^l\right)^{s-5/2}\cdot\left(2^{k_1}\right)^{s-5/2}}}\qquad&\mbox{, if }l\ge 0\vspace{0.3em}\\
    \displaystyle{\frac{2^{\frac{5}{2}l}}{\left(2^{k_1}\right)^{s-5/2}}}\qquad&\mbox{, if }l< 0
\end{cases}.$$
Finally, summing over all dyadic pieces and recalling that $s>\frac{5}{2}$, we obtain
\begin{align}
    I_B&=\sum_{\substack{k_1\\k_1\gg 1}}\sum_{\substack{l\\l\leq k_1-C}} I_{B,k_1,l}=\sum_{k_1}\left(\sum_{\substack{l\\l\leq 0}}I_{B,k_1,l}+\sum_{\substack{l\\0<l\leq k_1-C}}I_{B,k_1,l}\right)\notag\\
    &=\sum_{k_1}\left(\sum_{\substack{l\\l\leq 0}}\frac{2^{\frac{5}{2}l}}{\left(2^{k_1}\right)^{s-5/2}}+\sum_{\substack{l\\0<l\leq k_1-C}}\frac{1}{\left(2^{l}\right)^{s-5/2}\left(2^{k_1}\right)^{s-5/2}}\right)=\sum_{\substack{k_1\\k_1\gg 1}}\frac{1}{\left(2^{k_1}\right)^{s-5/2}}\lesssim 1.\label{2.7}
\end{align}
If $k_1-C<0$, the same argument still applies. This finishes the proof of $I_B\lesssim 1$. \par
It remains to estimate $I_A$. This contribution is easier and can be controlled similarly as the case of $\mathcal{J}_{1,\mbox{low},\mbox{high}}$. First, we see that
\begin{align*}
I_A&\lesssim\sup_{\substack{p\\\left|p\right|\lesssim 1}}\int_{\left|p_1\right|\gg 1}\frac{1}{\langle p_1\rangle^s \ang{p_3}^s}\int\delta\left(\left|p\right|+\left|p_1\right|-\left|p_2\right|-\left|p+p_1-p_2\right|\right)\,dp_2\,dp_1\\
&\triangleq\sup_{\substack{p\\\left|p\right|\lesssim 1}}\int_{\left|p_1\right|\gg 1}\frac{1}{\langle p_1\rangle^s \ang{p_3}^s} \,I_{\mbox{in},A}\,dp_2\,dp_1,
\end{align*}
where instead of cancellation, we used
$$    \left|\frac{\ang{p}^{s}}{\ang{p_1}^{s}\ang{p_2}^{s}\ang{p_3}^{s}}-\frac{1}{\ang{p_1}^{s}\ang{p_3}^{s}}\right|\lesssim\frac{1}{\ang{p_1}^{s}\ang{p_3}^{s}}.$$
Denote again $E=\left|p\right|+\left|p_1\right|$, $u=\left|p_2\right|$, $v=\left|p_3\right|$ and we perform the dyadic decompositions: $u\sim 2^{k_2}$ and $\left|\rho\right|=\left|p+p_1\right|\approx\left|p_1\right|\sim 2^{k_1}$, $k_2\le k_1-C$, where $C>0$ is a fixed large constant. We also write
$$I_A=\sum_{\substack{k_1,k_2\\k_2\le k_1-C}}I_{A,k_1,k_2}\qquad\mbox{and}\qquad I_{\mbox{in},A}=\sum_{\substack{k_1,k_2\\k_2\le k_1-C}}I_{\mbox{in},A,k_1,k_2}$$
Furthermore, we denote $\theta$ be the angle of $p_2$ and $\rho$, and we write $p_2=u\sigma$, where $u>0$, $\sigma\in\mathbb{S}^2$. Then,
$$\Phi\triangleq\left|p\right|+\left|p_1\right|-\left|p_2\right|-\left|p+p_1-p_2\right|=E-u-\left|\rho-u\sigma\right|=E-u-\sqrt{\left|\rho\right|^2+u^2-2\left|\rho\right|u\cos\theta}.$$
Denote $\mu\triangleq\cos\theta$ and we then have $d\sigma=\sin\theta\,d\theta d\phi=-d(\cos\theta)d\phi=-d\mu d\phi$. By a same computation as (\ref{2.4}), we also have
$$\left|\partial_\mu\Phi\right|=\frac{\left|\rho\right|u}{v}\sim u$$
Therefore, we compute
\begin{align*}
    I_{\mbox{in},A,k_1,k_2}&=\int_{u\sim 2^{k_2}}\delta(\Phi)dp_2=\int_{u\sim 2^{k_2}}u^2\int_{\mathbb{S}^2}\delta(E-u-\left|\rho-u\sigma\right|)d\sigma du\\
    &\sim \int_{u\sim 2^{k_2}}u^2\int^1_{-1}\delta\left(E-u-\sqrt{\left|\rho\right|^2+u^2-2\left|\rho\right|u\mu}\right)d\mu du\sim \int_{u\sim 2^{k_2}}u^2\cdot\frac{1}{u}\,du\sim 2^{2k_2}.
\end{align*}
Finally, summing over all dyadic pieces and recalling that $s>\frac{5}{2}$ as before, we again obtain
\begin{align}
    I_A&\lesssim\sum_{\substack{k_1,k_2\\k_2\le k_1-C}}\frac{2^{2k_2}}{2^{2sk_1}}\int_{\left|p_1\right|\sim 2^{k_1}}dp_1=\sum_{\substack{k_1,k_2\\k_2\le k_1-C}}\frac{2^{2k_2}}{2^{2sk_1}}\cdot 2^{3k_1}\notag\\
    &=\sum_{\substack{k_1\\k_1\gg 1}}2^{\left(3-2s\right)k_1}\sum_{\substack{k_2\\k_2\le k_1-C}}2^{2k_2}\lesssim\sum_{\substack{k_1\\k_1\gg 1}}2^{\left(5-2s\right)k_1}\lesssim 1.\label{2.8}
\end{align}
Thus, combining (\ref{2.7}) and (\ref{2.8}), we finished the proof of
$$\mathcal{J}_{1,\mbox{low},\mbox{high},\mbox{bad},1}+\mathcal{J}_{2,\mbox{low},\mbox{high},\mbox{bad},1}-2\mathcal{J}_{3,\mbox{low},\mbox{high},\mbox{bad},1}\lesssim 1.$$\par
\noindent\underline{{Step 3. We then prove $\mathcal{J}_{1,\mbox{low},\mbox{high},\mbox{bad},2}+\mathcal{J}_{2,\mbox{low},\mbox{high},\mbox{bad},2}-2\mathcal{J}_{3,\mbox{low},\mbox{high},\mbox{bad},2}\lesssim 1$.}}\par
In this case, we slightly modify (\ref{badarea}) and write
\begin{align*}
    &\mathcal{J}_{1,\mbox{low},\mbox{high},\mbox{bad},2}+\mathcal{J}_{2,\mbox{low},\mbox{high},\mbox{bad},2}-2\mathcal{J}_{3,\mbox{low},\mbox{high},\mbox{bad},2}\\
    \lesssim &\sup_{\substack{p\\\left|p\right|\lesssim 1}}\iint_{\left|p_1\right|\gg 1}\left(\frac{\langle p\rangle^s}{\langle p_1\rangle^s\langle p_2\rangle^s\langle p_3\rangle^s}+\frac{1}{\langle p_2\rangle^s\langle p_3\rangle^s}-\frac{1}{\langle p_1\rangle^s\langle p_2\rangle^s}-\frac{1}{\langle p_1\rangle^s\langle p_3\rangle^s}\right)\\
    &\qquad\qquad\qquad\qquad\qquad\qquad\qquad\qquad\times\delta\left(\left|p\right|+\left|p_1\right|-\left|p_2\right|-\left|p+p_1-p_2\right|\right)\,dp_1 dp_2\\
    \lesssim &\sup_{\substack{p\\\left|p\right|\lesssim 1}}\iint_{\left|p_1\right|\gg 1}\left(\frac{\langle p\rangle^s}{\langle p_1\rangle^s\langle p_2\rangle^s\langle p_3\rangle^s}-\frac{1}{\langle p_1\rangle^s\langle p_2\rangle^s}\right)+\left(\frac{1}{\langle p_2\rangle^s\langle p_3\rangle^s}-\frac{1}{\langle p_1\rangle^s\langle p_3\rangle^s}\right)\\
    &\qquad\qquad\qquad\qquad\qquad\qquad\qquad\qquad\times\delta\left(\left|p\right|+\left|p_1\right|-\left|p_2\right|-\left|p+p_1-p_2\right|\right)\,dp_1 dp_2\\
    \triangleq& I_A+I_B.
\end{align*}
Then we only need to apply the exact same argument as in Step 2. Namely, we estimate directly for $I_A$, but use the cancellation for $I_B$. We omit the details here.\par
\noindent\underline{{Step 4. We finally prove $\mathcal{J}_{2,\mbox{low},\mbox{high},\mbox{good}},\ \mathcal{J}_{3,\mbox{low},\mbox{high},\mbox{good}}\lesssim 1$.}}\par
We first consider $\mathcal{J}_{3,\mbox{low},\mbox{high},\mbox{good}}$. In the case of $\left|p_3\right|\lesssim \left|p\right|$, we can use 
$$\frac{1}{\langle p_1\rangle^s\langle p_2\rangle^s}\lesssim\frac{\langle p\rangle^s}{\langle p_1\rangle^s\langle p_2\rangle^s\langle p_3\rangle^s}$$
to control $\mathcal{J}_{3,\mbox{low},\mbox{high},\mbox{good}}$ by the corresponding part of $\mathcal{J}_{1,\mbox{low},\mbox{high},\mbox{good}}$. Therefore, we can assume $\left|p\right|\ll \left|p_3\right|$. Now, according to our definition of $\mathcal{J}_{3,\mbox{low},\mbox{high},\mbox{good}}$, we only need to prove
\begin{align}
&\sup_{\substack{p\\\left|p\right|\lesssim 1}}\iiint_{\substack{\left|p_1\right|\gg 1\\\left|p\right|\ll\left|p_1\right|\\\left|p_1\right|\sim\left|p_2\right|\sim\left|p_3\right|}}\frac{1}{\langle p_1\rangle^s \langle p_2\rangle^s}\delta\left(p+p_1-p_2-p_3\right)\delta\left(\left|p\right|+\left|p_1\right|-\left|p_2\right|-\left|p_3\right|\right)\,dp_1 dp_2 dp_3\lesssim 1.\notag
\end{align}
However, this can be done by using the exact same argument of dealing with $I_A$ part in Step 2 above. Note that in this case all of $\left|p_i\right|$ ($i=1,2,3$) are large, so this case is even better than the $I_A$ part in Step 2.\par
Next, we consider $\mathcal{J}_{2,\mbox{low},\mbox{high},\mbox{good}}$. Similar as before, in the case of $\left|p_1\right|\lesssim \left|p\right|$, we can use 
$$\frac{1}{\langle p_2\rangle^s\langle p_3\rangle^s}\lesssim\frac{\langle p\rangle^s}{\langle p_1\rangle^s\langle p_2\rangle^s\langle p_3\rangle^s}$$
to control $\mathcal{J}_{2,\mbox{low},\mbox{high},\mbox{good}}$ by the corresponding part of $\mathcal{J}_{1,\mbox{low},\mbox{high},\mbox{good}}$. Moreover, in the case of $\left|p_1\right|\lesssim \left|p_3\right|$, we can use 
$$\frac{1}{\langle p_2\rangle^s\langle p_3\rangle^s}\lesssim\frac{1}{\langle p_2\rangle^s\langle p_3\rangle^s}\frac{\langle p_3\rangle^s}{\langle p_1\rangle^s}=\frac{1}{\langle p_1\rangle^s\langle p_2\rangle^s}$$
to control $\mathcal{J}_{2,\mbox{low},\mbox{high},\mbox{good}}$ by the corresponding part of $\mathcal{J}_{3,\mbox{low},\mbox{high},\mbox{good}}$. Therefore, we can assume $\left|p_1\right|\gg\left|p\right|$ and $\left|p_1\right|\gg\left|p_3\right|$. However, if so, then we must have $\left|p_1\right|\approx\left|p_2\right|$, which implies that this is a region classified in $\mathcal{J}_{2,\mbox{low},\mbox{high},\mbox{bad}}$.\par
To sum up, we already conclude $\mathcal{J}_{2,\mbox{low},\mbox{high},\mbox{good}},\ \mathcal{J}_{3,\mbox{low},\mbox{high},\mbox{good}}\lesssim 1$.    
\end{proof}

\vspace{1em}

\begin{prop}
    When $1<a<2$, we still have 
    \begin{enumerate}[label=(\roman*),itemsep=3pt]
        \item $\mathcal{J}_{1,\mbox{low},\mbox{high}},\ \mathcal{J}_{2,\mbox{low},\mbox{high},\mbox{good}},\ \mathcal{J}_{3,\mbox{low},\mbox{high},\mbox{good}}\lesssim 1$;
        \item $\mathcal{J}_{1,\mbox{low},\mbox{high},\mbox{bad}}+\mathcal{J}_{2,\mbox{low},\mbox{high},\mbox{bad}}-2\mathcal{J}_{3,\mbox{low},\mbox{high},\mbox{bad}}\lesssim 1.$
    \end{enumerate}
   In particular, this implies $\mathcal{J}_{1,\mbox{low},\mbox{high}}+\mathcal{J}_{2,\mbox{low},\mbox{high}}-2\mathcal{J}_{3,\mbox{low},\mbox{high}}\lesssim 1$.
\end{prop}
\begin{proof}
Recall the assumptions (H1) and (H2). Also recall that we are now in the low-high case, where $|p|\lesssim 1$ and $1\ll r$. In particular, $|\rho|\approx r\gg1$. The proof is similar to that of Proposition 2.2, and therefore we will closely follow its argument. Especially, we will use the same notation as before. Also, we divide the proof into four steps as before. The last two steps are identical to the previous argument, so it suffices to focus on the first two steps. We will keep the proof relatively brief, emphasizing only the differences from the previous argument.\par
\noindent\underline{{Step 1. We first prove $\mathcal{J}_{1,\mbox{low},\mbox{high}}\lesssim 1$.}}\par Since $\langle p \rangle\sim 1$, we now have
\begin{align}
&\mathcal{J}_{1,\mbox{low},\mbox{high}}\notag\\
\lesssim&\sup_{\substack{p\\\left|p\right|\lesssim 1}}\int_{\left|p_1\right|\gg 1}\frac{|p_1|^{2\beta}}{\langle p_1\rangle^s}\int_{\mathbb{R}^3}\delta\left(\left|p\right|^a+\left|p_1\right|^a-\left|p_2\right|^a-\left|p+p_1-p_2\right|^a\right)\frac{|p_2|^{2\beta}|p+p_1-p_2|^{2\beta}}{\langle p_2\rangle^s \langle p+p_1-p_2\rangle^s}\,dp_2\,dp_1\notag\\
\triangleq& \sup_{\substack{p\\\left|p\right|\lesssim 1}}\int_{\left|p_1\right|\gg 1}\frac{|p_1|^{2\beta}}{\langle p_1\rangle^s}I_{\mbox{in}}\,dp_1.\notag
\end{align}
Then we compute $I_{\mbox{in}}$ as
\begin{align*}
    I_{\mbox{in}}&=\int^{+\infty}_0\int_{\mathbb{S}^2}\delta\left(E-u^a-v^a\right)\frac{|u|^{2\beta}|v|^{2\beta}u^2}{\langle u\rangle^s \langle v\rangle^s}d\omega du\sim \int^{+\infty}_0\int^1_{-1}\delta\left(E-u^a-v^a(\mu)\right)\frac{|u|^{2\beta+2}|v|^{2\beta}}{\langle u\rangle^s \langle v(\mu)\rangle^s}d\mu du\\
    &\sim \int\frac{|u|^{2\beta+2}|v|^{2\beta}}{\langle u\rangle^s \langle v\rangle^s}\frac{v^{2-a}}{|\rho|u}du=\frac{1}{|\rho|}\int\frac{|u|^{2\beta+1}}{\langle u\rangle^s }\frac{v^{2\beta+2-a}}{\langle v\rangle^s}du,
\end{align*}
where we used
$$\partial_\mu\Phi=-\frac{a}{2}\left(|\rho|^2+u^2-2|\rho|u
\,\mu\right)^{a/2-1}\left(-2|\rho|u\right)\sim|\rho|u\,v^{a-2}.$$
Now without loss of generality we assume $u\le v$ since the roles of $p_2$ and $p_3$ in $\mathcal{J}_{1,\mbox{low},\mbox{high}}$ are symmetric. Therefore, in view of the resonance function $|p|^a+|p_1|^a=|p_2|^a+|p_3|^a$, we must have $|\rho|\sim|p_1|\sim|p_3|=v$ and $u\lesssim |\rho|$. Thus, we can further compute $I_{\mbox{in}}$ as
\begin{align*}
    I_{\mbox{in}}\sim \frac{1}{|\rho|^{s-2\beta+a-1}}\int_0^{C|\rho|}\frac{u^{2\beta+1}}{\ang{u}^s}du\lesssim \frac{1}{|\rho|^{s-2\beta+a-1}},
\end{align*}
where we used $s>2\beta+2$. Finally, noticing that $|\rho|\approx r$ we conclude that
$$\mathcal{J}_{1,\mbox{low},\mbox{high}}\lesssim\int_{r\gg 1}\frac{1}{r^{s-2\beta-2}}\cdot\frac{1}{r^{s-2\beta+a-1}}\,dr=\int_{r\gg 1}\frac{dr}{r^{2s-4\beta+a-3}}\lesssim 1,$$
where we used $\displaystyle{s\ge 4\beta+3-\frac{a}{2}=(2\beta-\frac{a}{2}+2)+(2\beta+1)>2\beta-\frac{a}{2}+2}$.\par
\noindent\underline{{Step 2. We next prove $\mathcal{J}_{1,\mbox{low},\mbox{high},\mbox{bad},1}+\mathcal{J}_{2,\mbox{low},\mbox{high},\mbox{bad},1}-2\mathcal{J}_{3,\mbox{low},\mbox{high},\mbox{bad},1}\lesssim 1$.}}\par
In this step, we use a decomposition similar to (\ref{badarea}). Note that the resonance function is different in the present setting, and that the factor $|p|^{2\beta}|p_1|^{2\beta}|p_2|^{2\beta}|p_3|^{2\beta}$ also appears at the beginning. We first estimate the cancellation term $I_B$, where
\begin{align*}
I_B&=\sup_{\substack{p\\\left|p\right|\lesssim 1}}\int\frac{|p_2|^{2\beta}}{\langle p_2\rangle^s}\int\left(\frac{|p_1|^{2\beta}|p_3|^{2\beta}}{\langle p_3\rangle^s}-\frac{|p_1|^{2\beta}|p_3|^{2\beta}}{\langle p_1\rangle^s}\right)\delta\left(\left|p\right|^a+\left|p_1\right|^a-\left|p_2\right|^a-\left|p+p_1-p_2\right|^a\right)\,dp_1\,dp_2\\
&\triangleq \sup_{\substack{p\\\left|p\right|\lesssim 1}}\int\frac{|p_2|^{2\beta}}{\langle p_2\rangle^s} I_{\mbox{in},B}\,dp_2.
\end{align*}
As before, We also decompose $I_B$ and $I_{\mbox{in},B}$ as
\begin{align*}
    I_B=\sum_{\substack{k_1\\k_1\gg 1}}\sum_{\substack{l\\l\leq k_1-C}} I_{B,k_1,l}\qquad\mbox{and}\qquad I_{\mbox{in},B}=\sum_{\substack{k_1\\k_1\gg 1}}\sum_{\substack{l\\l\leq k_1-C}} I_{\mbox{in},B,k_1,l}.
\end{align*}
Using the same notation as before, the corresponding resonance functions now become
$$\begin{cases}
\Phi_+(Y,z)=\left|p\right|^a+\left|Y+h\right|^a-\left|z\right|^a-\left|Y\right|^a;\\
\Phi_-(Y,z)=\left|p\right|^a+\left|Y\right|^a-\left|z\right|^a-\left|Y-h\right|^a
\end{cases}$$
and $A\triangleq\left|z\right|^a-\left|p\right|^a=\left|p+h\right|^a-\left|p\right|^a$.
If $\Phi_+=0$ and we then have $|r\sigma+h|^a=r^a+A$ which is equivalent to $|r\sigma+h|^2=\left(r^a+A\right)^{2/a}$. This yields 
\begin{align*}
\sigma\cdot\omega=\mu_+(r,z)\triangleq\frac{\left(r^a+A\right)^{2/a}-r^2-m^2}{2rm}.\
\end{align*}
Similarly, if $\Phi_-=0$, then we can get 
\begin{align*}
\sigma\cdot\omega=\mu_-(r,z)\triangleq\frac{r^2+m^2-\left(r^a-A\right)^{2/a}}{2rm}.
\end{align*}
Note that
\begin{align*}
    &\left(r^a+A\right)^{2/a}+\left(r^a-A\right)^{2/a}-2r^2=r^2\left[\left(1+\frac{A}{r^a}\right)^{2/a}+\left(1-\frac{A}{r^a}\right)^{2/a}\right]-2r^2\\
    =&r^2\left[2+O\left(\frac{A}{r^a}\right)^2\right]-2r^2=O\left(\frac{A^2}{r^{2a-2}}\right)
\end{align*}
and we get
$$\left|\mu_+-\mu_-\right|\lesssim\frac{r^{2-2a}A^2+m^2}{rm}.$$
Now we claim that 
\begin{align}
    \left|\mu_+-\mu_-\right|\lesssim\frac{m}{r}.\label{2.10}
\end{align}
Indeed, if $|h|\gg 1$, then $A=|p+h|^a-|p|^a\lesssim |h|^a=m^a$, which implies that
\begin{align*}
    \left|\mu_+-\mu_-\right|\lesssim\frac{r^{2-2a}A^2+m^2}{rm}=r^{1-2a}m^{2a-1}+\frac{m}{r}=\left(\frac{m}{r}\right)^{2a-1}+\frac{m}{r}\lesssim\frac{m}{r},
\end{align*}
where we used $\displaystyle{\frac{m}{r}\ll 1}$. On the other hand, if $|h|\lesssim 1$, then $A=|p+h|^a-|p|^a\lesssim |h|=m$, which gives that
\begin{align*}
    \left|\mu_+-\mu_-\right|\lesssim\frac{r^{2-2a}A^2+m^2}{rm}=\frac{m}{r^{2a-1}}+\frac{m}{r}\lesssim\frac{m}{r},
\end{align*}
where we used $r\gg 1\Rightarrow r^{2a-1}\gg r$. Hence (\ref{2.10}) follows. Moreover, we claim that
\begin{align}
    1-|\mu_0|\approx 1.\label{2.11}
\end{align}
In fact, it suffices to show that $|\mu_0|\ll 1$. We first observe that
$$\mu_0=\frac{\mu_++\mu_-}{2}=\frac{\left(r^a+A\right)^{2/a}-\left(r^a-A\right)^{2/a}}{4rm}.$$
Using mean value theorem to the function $x\mapsto x^{2/a}$ and we obtain that 
\begin{align*}
|\mu_0|\lesssim\frac{r^{2-a}|A|}{rm}=\frac{|A|}{mr^{a-1}}\lesssim\begin{cases}
    \displaystyle{\frac{m^a}{mr^{a-1}}=\left(\frac{m}{r}\right)^{a-1}}\vspace{0.5em}\ \ &\mbox{, if }|h|\gg 1\\
    \displaystyle{\frac{m}{mr^{a-1}}=\frac{1}{r^{a-1}}}\ \ &\mbox{, if }|h|\lesssim 1
\end{cases}=\frac{\max\left(1,m^{a-1}\right)}{r^{a-1}}\ll 1.
\end{align*}
Hence (\ref{2.11}) also follows. Combine (\ref{2.10}) and (\ref{2.11}), we again achieve that
\begin{align*}
    \frac{\left|\mu_+-\mu_-\right|}{1-|\mu_0|}\lesssim\frac{m}{r}.
\end{align*}
In addition, we recall that 
\begin{align*}
    \Phi_+=\left|p\right|^a+\left|r\sigma+h\right|^a-\left|z\right|^a-\left|r\sigma\right|^a=\left|p\right|^a+\left(r^2+m^2+2rm\mu\right)^{a/2}-\left|z\right|^a-\left|r\sigma\right|^a.
\end{align*}
and we have
$$\partial_\mu|r\sigma+h|^a=\frac{a}{2}\left(r^2+m^2+2rm\mu\right)^{a/2-1}\cdot 2rm\sim r^{a-2}\cdot rm=mr^{a-1}.$$
An analogous argument also yields the same result for $\partial_\mu\Phi_-$. These tell us that
$$\delta(\Phi_\pm)\sim\frac{1}{mr^{a-1}}\delta(\mu-\mu_\pm)=\frac{1}{mr^{a-1}}\delta(\sigma\cdot\omega-\mu_\pm).$$
Thus, by Lemma 6.3, we can estimate as
\begin{align*}
    I_{\mbox{in},B,k_1,l}&\sim \frac{1}{mr^{a-1}}\int \frac{r^2}{\ang{r}^{s-4\beta}}\int_{\mathbb{S}^2}\left[\delta\left(\sigma\cdot\omega-\mu_+\right)-\delta\left(\sigma\cdot\omega-\mu_-\right)\right]d\sigma dr\\
    &\lesssim\frac{1}{2^l\,2^{(a-1)k_1}}\cdot 2^{\frac{1}{2}\left(l-k_1\right)}\cdot\int_{r\sim  2^{k_1}} \frac{1}{r^{s-4\beta-2}}dr=\frac{2^{\frac{1}{2}\left(l-k_1\right)}}{2^l\,2^{(a-1)k_1}} \frac{1}{\left(2^{k_1}\right)^{s-4\beta-3}}=\frac{1}{2^{l/2}\cdot\left(2^{k_1}\right)^{s-4\beta+a-3.5}}.
\end{align*}
Consequently, we obtain 
\begin{align*}
     I_{B,k_1,l}\lesssim \frac{1}{2^{l/2}\cdot\left(2^{k_1}\right)^{s-4\beta+a-3.5}}\sup_p \int\frac{|p_2|^{2\beta}}{\langle p_2\rangle^s} \,dp_2\lesssim\begin{cases}
         \displaystyle{\frac{1}{\left(2^l\right)^{s-2\beta-2.5}}\frac{1}{\left(2^{k_1}\right)^{s-4\beta+a-3.5}}}\vspace{0.7em}\ \ &\mbox{, if }l\ge 0\\
         \displaystyle{\frac{\left(2^l\right)^{2\beta+2.5}}{\left(2^{k_1}\right)^{s-4\beta+a-3.5}}}\ \ &\mbox{, if }l< 0
     \end{cases},
\end{align*}
where we used the fact that
$$\int\frac{|p_2|^{2\beta}}{\langle p_2\rangle^s}dp_2\sim\begin{cases}
    \displaystyle{\int_{m\sim 2^l}\frac{dm}{m^{s-2\beta-2}}=\frac{1}{\left(2^l\right)^{s-2\beta-3}}}\vspace{0.7em}\qquad&\mbox{, if }l\ge \\
    \displaystyle{\int_{m\sim 2^l}m^{2+2\beta}\,dm=\left(2^l\right)^{3+2\beta}}\qquad&\mbox{, if }l< 0
\end{cases}.$$
Finally, summing over all dyadic pieces and we obtain
\begin{align*}
    I_B&=\sum_{k_1}\left(\sum_{\substack{l\\l\leq 0}}\frac{(2^l)^{2\beta+2.5}}{\left(2^{k_1}\right)^{s-4\beta+a-3.5}}+\sum_{\substack{l\\0<l\leq k_1-C}}\frac{1}{\left(2^{l}\right)^{s-2\beta-2.5}\left(2^{k_1}\right)^{s-4\beta+a-3.5}}\right)\\
    &\lesssim\sum_{\substack{k_1\\k_1\gg 1}}\frac{1}{\left(2^{k_1}\right)^{s-4\beta+a-3.5}}+\sum_{\substack{l\\0<l\leq k_1-C}}\frac{1}{\left(2^{k_1}\right)^{s-4\beta+a-3.5}}\times\begin{cases}
        1\ \ &\mbox{, if }s>2\beta+2.5\\
        \displaystyle{\frac{1}{(2^{k_1})^{s-2\beta-2.5}}}\ \ &\mbox{, if }s=2\beta+2.5\\
        k_1\ \ &\mbox{, if }s<2\beta+2.5
    \end{cases}\\
    &\lesssim\begin{cases}
        \displaystyle{\sum_{\substack{k_1\\k_1\gg 1}}\frac{1}{\left(2^{k_1}\right)^{s-4\beta+a-3.5}}}\ \ &\mbox{, if }s>2\beta+2.5\\
        \displaystyle{\sum_{\substack{k_1\\k_1\gg 1}}\frac{k_1}{\left(2^{k_1}\right)^{s-4\beta+a-3.5}}}\ \ &\mbox{, if }s=2\beta+2.5\\
        \displaystyle{\sum_{\substack{k_1\\k_1\gg 1}}\frac{1}{\left(2^{k_1}\right)^{2s-6\beta+a-6}}}\ \ &\mbox{, if }s<2\beta+2.5
    \end{cases}
\end{align*}
In the case of $s=2\beta+2.5$, we have
$$s-4\beta+a-3.5=a-2\beta-1\ge a+\frac{1-a}{2}-1=\frac{a}{2}-\frac{1}{2}>0.$$
Then we can pick $\eps=\eps(a)\ll 1$ such that $s-4\beta+a-3.5-\eps>0$. In the case of $s<2\beta+2.5$, if $\displaystyle{s>4\beta+3-\frac{a}{2}}$, then 
$$2s-6\beta+a-6\ge 2s-8\beta+a-6>0;$$
if $\displaystyle{s=4\beta+3-\frac{a}{2}}$, then $\beta>0$, which implies that 
$$2s-6\beta+a-6=(2s-8\beta+a-6)+2\beta=2\beta>0.$$
Hence, we always have $2s-6\beta+a-6>0$. Consequently, $I_B\lesssim 1$.\par
Next, it remains to estimate $I_A$. The argument is exactly the same except the resonance function. In the present setting, we have
$$\Phi=E-u^a-\left(|\rho|^2+u^2-2|\rho|u\cos\theta\right)^{a/2}$$
Therefore, 
$$\partial_\mu\Phi=-\frac{a}{2}\left(|\rho|^2+u^2-2|\rho|u\,\mu\right)^{a/2-1}\left(-2|\rho|u\right)$$
and this gives that
$$|\partial_\mu\Phi|\sim|\rho|^{a-2}\cdot|\rho|\cdot u=|\rho|^{a-1}\cdot u\sim \left(2^{k_1}\right)^{a-1}\cdot 2^{k_2}.$$
Therefore, we compute
\begin{align}
    I_{\mbox{in},A,k_1,k_2}&\sim \int_{u\sim 2^{k_2}}u^2\int^1_{-1}\delta\left(E-u^a-\left(\left|\rho\right|^2+u^2-2\left|\rho\right|u\mu\right)^{a/2}\right)d\mu du\notag\\
    &\sim \int_{u\sim 2^{k_2}}u^2\cdot\frac{1}{(2^{k_1})^{a-1}2^{k_2}}\,du\sim \frac{2^{2k_2}}{(2^{k_1})^{a-1}}.\label{2.12}
\end{align}
Finally, noting that 
\begin{align}
|p|^{2\beta}|p_1|^{2\beta}|p_2|^{2\beta}|p_3|^{2\beta}\sim 2^{4\beta k_1}\cdot 2^{2\beta k_2}\label{2.13}
\end{align}
and summing over all dyadic pieces as before, we again obtain
\begin{align}
    I_A&\lesssim\sum_{\substack{k_1,k_2\\k_2\le k_1-C}}\frac{2^{4\beta k_1}2^{2\beta k_2}}{2^{2sk_1}}\frac{2^{2k_2}}{\left(2^{k_1}\right)^{a-1}}\int_{\left|p_1\right|\sim 2^{k_1}}dp_1\notag\\
    &=\sum_{\substack{k_1\\k_1\gg 1}}\frac{1}{(2^{k_1})^{2s+a-4\beta-4}}\sum_{\substack{k_2\\k_2\le k_1-C}}2^{(2\beta+2)k_2}\lesssim\sum_{\substack{k_1\\k_1\gg 1}}\frac{1}{(2^{k_1})^{2s-6\beta+a-6}}\lesssim 1.\,\label{2.14}
\end{align}
where we used $2s-6\beta+a-6>0$ as in the estimate of $I_B$ above. Thus, we completed the proof of
$$\mathcal{J}_{1,\mbox{low},\mbox{high},\mbox{bad},1}+\mathcal{J}_{2,\mbox{low},\mbox{high},\mbox{bad},1}-2\mathcal{J}_{3,\mbox{low},\mbox{high},\mbox{bad},1}\lesssim 1.$$\par
\end{proof}
\begin{remark}
\ \par
For later use, we also record that the same argument in both Proposition 2.2 and Proposition 2.3 applies to the high-high regime
$$r=|p_1|\gg |p|\gg 1.$$
Indeed, the proof is identical to that of Proposition 2.2 and 2.3, except for the following two modifications. In the present regime $|p|\gg1$, let $k$ be chosen such that $|p|\sim 2^k$.\par
The first modification is in the cancellation part $I_B$ of Step 2 and concerns only the case $1<a<2$; when $a=1$, no change is needed. Thus, in this part we assume $1<a<2$. In the cancellation part of Step 2, we now have
$$A=|p+h|^a-|p|^a\qquad\mbox{ and }\qquad|A|\lesssim |p|^{a-1}m,$$
where the inequality follows from the mean value theorem. Therefore, we still have
\begin{align*}
    |\mu_+-\mu_-|\lesssim\frac{r^{2-2a} A^2+m^2}{rm}\lesssim \frac{|p|^{2a-2}m^2}{r^{2a-1}m}+\frac{m}{r}=\frac{m}{r}\left(\frac{|p|}{r}\right)^{2a-2}+\frac{m}{r}\lesssim\frac{m}{r},
\end{align*}
where we used $|p|\ll r$ and $a>1$. Moreover, we have
\begin{align*}
    |\mu_0|\lesssim\frac{|A|}{mr^{a-1}}\lesssim\frac{|p|^{a-1}m}{mr^{a-1}}\ll 1,
\end{align*}
which also leads to
$$\frac{|\mu_+-\mu_-|}{1-|\mu_0|}
\lesssim \frac{m}{r}.$$
Consequently, the same spherical difference estimate gives the gain $(m/r)^{1/2}$, which is exactly the same gain as in Proposition 2.3. Another minor difference occurs when computing $p_2$-integration. In the current setting, if $m\ll |p|$, then $|p_2|\approx|p|\sim 2^k$ and
$$\int \frac{|p_2|^{2\beta}}{\ang{p_2}^s}dp_2\lesssim\int_{|p_2-p|\sim 2^l}\frac{dp_2}{|p|^{s-2\beta-2}}\sim\frac{2^{3l}}{(2^k)^{s-2\beta}};$$
whereas if $m\gtrsim |p|$, then 
$$\int \frac{|p_2|^{2\beta}}{\ang{p_2}^s}dp_2\sim (2^l)^{2\beta+3-s}.$$
However, by a similar elementary computation, it turns out that this does not hurt.\par 
The second modification is in the direct non-cancellation part $I_A$ of Step 2 and applies to both the endpoint case $a=1$ and the case $1<a<2$. When $|p_2|\gtrsim |p|$, the proof is the same, because the factor $\displaystyle{\frac{\langle p\rangle^s}{\langle p_2\rangle^s}}$ is harmless. Thus we only need to consider the subregion $|p_2|\sim 2^{k_2}\ll |p|\sim 2^k$. It turns out that in order to kill this additional factor, one needs some further gain. Denote by $\theta$ the angle between $p$ and $p_1$. We claim that:
\begin{align}
\Delta(\cos\theta)\lesssim\frac{u}{|p|}\sim\frac{2^{k_2}}{2^k}\qquad\qquad\mbox{ for all }1\le a<2.\label{2.15}
\end{align}
Equivalently, when integrating in the angular variable of $p_1$, we gain a factor of order $2^{k_2}/2^k$. We prove the claim (\ref{2.15}) later; for now, we show that this gain is sufficient to obtain the desired bound. In fact, in view of (\ref{2.12}), (\ref{2.13}) and (\ref{2.14}), plugging in this angular gain (\ref{2.15}), we obtain
\begin{align*}
    I_A&\lesssim
    \sum_{\substack{k_1,k_2\\k_1\ge k+C\\k_2\le k-C}}2^{(s+2\beta)k}\,\frac{2^{4\beta k_1}\,2^{2\beta k_2}}{2^{2sk_1}\,\ang{2^{k_2}}^s}\,I_{\mbox{in},A,k_1,k_2}\,\frac{2^{k_2}}{2^k}\int_{\left|p_1\right|\sim 2^{k_1}}dp_1\\
    &\sim\sum_{\substack{k_1,k_2\\k_1\ge k+C\\k_2\le k-C}}2^{(s+2\beta)k}\,\frac{2^{4\beta k_1}\,2^{2\beta k_2}}{2^{2sk_1}\,\ang{2^{k_2}}^s}\frac{2^{2k_2}}{\left(2^{k_1}\right)^{a-1}}\frac{2^{k_2}}{2^k}2^{3k_1}\\
    &\sim 2^{(s+2\beta)k}\sum_{\substack{k_1\\k_1\ge k+C}}\frac{1}{(2^{k_1})^{2s+a-4\beta-4}}\sum_{\substack{k_2\\k_2\le k-C}}\frac{(2^{k_2})^{2\beta+3}}{\ang{2^{k_2}}^s}\\
    &\lesssim\begin{cases}
        \displaystyle{\frac{1}{(2^k)^{s-6\beta+a-3}}}\qquad&\mbox{, if }s>2\beta+3\\[8pt]
        \displaystyle{\frac{1}{(2^k)^{s-6\beta+a-3-\eps}}}\qquad&\mbox{, if }s=2\beta+3\\[8pt]
        \displaystyle{\frac{1}{(2^k)^{2s-8\beta+a-6}}}\qquad&\mbox{, if }s<2\beta+3\\
    \end{cases}\ \ \lesssim 1,
\end{align*}
where we used $\log x\lesssim x^\eps$ for some small $\eps>0$.\par
Now, let's turn to prove the claim (\ref{2.15}). Recall our setting that $1\ll |p|\ll|p_1|\approx|p_3|$ and $|p_2|\ll |p|$. By triangle inequality, we have $\Big|\left|\rho\right|-v\Big|\le u$, where $u\ll |\rho|$. Note that $|\rho|\approx v$, so by mean value theorem we get $\Big|\left|\rho\right|^a-v^a\Big|\lesssim r^{a-1}u$. This yields 
$$\Big||p|^a+|p_1|^a-|\rho|^a\Big|=\Big|E-|\rho|^a\Big|\le\Big|E-v^a\Big|+\Big|\left|\rho\right|^a-v^a\Big|\lesssim u^a+r^{a-1}u\lesssim r^{a-1}u.$$
Set
$$G(\cos\theta)=|p|^a+r^a-\left(|p|^2+r^2+2|p|r\cos\theta\right)^{a/2}$$
and then we see
$$G^\prime(\cos\theta)=-a|p|r\left(|p|^2+r^2+2|p|r\cos\theta\right)^{a/2-1}\sim |p|r\,r^{a-2}$$
due to $r\gg |p|$. This gives (\ref{2.15}).\par
This completes the justification of the remark.
\end{remark}

\section{LWP Proof Part II -- large-$\left|p\right|$ regime}
In this chapter, we will prove the boundedness of operator $\mathcal{T}_1,\mathcal{T}_2$ and $\mathcal{T}_3$ when $\left|p\right|$ and/or $\left|p_1\right|$ are $\gg 1$. For convenience, and also to illustrate a different approach, we parametrize the resonant manifold associated with the integral $\mathcal{J}_1$, $\mathcal{J}_2$ and$\mathcal{J}_3$. This allows us to rewrite the integral in a completely explicit and elementary form, which can then be analyzed directly. We will follow this parametrization procedure in \cite{GermainIonescuTran2020}.\par
We first consider $\mathcal{J}_1$. Recall that
$$\mathcal{J}_1\triangleq\sup_p\iiint_{\mathbb{R}^9}\frac{\langle p\rangle^{s}\left|p\right|^{2\beta}\left|p_1\right|^{2\beta}\left|p_2\right|^{2\beta}\left|p_3\right|^{2\beta}}{\langle p_1\rangle^{s}\langle p_2\rangle^{s}\langle p_3\rangle^{s}}\delta(p+p_1-p_2-p_3) \delta\left(\left|p\right|^a+\left|p_1\right|^a-\left|p_2\right|^a-\left|p_3\right|^a\right)\,dp_1 dp_2 dp_3.$$
Since $\left|p\right|^a+\left|p_1\right|^a=\left|p_2\right|^a+\left|p_3\right|^a$, we have either $\ang{p}\lesssim\ang{p_2}$ or $\ang{p}\lesssim\ang{p_3}$. Without loss of generality, we assume $\ang{p}\lesssim\ang{p_3}$ here. Therefore, we obtain
\begin{align}
    \mathcal{J}_1&\lesssim\sup_p\iiint_{\mathbb{R}^9}\frac{\left|p\right|^{2\beta}\left|p_1\right|^{2\beta}\left|p_2\right|^{2\beta}\left|p_3\right|^{2\beta}}{\langle p_1\rangle^{s}\langle p_2\rangle^{s}}\delta(p+p_1-p_2-p_3) \delta\left(\left|p\right|^a+\left|p_1\right|^a-\left|p_2\right|^a-\left|p_3\right|^a\right)\,dp_1 dp_2 dp_3\notag\\
    &\lesssim\sup_p\iint_{\mathbb{R}^6}\frac{\left|p\right|^{4\beta}\left|p_1\right|^{2\beta}\left|p_2\right|^{2\beta}}{\langle p_1\rangle^{s}\langle p_2\rangle^{s}} \delta\left(\left|p\right|^a+\left|p_1\right|^a-\left|p_2\right|^a-\left|p+p_1-p_2\right|^a\right)\,dp_1 dp_2\notag\\
    &\lesssim\sup_p\int_{\mathbb{R}^3}\left(\int_{\mathcal{S}_{p,p_1}}\frac{\left|p\right|^{4\beta}d\mu(z)}{\ang{p_1}^{s-2\beta}\ang{z}^{s-2\beta}\left|\nabla_z\Phi(z)\right|}\right)dp_1,\label{3.1}
\end{align}
where $z=p_2$, $\mathcal{S}_{p,p_1}\triangleq\left\{z:\Phi(z)\triangleq\left|p+p_1-z\right|^a+\left|z\right|^a-\left|p\right|^a-\left|p_1\right|^a=0\right\}$ and $\mu$ is the surface measure on $\mathcal{S}_{p,p_1}$. (Note that $\mu$ here has different meaning as in Chapter 2.)\par
Now, we are ready to parameterize the resonant manifold $\mathcal{S}_{p,p_1}$. In order to do this, we compute
$$\nabla_z\Phi=a\frac{z-\rho}{\left|z-\rho\right|}\left|\rho-z\right|^{a-1}+a\frac{z}{\left|z\right|}\left|z\right|^{a-1}.$$
In particular, let $q$ be any vector orthogonal to $\rho$, i.e. $\rho\cdot q=0$. The directional derivative of $\Phi$ in the direction of $q$, with $z=\alpha\rho+q$, $\alpha\in\mathbb{R}$, satisfies
$$q\cdot\nabla_z\Phi=a\left|q\right|^2\left[\left|\rho-z\right|^{a-1}+\left|z\right|^{a-1}\right]>0,$$
which means that $\Phi(z)$ is strictly increasing in any direction that is orthogonal to $\rho$. This proves that the intersection between the surface $\mathcal{S}_{p,p_1}$ and the plane
$$\mathcal{P}_\alpha=\left\{\alpha \rho+q,\ \rho\cdot q=0\right\}$$
is either empty or the circle centered at $\alpha\rho$ and of a finite radius $r_\alpha$, for $\alpha\in \mathbb{R}$. As a consequence, we can parametrize $S_{p,p_1}$ as follows. Let $\rho^\perp$ be the vector orthogonal to both $\rho$ and a fixed vector $e$ of $\mathbb{R}^3$, and let $e_\theta$ be the unit vector in $\mathcal{P}_0=\left\{\rho\cdot q=0\right\}$ such that the angle between $\rho^\perp$ and $e_\theta$ is $\theta$. We parameterize $S_{p,p_1}$ by
\begin{align*}
\left\{z=\alpha \rho+r_\alpha e_\theta:\theta\in[0,2\pi],\ \alpha\in A_{p,p_1}\right\},
\end{align*}
where $A_{p,p_1}$ is the set of $\alpha$ for which a solution to $\Phi(z)=0$ exists.\par
We can also think of $\Phi$ as a function of $\alpha$ and $r$ as $\Phi=\Phi(r,\alpha)$. (Note that this $r$ temporarily does not represent $\left|p_1\right|$ here.) We just saw that $\partial_r \Phi>0$ for $r>0$. Therefore, by the implicit function theorem, the zero set of $\Phi$ can be parameterized as
$$\left\{(\alpha,r=r_\alpha),\ \alpha\in A_{p,p_1}\right\},$$
where $\alpha\mapsto r_\alpha$ is a smooth function on $A_{p,p_1}$ vanishing on its boundary. Moreover, we have by definition that $\Phi(z_\alpha)=0$ for all $\alpha$ and therefore, keeping $\theta$ fixed,
\begin{align}
0
&= \partial_\alpha z_\alpha\cdot \nabla_z \Phi= \partial_\alpha z_\alpha\cdot\left(a\frac{z_\alpha-\rho}{|z_\alpha-\rho|}\left|z_\alpha-\rho\right|^{a-1}+a\frac{z_\alpha}{|z_\alpha|}\left|z_\alpha\right|^{a-1}\right) \notag\\
&= \partial_\alpha z_\alpha\cdot\left(a\frac{z_\alpha}{|z_\alpha-\rho|}\left|z_\alpha-\rho\right|^{a-1}+a\frac{z_\alpha}{|z_\alpha|}\left|z_\alpha\right|^{a-1}\right)-a\partial_\alpha z_\alpha\cdot\frac{\rho}{|z_\alpha-\rho|}\left|z_\alpha-\rho\right|^{a-1} \label{3.2}\\
&= \frac{a}{2} \partial_\alpha |z_\alpha|^2\left(\left|\rho-z_\alpha\right|^{a-2}+\left|z_\alpha\right|^{a-2}\right)-a|\rho|^2\left|\rho-z_\alpha\right|^{a-2},\notag
\end{align}
where we used $\partial_\alpha z_\alpha=  \rho+e_\theta\partial_\alpha r_\alpha$ and $e_\theta\perp\rho$.
Therefore,
\begin{align}
\partial_\alpha \left|z_\alpha\right|^2=2\frac{\displaystyle\left|\rho-z_\alpha\right|^{a-2}|\rho|^2}{\displaystyle\left|\rho-z_\alpha\right|^{a-2}+\left|z_\alpha\right|^{a-2}}.\label{3.3}
\end{align}
This implies in particular that $\alpha\mapsto |z_\alpha|$ is increasing on $A_{p,p_1}$. Defining $r$ to be zero on the complement of $A_{p,p_1}$, we get that $\alpha\mapsto |z_\alpha|$ is an increasing function on $\mathbb R$; therefore, the change of coordinates $\alpha\mapsto |z_\alpha|$ is well-defined.\par
Next, we need to compute the surface measure on the resonant manifold. Since $\partial_\theta e_\theta$ is orthogonal to both $\rho$ and $e_\theta$, we compute
the surface area
$$d\mu(z)=\left|\partial_\alpha z\times \partial_\theta z\right|\,d\alpha d\theta=\left|\left(\rho+\partial_\alpha r_\alpha e_\theta\right)\times r_\alpha \partial_\theta e_\theta\right|\,d\alpha d\theta=\sqrt{\left|\rho\right|^2 r_\alpha^2+\frac14 \left|\partial_\alpha(r_\alpha^2)\right|^2}\,d\alpha d\theta.$$
Using $\left|z\right|^2=\alpha^2\left|\rho\right|^2+r_\alpha^2$,we learn from the last line of (\ref{3.2}) that
\begin{align}
\partial_\alpha r_\alpha^2=-2|\rho|^2\frac{\displaystyle\alpha \left|z_\alpha\right|^{a-2}+(\alpha-1)\left|\rho-z_\alpha\right|^{a-2}}{\displaystyle\left|\rho-z_\alpha\right|^{a-2}+\left|z_\alpha\right|^{a-2}}.\label{3.4}
\end{align}
Now, let us compute $|\nabla_z \Phi|$ under the new parameterization:
\begin{align*}
|\nabla_z \Phi|^2
&=
\left|a\frac{z_\alpha}{|z_\alpha|}\left|z_\alpha\right|^{a-1}+a\frac{z_\alpha-\rho}{|z_\alpha-\rho|}\left|z_\alpha-\rho\right|^{a-1}\right|^2 \\
&=\left|a\frac{\alpha\rho+q}{|z_\alpha|}\left|z_\alpha\right|^{a-1}+a\frac{(\alpha-1)\rho+q}{|z_\alpha-\rho|}\left|z_\alpha-\rho\right|^{a-1}\right|^2 \\
&=a^2|\rho|^2\left[\alpha\left|z_\alpha\right|^{a-2}+(\alpha-1)\left|\rho-z_\alpha\right|^{a-2}\right]^2+a^2 r_\alpha^2\left[\left|\rho-z_\alpha\right|^{a-2}+\left|z_\alpha\right|^{a-2}\right]^2 .
\end{align*}
In addition to (\ref{3.4}), this implies that
\begin{align*}
|\nabla_z \Phi|^2=\frac{a^2|\partial_\alpha r_\alpha^2|^2}{4|\rho|^2}\left[\left|\rho-z_\alpha\right|^{a-2}+\left|z_\alpha\right|^{a-2}\right]^2+a^2 r_\alpha^2\left[\left|\rho-z_\alpha\right|^{a-2}+\left|z_\alpha\right|^{a-2}\right]^2 .
\end{align*}
Therefore,
\begin{align}
\frac{d\mu(z)}{|\nabla_z \Phi|}=\frac{|\rho|}{\displaystyle a\left(\left|\rho-z_\alpha\right|^{a-2}+\left|z_\alpha\right|^{a-2}\right)}\,d\alpha d\theta .\label{3.5}
\end{align}
which, by (\ref{3.3}), implies that
$$\frac{d\mu(z)}{|\nabla_z \Phi|}=\frac{1}{a\left|\rho-z_\alpha\right|^{a-2}}\frac{|z|}{\left|\rho\right|}\,d|z| \,d\theta.$$
Set $h_\alpha\triangleq r_\alpha^2$. Since $\left|z\right|=\sqrt{r_\alpha^2+\alpha^2\left|\rho\right|^2}$, we have 
\begin{align}
\left(h_\alpha+\left(1-\alpha\right)^2\left|\rho\right|^2\right)^{a/2}+\left(h_\alpha+\alpha^2\left|\rho\right|^2\right)^{a/2}=C_{p,p_1}\triangleq\left|p\right|^a+\left|p_1\right|^a=\left|\rho-z\right|^a+\left|z\right|^a.\label{3.6}
\end{align}
Thus, by (\ref{3.3}), (\ref{3.6}) and $\displaystyle{\int_0^{2\pi}d\theta\sim1}$, we finally get
\begin{align}
\mathcal{J}_1&\lesssim\sup_p\int_{\mathbb{R}^3}\frac{\left|p\right|^{4\beta}}{\ang{p_1}^{s-2\beta}}\int_0^{+\infty}\int_0^{2\pi}\frac{\left|z\right|}{\ang{z}^{s-2\beta}\left|\rho-z\right|^{a-2}\left|\rho\right|}\,d\theta d\left|z\right| dp_1.\label{para}\\
&\sim\sup_p\int_{\mathbb{R}^3}\frac{\left|p\right|^{4\beta}\left|\rho\right|}{\ang{p_1}^{s-2\beta}}\int\int_0^{2\pi}\frac{d\theta d\alpha}{\ang{z}^{s-2\beta}\left(\left|\rho-z\right|^{a-2}+\left|z\right|^{a-2}\right)}\,dp_1\notag\\
&\sim\displaystyle{\sup_p\int_{\mathbb{R}^3}\frac{\left|p\right|^{4\beta}\left|\rho\right|}{\ang{p_1}^{s-2\beta}}\int\int_0^{2\pi}\frac{d\theta d\alpha}{\ang{z}^{s-2\beta}\left[\left(h_\alpha+\left|\rho\right|^2\left(1-\alpha\right)^2\right)^{(a-2)/2}+\left(h_\alpha+\left|\rho\right|^2\alpha^2\right)^{(a-2)/2}\right]}\,dp_1}\notag\\
&\lesssim\displaystyle{\sup_p\int_{\mathbb{R}^3}\frac{\left|p\right|^{4\beta}\left|\rho\right|}{\ang{p_1}^{s-2\beta}}\int\frac{d\alpha}{\ang{z}^{s-2\beta}\left[\left(h_\alpha+\left|\rho\right|^2\left(1-\alpha\right)^2\right)^{(a-2)/2}+\left(h_\alpha+\left|\rho\right|^2\alpha^2\right)^{(a-2)/2}\right]}\,dp_1}.\label{3.7}
\end{align}
For later purpose, we denote 
\begin{align}
    c_\star\triangleq\frac{\left|p\right|^a+r^a}{\left|\rho\right|^a}\label{3.8}
\end{align}
and from now on, in the proof below, we shall return to the notation used previously: $\rho=p+p_1$ and $r=\left|p_1\right|$.\par

\subsection{The Cancellation Case}
\ \par
First, we consider a case in which cancellation from the nonlinear term is needed.
\begin{prop}
    When $c_*\gg 1$, we have 
    $$\mathcal{J}_{1,\mbox{high},\mbox{high}}+\mathcal{J}_{2,\mbox{high},\mbox{high}}-2\mathcal{J}_{3,\mbox{high},\mbox{high}}\lesssim 1.$$
\end{prop}
\begin{proof}
First, since $\displaystyle{c^{1/a}_\star\sim \frac{\left|p\right|+\left|p_1\right|}{\left|p+p_1\right|}\gg 1}$, we must have $\left|p\right|\approx\left|p_1\right|\gg\left|\rho\right|$. Due to \vspace{0.3em}$\left|p\right|^a+\left|p_1\right|^a=\left|p_2\right|^a+\left|p_3\right|^a$, we must have $\left|p\right|\approx\left|p_1\right|\approx\left|p_2\right|\approx\left|p_3\right|$. Then pick $k\in\mathbb{Z}^+$, $k\gg 1$ such that $\left|p\right|\sim 2^k$. Therefore\vspace{0.3em}, we get \vspace{0.5em}$\left|p\right|\approx\left|p_1\right|\approx\left|p_2\right|\approx\left|p_3\right|\sim 2^k$. Now, we perform dyadic decomposition over $\left|\rho\right|$ as $\left|\rho\right|\sim 2^{k_1}$, where $k_1\le k-C$ and $C>0$ is a fixed large number. With these notation, we immediately see that $\displaystyle{c_\star=\frac{2^{ak}}{2^{ak_1}}=2^{a(k-k_1)}}$.\par
Now, let's explore the cancellation effect. We first write
\begin{align}
    &\frac{\langle p\rangle^s}{\langle p_1\rangle^s\langle p_2\rangle^s\langle p_3\rangle^s}+\frac{1}{\langle p_2\rangle^s\langle p_3\rangle^s}-\frac{2}{\langle p_1\rangle^s\langle p_2\rangle^s}\notag\\[5pt]
    =&\frac{1}{\langle p_2\rangle^s}\left[\frac{1}{\langle p_3\rangle^s}\left(1+\frac{\langle p\rangle^s}{\langle p_1\rangle^s}\right)-\frac{2}{\langle p_1\rangle^s}\right]\notag\\[5pt]
    =&\frac{1}{\langle p_2\rangle^s}\left[2\left(\frac{1}{\langle p_3\rangle^s}-\frac{1}{\langle p_1\rangle^s}\right)+\frac{1}{\langle p_3\rangle^s}\left(\frac{\langle p\rangle^s}{\langle p_1\rangle^s}-1\right)\right],\label{3.9}
\end{align}
By triangle inequality, we have $\bigl|\left|p\right|-\left|p_1\right|\bigr|\le|p+p_1|=|\rho|\sim 2^{k_1}$. Moreover, by mean value theorem and $\left|p\right|^a+\left|p_1\right|^a=\left|p_2\right|^a+\left|p_3\right|^a$, we also see that $\bigl|\left|p_1\right|-\left|p_3\right|\bigr|\le|p+p_1|=|\rho|\sim 2^{k_1}$. Therefore, apply mean value theorem once again, we obtain
\begin{align*}
    \begin{cases}
        \displaystyle{\left|\frac{1}{\langle p_3\rangle^s}-\frac{1}{\langle p_1\rangle^s}\right|\lesssim\frac{1}{2^{(s+1)k}}\bigl|\left|p\right|-\left|p_1\right|\bigr|\lesssim\frac{2^{k_1}}{2^{(s+1)k}}}\vspace{0.8em}\\
        \displaystyle{\left|\frac{\langle p\rangle^s}{\langle p_1\rangle^s}-1\right|\lesssim\frac{1}{2^k}\bigl|\left|p_1\right|-\left|p_3\right|\bigr|\lesssim\frac{2^{k_1}}{2^k}}
    \end{cases}.
\end{align*}
Thus, in view of (\ref{3.9}), we compute
\begin{align*}
    &\left|p\right|^{2\beta}\left|p_1\right|^{2\beta}\left|p_2\right|^{2\beta}\left|p_3\right|^{2\beta}\left[\frac{\langle p\rangle^s}{\langle p_1\rangle^s\langle p_2\rangle^s\langle p_3\rangle^s}+\frac{1}{\langle p_2\rangle^s\langle p_3\rangle^s}-\frac{2}{\langle p_1\rangle^s\langle p_2\rangle^s}\right]\\
    \lesssim\,&2^{8\beta k}\,\frac{1}{2^{sk}}\,\frac{2^{k_1}}{2^{(s+1)k}}=\frac{2^{k_1}}{2^{(2s+1-8\beta)k}}.
\end{align*}
Recall from (\ref{3.5}) that 
$$\frac{d\mu(z)}{\left|\nabla_z\Phi(z)\right|}\sim \frac{|\rho|}{\displaystyle \left|\rho-z\right|^{a-2}+\left|z\right|^{a-2}}\,d\alpha d\theta\sim\frac{2^{k_1}}{2^{(a-2)k}}d\alpha\, d\theta.$$
In addition, we observe that $\left|\alpha\right|\lesssim 2^{k-k_1}$. This is due to $\left|z\right|\sim 2^k$ and $\left|z\right|^2=\alpha^2\left|\rho\right|^2+r_\alpha^2\ge\alpha^2\left|\rho\right|^2$.
Then, as in (\ref{3.1}), this gives us that
\begin{align*}
&\mathcal{J}_{1,\mbox{high},\mbox{high}}+\mathcal{J}_{2,\mbox{high},\mbox{high}}-2\mathcal{J}_{3,\mbox{high},\mbox{high}}\\
\lesssim &\sup_{\substack{k\\k\gg 1}}\sum_{\substack{k_1\\k_1\le k-C}}\frac{2^{k_1}}{2^{(2s+1-8\beta)k}}\int_{\left|p_1\right|\sim 2^k}\left(\int_{\mathcal{S}_{p,p_1}}\frac{d\mu(z)}{\left|\nabla_z\Phi(z)\right|}\right)dp_1.\\
\lesssim &\sup_{\substack{k\\k\gg 1}}\sum_{\substack{k_1\\k_1\le k-C}}\frac{2^{k_1}}{2^{(2s+1-8\beta)k}}\frac{2^{k_1}}{2^{(a-2)k}}\int_{\left|p_1\right|\sim 2^k}\int_{\left|\alpha\right|\lesssim 2^{k-k_1}}d\alpha\,dp_1\\
\lesssim &\sup_{\substack{k\\k\gg 1}}\sum_{\substack{k_1\\k_1\le k-C}}\frac{2^{k_1}}{2^{(2s+1-8\beta)k}}\frac{2^{k_1}}{2^{(a-2)k}}\,2^{k-k_1}\,2^{3k_1}=\sup_{\substack{k\\k\gg 1}}\frac{1}{2^{(2s-8\beta+a-2)k}}\sum_{\substack{k_1\\k_1\le k-C}}2^{4k_1}\\
\lesssim &\sup_{\substack{k\\k\gg 1}}\frac{1}{2^{(2s-8\beta+a-6)k}}\lesssim 1,
\end{align*}
where we used $\displaystyle{s\ge4\beta+3-\frac{a}{2}}$ in the last inequality.
\end{proof}
\begin{remark}
    In fact, the proof above also works under the weaker assumption $k\ge1$, instead of $k\gg1$, in the dyadic regime $|p|\approx |p_1|\approx |p_2|\approx |p_3|\sim 2^k$.
\end{remark}
Now it remains to consider the case $c_\star \lesssim 1$, so from now on we assume $c_\star \lesssim 1$. Thanks to Proposition 2.1, we can assume that $\max\left(|p|,|p_1|\right)\gg 1$. Therefore, we can also assume $|\rho|\gtrsim 1$. This is because if $|\rho|\ll 1$, then we must have $|p|\approx|p_1|\gg 1$, which implies $c_\star\gg 1$. We divide the argument into the following three cases: (i) $2\le a\le 5$;  (ii) $1<a<2$ and (iii) $a=1$.

\subsection{The Non-cancellation Case: $2\le a\le 5$}
\ \par
In this case, we use the identity
$$\frac{1}{A^{t-1}+B^{t-1}}\sim \frac{\max\left(A,B\right)}{A^t+B^t},$$
where $A,B>0$ and $t\ge 1$. Therefore, in view of (\ref{3.7}) we have
\begin{align*}
    \mathcal{J}_1&\lesssim\sup_p\int_{\mathbb{R}^3}\frac{\left|p\right|^{4\beta}\left|\rho\right|}{\ang{p_1}^{s-2\beta}\left(\left|p\right|^a+\left|p_1\right|^a\right)}\int\frac{\max\left(h_\alpha+\left|\rho\right|^2\left(1-\alpha\right)^2,h_\alpha+\left|\rho\right|^2\alpha^2\right)}{\ang{z}^{s-2\beta}}\,d\alpha dp_1\\
    &\lesssim \sup_p\int_{\mathbb{R}^3}\frac{\left|p\right|^{4\beta}\left|\rho\right|}{\ang{p_1}^{s-2\beta}\left(\left|p\right|^a+\left|p_1\right|^a\right)}\left(\int_{\substack{\alpha\le 1/2\\h_\alpha\ge 0}}\frac{h_\alpha+\left|\rho\right|^2\left(1-\alpha\right)^2}{\ang{z}^{s-2\beta}}d\alpha+\int_{\substack{\alpha> 1/2\\h_\alpha\ge 0}}\frac{h_\alpha+\left|\rho\right|^2\alpha^2}{\ang{z}^{s-2\beta}}\, d\alpha\right)\, dp_1\\
    &\sim\sup_p\int_{\mathbb{R}^3}\frac{\left|p\right|^{4\beta}\left|\rho\right|}{\ang{p_1}^{s-2\beta}\left(\left|p\right|^a+\left|p_1\right|^a\right)}\left(\int_{\substack{\alpha\le 1/2\\h_\alpha\ge 0}}\frac{h_\alpha+\left|\rho\right|^2\left(1-\alpha\right)^2}{\ang{h_\alpha+\left|\rho\right|^2\alpha^2}^{s/2-\beta}}d\alpha+\right.\\
    &\qquad\qquad\qquad\qquad\qquad\qquad\qquad\qquad\qquad\qquad\left.\int_{\substack{\alpha> 1/2\\h_\alpha\ge 0}}\frac{h_\alpha+\left|\rho\right|^2\alpha^2}{\ang{h_\alpha+\left|\rho\right|^2\alpha^2}^{s/2-\beta}}\, d\alpha\right)\, dp_1,
\end{align*}
where $\left|z\right|=\sqrt{r_\alpha^2+\alpha^2\left|\rho\right|^2}=\sqrt{h_\alpha+\alpha^2\left|\rho\right|^2}$ is used in the last step.\par
If $h_\alpha^{a/2}\gtrsim \left|\rho\right|^a\max\left(\left|1-\alpha\right|^a,\left|\alpha\right|^a\right)$, then in view of (\ref{3.6}), we must have 
$$h_\alpha\sim C_{p,p_1}^{2/a}=\left(\left|p\right|^a+\left|p_1\right|^a\right)^{2/a},$$
and
$$\left|\alpha\right|\lesssim\frac{C_{p,p_1}^{1/a}}{\left|\rho\right|}=c_\star^{1/a}=\frac{\left(\left|p\right|^a+r^a\right)^{1/a}}{\left|\rho\right|}.$$
Therefore, in view of (\ref{3.7}), we see that
\begin{align*}
    \mathcal{J}_1&\lesssim\sup_p\int_{\mathbb{R}^3}\frac{\left|p\right|^{4\beta}\left|\rho\right|}{\ang{p_1}^{s-2\beta}\left(\left|p\right|^a+\left|p_1\right|^a\right)}\left(\int_{\substack{\alpha\le 1/2\\\left|\alpha\right|\lesssim c_\star^{1/a}}}\frac{\left(\left|p\right|^a+r^a\right)^{2/a}+\left|\rho\right|^2\left(1-\alpha\right)^2}{\ang{\left(\left|p\right|^a+r^a\right)^{2/a}+\left|\rho\right|^2\alpha^2}^{s/2-\beta}}d\alpha\right.\\
    &\left.\qquad\qquad\qquad\qquad\qquad+\int_{\substack{\alpha> 1/2\\\left|\alpha\right|\lesssim c_\star^{1/a}}}\frac{\left(\left|p\right|^a+r^a\right)^{2/a}+\left|\rho\right|^2\alpha^2}{\ang{\left(\left|p\right|^a+r^a\right)^{2/a}+\left|\rho\right|^2\alpha^2}^{s/2-\beta}}\, d\alpha\right)\, dp_1\\
    &\triangleq \sup_p\int_{\mathbb{R}^3}\frac{\left|p\right|^{4\beta}\left|\rho\right|}{\ang{p_1}^{s-2\beta}\left(\left|p\right|^a+\left|p_1\right|^a\right)}\left(I_{11}+I_{12}\right)\, dp_1.
\end{align*}\par
If $h_\alpha^{a/2}\ll \left|\rho\right|^a\max\left(\left|1-\alpha\right|^a,\left|\alpha\right|^a\right)$, then denote
\begin{align}
G(\alpha,h_\alpha)=\left(h_\alpha+\left(1-\alpha\right)^2\left|\rho\right|^2\right)^{a/2}+\left(h_\alpha+\alpha^2\left|\rho\right|^2\right)^{a/2}\label{3.10}
\end{align}
and we see that
$$\partial_{h_\alpha}G(\alpha,h_\alpha)=\frac{a}{2}\left(h_\alpha+\left(1-\alpha\right)^2\left|\rho\right|^2\right)^{\frac{a}{2}-1}+\frac{a}{2}\left(h_\alpha+\alpha^2\left|\rho\right|^2\right)^{\frac{a}{2}-1},$$
which implies that 
$$\partial_{h_\alpha}G\sim \left|\rho\right|^{a-2}\max\left(\left|1-\alpha\right|^{a-2},\left|\alpha\right|^{a-2}\right)$$ 
provided that $h_\alpha^{a/2}\ll \left|\rho\right|^a\max\left(\left|1-\alpha\right|^a,\left|\alpha\right|^a\right)$. Therefore, by mean value theorem, we obtain that
\begin{align}
h_\alpha\sim\frac{G(\alpha,h_\alpha)-G(\alpha,0)}{\partial_{h_\alpha}G(\alpha,\xi)}\sim \frac{C_{p,p_1}-\left|\rho\right|^a\left|1-\alpha\right|^a-\left|\rho\right|^a\left|\alpha\right|^a}{\left|\rho\right|^{a-2}\max\left(\left|1-\alpha\right|^{a-2},\left|\alpha\right|^{a-2}\right)}\sim \left|\rho\right|^2 \frac{c_\star-\left|1-\alpha\right|^a-\left|\alpha\right|^a}{\max\left(\left|1-\alpha\right|^{a-2},\left|\alpha\right|^{a-2}\right)},\label{3.11}
\end{align}
where $\xi\in(0,h_\alpha)$. Moreover, since $h_\alpha\ge 0$, we have $c_\star-\left|1-\alpha\right|^a-\left|\alpha\right|^a\ge 0$. Together with the assumption $c_* \lesssim 1$, this implies $|\alpha|\lesssim 1$. Therefore, (\ref{3.11}) can be further rewritten as
$$h_\alpha\sim \left|\rho\right|^2\left(c_\star-\left|1-\alpha\right|^a-\left|\alpha\right|^a\right)$$
Thus, in view of (\ref{3.7}), we can write
\begin{align*}
    \mathcal{J}_1&\lesssim\sup_p\int_{\mathbb{R}^3}\frac{\left|p\right|^{4\beta}\left|\rho\right|}{\ang{p_1}^{s-2\beta}\left(\left|p\right|^a+\left|p_1\right|^a\right)}\left(\int_{\substack{\alpha\le 1/2\\h_\alpha\ge 0}}\frac{c_\star-\left|1-\alpha\right|^a-\left|\alpha\right|^a+\left|\rho\right|^2\left(1-\alpha\right)^2}{\ang{\left|\rho\right|^2 \left(c_\star-\left|1-\alpha\right|^a-\left|\alpha\right|^a\right)+\left|\rho\right|^2\alpha^2}^{s/2-\beta}}d\alpha\right.\\
    &\qquad\qquad\qquad\qquad\qquad\left.+\int_{\substack{\alpha> 1/2\\h_\alpha\ge 0}}\frac{c_\star-\left|1-\alpha\right|^a-\left|\alpha\right|^a+\left|\rho\right|^2\alpha^2}{\ang{\left|\rho\right|^2 \left(c_\star-\left|1-\alpha\right|^a-\left|\alpha\right|^a\right)+\left|\rho\right|^2\alpha^2}^{s/2-\beta}}\, d\alpha\right)\, dp_1.\\
    &\triangleq\sup_p\int_{\mathbb{R}^3}\frac{\left|p\right|^{4\beta}\left|\rho\right|}{\ang{p_1}^{s-2\beta}\left(\left|p\right|^a+\left|p_1\right|^a\right)}\left(I_{21}+I_{22}\right)dp_1,
\end{align*}
where 
\begin{align}
    I_{21}&\triangleq\int_{\substack{\alpha\le 1/2\\h_\alpha\ge 0}}\frac{\left|\rho\right|^2\left(c_\star-\left|1-\alpha\right|^a-\left|\alpha\right|^a\right)+\left|\rho\right|^2\left(1-\alpha\right)^2}{\ang{\left|\rho\right|^2 \left(c_\star-\left|1-\alpha\right|^a-\left|\alpha\right|^a\right)+\left|\rho\right|^2\alpha^2}^{s/2-\beta}}d\alpha\notag\\
    &=\frac{1}{\left|\rho\right|^{s-2\beta-2}}\int_{\substack{\alpha\le 1/2\\h_\alpha\ge 0}}\frac{\left|c_\star-\left|1-\alpha\right|^a-\left|\alpha\right|^a+\left(1-\alpha\right)^2\right|}{\left[c_\star-\left|1-\alpha\right|^a-\left|\alpha\right|^a+\alpha^2+\frac{1}{\left|\rho\right|^2}\right]^{s/2-\beta}} d\alpha\label{3.12}
\end{align}
and $I_{22}$ can be written analogously. As in Chapter 1, We define $I_{21,\mbox{low},\mbox{high}}$, $I_{21,\mbox{high},\mbox{low}}$, $I_{21,\mbox{high},\mbox{high}}$, etc., accordingly.\par
Next two lemmas give the estimate of $I_{1}$ and $I_{2}$. Before proceeding, we first note that $I_{12}\leq I_{11}$ and $I_{22}\leq I_{21}$.
\begin{lemma}
    Assume $2\le a\le 5$. Then $\displaystyle{I_1\sim I_{11}\lesssim\frac{1}{\left|\rho\right|^{s-2\beta-2}}=\frac{1}{\left(|p|+|p_1|\right)^{s-2\beta-2}}}$.
\end{lemma}
\begin{proof}
    Since $h_\alpha\ge 0$, (\ref{3.11}) tells us that $c_\star\ge 2^{1-a}$. Recall that we already assume $c_\star\lesssim 1$. This means that $c_\star\sim 1$, which implies $\left|p\right|^a+\left|p_1\right|^a\approx\left|\rho\right|^a$ and $\left|p\right|+\left|p_1\right|\approx\left|\rho\right|$.\par
    Since $\left(1-\alpha\right)^2\lesssim 1$, we see the numerator
    $$\left(|p|^a+|p_1|^a\right)^{2/a}+|\rho|^2 \left(1-\alpha\right)^2\lesssim |\rho|^2+|\rho|^2 \sim |\rho|^2.$$
    A same argument gives us that the denumerator $\lesssim \ang{|\rho|^2}^{s/2-\beta}\sim |\rho|^{s-2\beta}$ due to our previous assumption that $|\rho|\gtrsim 1$.\par
    Finally, note that the integration interval has length $O(1)$ and we conclude that 
    $$I_{11}\lesssim\frac{1}{\left|\rho\right|^{s-2\beta-2}}=\frac{1}{\left(|p|+|p_1|\right)^{s-2\beta-2}}.$$
\end{proof}

\begin{lemma}
    Assume $2\le a\le 5$. Pick a small $\eps$ such that $0<\eps\ll 1$. Then we have the following estimates of $I_{2}$:
    \begin{enumerate}[label=(\roman*),itemsep=6pt]
        \item When $1+\eps\le c_\star\lesssim 1$, $\displaystyle{I_2\sim I_{21}\lesssim\frac{1}{|\rho|^{s-2\beta-2}}}$;
        \item When $2^{1-a}\le c_\star<1+\eps$, 
        $$I_2\sim I_{21}\lesssim\begin{cases}
            \displaystyle{\frac{1}{|\rho|^{s-2\beta-2}}\cdot\frac{1}{\left[\left(c_\star-1\right)^2+\frac{1}{|\rho|^2}\right]^{s/2-\beta-1}}}\qquad&\mbox{, if }s>2\beta+2\vspace{0.5em}\\
            \displaystyle{\frac{1}{|\rho|^{s-2\beta-2}}\cdot\log\left(1+\frac{1}{\left(c_\star-1\right)^2+\frac{1}{|\rho|^2}}\right)}\qquad&\mbox{, if }s=2\beta+2\vspace{0.6em}\\
            \displaystyle{\frac{1}{|\rho|^{s-2\beta-2}}\cdot\left(1+\frac{1}{|\rho|^2}\right)^{\beta+1-s/2}}\qquad&\mbox{, if }s<2\beta+2
        \end{cases};$$
        \item When $c_\star<2^{1-a}$, $I_2=I_{21}=0$.
    \end{enumerate}
\end{lemma}
\begin{proof}
First, observe that when $c_\star<2^{1-a}$, there is no solution to the condition $h_\alpha\ge 0$. Hence, by definition, we immediately have $I_{21}=0$. Next, we split the proof into three cases according to the size of $c_\star$.\par
\noindent\underline{Case 1. $1+\eps\le c_\star\lesssim1$}\par
In this case, we rewrite (\ref{3.12}) as
\begin{align*}
    I_{21}&=\frac{1}{\left|\rho\right|^{s-2\beta-2}}\left(\int_{\substack{\alpha< 0\\h_\alpha\ge 0}}\frac{\left|c_\star-\left|1-\alpha\right|^a-\left|\alpha\right|^a+\left(1-\alpha\right)^2\right|}{\left[c_\star-\left|1-\alpha\right|^a-\left|\alpha\right|^a+\alpha^2+\frac{1}{\left|\rho\right|^2}\right]^{s/2-\beta}} d\alpha\right.\\
    &\qquad\qquad\qquad\qquad\left.+\int_{\substack{0\le \alpha\le 1/2\\h_\alpha\ge 0}}\frac{\left|c_\star-\left|1-\alpha\right|^a-\left|\alpha\right|^a+\left(1-\alpha\right)^2\right|}{\left[c_\star-\left|1-\alpha\right|^a-\left|\alpha\right|^a+\alpha^2+\frac{1}{\left|\rho\right|^2}\right]^{s/2-\beta}} d\alpha\right)\\
    &=\frac{1}{\left|\rho\right|^{s-2\beta-2}}\int_0^{O(1)}\frac{\left|c_\star-\left|1+y\right|^a-\left|y\right|^a+\left(1+y\right)^2\right|}{\left[c_\star-\left|1+y\right|^a-\left|y\right|^a+y^2+\frac{1}{\left|\rho\right|^2}\right]^{s/2-\beta}} dy\\
    &\qquad\qquad\qquad\qquad+\frac{1}{\left|\rho\right|^{s-2\beta-2}}\int_0^{1/2}\frac{\left|c_\star-\left|1-\alpha\right|^a-\left|\alpha\right|^a+\left(1-\alpha\right)^2\right|}{\left[c_\star-\left|1-\alpha\right|^a-\left|\alpha\right|^a+\alpha^2+\frac{1}{\left|\rho\right|^2}\right]^{s/2-\beta}} d\alpha.
\end{align*}\par
For the first term, since $0\le y\lesssim 1$, the numerator satisfies
\begin{align}
c_\star-\left|1+y\right|^a-\left|y\right|^a+\left(1+y\right)^2\lesssim 1.\label{3.13}
\end{align}
Moreover, we observe that 
\begin{align}
c_\star-\left|1+y\right|^a-\left|y\right|^a+y^2+\frac{1}{\left|\rho\right|^2}\gtrsim 1.\label{3.14}
\end{align}
In fact, denote $\delta_0\triangleq1+\eps$ and then $c_\star-1\ge \delta_0$. We can pick $\eta=\eta(\delta_0)$ such \vspace{0.5em} that whenever $0\le y\le \eta$ we have $\displaystyle{\left(1+y\right)^a-1+y^a\le \frac{\delta_0}{2}}$. Therefore, when $0\le y\le \eta$, the denominator satisfies
\begin{align}
c_\star-\left|1+y\right|^a-\left|y\right|^a+y^2+\frac{1}{\left|\rho\right|^2}\ge c_\star-\left|1+y\right|^a-\left|y\right|^a=c_\star-1-\left[\left(1+y\right)^a-1+y^a\right]\ge\frac{\delta_0}{2}.\label{3.15}
\end{align}
On the other hand, when $\eta\le y\lesssim 1$, we directly estimate the lower bound of the denominator as
\begin{align}
    c_\star-\left|1+y\right|^a-\left|y\right|^a+y^2+\frac{1}{\left|\rho\right|^2}\ge y^2\ge \eta^2,\label{3.16}
\end{align}
where we used $c_\star-\left|1+y\right|^a-\left|y\right|^a\ge 0$ due to $h_\alpha\ge 0$. Now, combine (\ref{3.15}) and (\ref{3.16}) and we obtain (\ref{3.14}). Finally, using (\ref{3.13}) and (\ref{3.14}), we trivially estimate the first term as 
\begin{align}
    \frac{1}{\left|\rho\right|^{s-2\beta-2}}\int_0^{O(1)}\frac{\left|c_\star-\left|1+y\right|^a-\left|y\right|^a+\left(1+y\right)^2\right|}{\left[c_\star-\left|1+y\right|^a-\left|y\right|^a+y^2+\frac{1}{\left|\rho\right|^2}\right]^{s/2-\beta}} dy\lesssim\frac{1}{\left|\rho\right|^{s-2\beta-2}}\int_0^{O(1)}dy\lesssim\frac{1}{\left|\rho\right|^{s-2\beta-2}},\label{3.17}
\end{align}
where we note that $s>2\beta$.\par
For the second term, since $\displaystyle{0\le\alpha\le\frac{1}{2}}$, the numerator again satisfies
\begin{align}
    c_\star-\left|1-\alpha\right|^a-\left|\alpha\right|^a+\left(1-\alpha\right)^2\lesssim1.\label{3.18}
\end{align}
On the other hand, the denominator satisfies
\begin{align}
    c_\star-\left|1-\alpha\right|^a-\left|\alpha\right|^a+\alpha^2+\frac{1}{\left|\rho\right|^2}\ge c_\star-\left|1-\alpha\right|^a-\left|\alpha\right|^a \ge c_\star-1\gtrsim 1.\label{3.19}
\end{align}
Thus, using (\ref{3.18}) and (\ref{3.19}), we can trivially control the second term as
\begin{align}
    \frac{1}{\left|\rho\right|^{s-2\beta-2}}\int_0^{1/2}\frac{\left|c_\star-\left|1-\alpha\right|^a-\left|\alpha\right|^a+\left(1-\alpha\right)^2\right|}{\left[c_\star-\left|1-\alpha\right|^a-\left|\alpha\right|^a+\alpha^2+\frac{1}{\left|\rho\right|^2}\right]^{s/2-\beta}} d\alpha\lesssim\frac{1}{\left|\rho\right|^{s-2\beta-2}}.\label{3.20}
\end{align}\par
Finally, it follows from (\ref{3.17}) and (\ref{3.20}) that $\displaystyle{I_{21}\lesssim\frac{1}{\left|\rho\right|^{s-2\beta-2}}}$.\par
\noindent\underline{Case 2. $1\le c_\star\le 1+\eps$}\par
In this case, we denote by $\alpha_c$ the left zero of 
\begin{align}
\left|1-\alpha_c\right|^a+\left|\alpha_c\right|^a=c_\star.\label{3.21}
\end{align}
Then we claim that $\left|\alpha_c\right|\lesssim 1$. This follows from (\ref{3.21}), $1\le c_\star\le 1+\eps$. Moreover, we claim $\displaystyle{\alpha_c\approx\frac{1-c_\star}{a}}$, which also implies $c_\star-1>|\alpha_c|$. By Taylor expansion at $\alpha=\alpha_c$, we can rewrite
\begin{align}
c_\star-\left|1-\alpha\right|^a-\left|\alpha\right|^a\approx\alpha-\alpha_c\qquad\qquad\mbox{if }|\alpha-\alpha_c|\ll 1.\label{3.22}
\end{align}
Now, perform the change of variable $\alpha\leftrightarrow y=\alpha-\alpha_c$ and we can rewrite
\begin{align*}
    I_{21}\lesssim&\frac{1}{\left|\rho\right|^{s-2\beta-2}}\int_0^{\alpha_c^\prime}\frac{dy}{\left[y+\left(y-\alpha_c^\prime\right)^2+\frac{1}{\left|\rho\right|^2}\right]^{s/2-\beta}} \\
    &\qquad\qquad\qquad\qquad+\frac{1}{\left|\rho\right|^{s-2\beta-2}}\int_{\alpha_c^\prime}^{O(1)}\frac{dy}{\left[c_\star-\left|1-\alpha\right|^a-\left|\alpha\right|^a+\alpha^2+\frac{1}{\left|\rho\right|^2}\right]^{s/2-\beta}},
\end{align*}
where $\alpha_c^\prime\triangleq -\alpha_c$ and we used that $\left|c_\star-\left|1-\alpha\right|^a-\left|\alpha\right|^a+\left(1-\alpha\right)^2\right|\lesssim 1$ whenever $|\alpha|\lesssim 1$.\par
For the first term, we first observe that
$$y+\left(y-\alpha_c^\prime\right)^2\sim y+\left(c_\star-1\right)^2\qquad\qquad\mbox{if }0\le y\le \alpha_c^\prime.$$
Indeed, when $\displaystyle{0\le y\le \frac{\alpha_c^\prime}{2}}$, we have LHS $\sim \left(y-\alpha_c^\prime\right)^2\sim \left(\alpha_c^\prime\right)^2\sim \left(c_\star-1\right)^2\sim$ RHS. Meanwhile, when $\displaystyle{\frac{\alpha_c^\prime}{2}\le y\le \alpha_c^\prime}$, we have 
$$y\ge\frac{\alpha_c^\prime}{2}\ge\frac{\left(\alpha_c^\prime\right)^2}{2}\ge\frac{\left(c_\star-1\right)^2}{2}>\left(y-(c_\star-1)\right)^2,$$
noticing that $\displaystyle{\alpha_c^\prime\approx\frac{c_\star-1}{a}}$. This allows us to compute the first term as
\begin{align}
    I_{21}\lesssim&\frac{1}{\left|\rho\right|^{s-2\beta-2}}\int_0^{\alpha_c^\prime}\frac{dy}{\left[y+\left(y-\alpha_c^\prime\right)^2+\frac{1}{\left|\rho\right|^2}\right]^{s/2-\beta}}\sim \frac{1}{\left|\rho\right|^{s-2\beta-2}}\int_0^{\alpha_c^\prime}\frac{dy}{\left[y+\left(c_\star-1\right)^2+\frac{1}{\left|\rho\right|^2}\right]^{s/2-\beta}}\notag\\
    \lesssim&\begin{cases}
        \displaystyle{\frac{1}{|\rho|^{s-2\beta-2}}\cdot\frac{1}{\left[\left(c_\star-1\right)^2+\frac{1}{|\rho|^2}\right]^{s/2-\beta-1}}}\qquad&\mbox{, if }s>2\beta+2\vspace{0.5em}\\
            \displaystyle{\frac{1}{|\rho|^{s-2\beta-2}}\cdot\log\left(1+\frac{1}{\left(c_\star-1\right)^2+\frac{1}{|\rho|^2}}\right)}\qquad&\mbox{, if }s=2\beta+2\vspace{0.6em}\\
            \displaystyle{\frac{1}{|\rho|^{s-2\beta-2}}\cdot\left(1+\frac{1}{|\rho|^2}\right)^{\beta+1-s/2}}\qquad&\mbox{, if }s<2\beta+2
    \end{cases}.\label{3.23}
\end{align}\par
For the second term, we pick a small $\delta$ ($0<\delta\ll 1$) such that $c_\star-\left|1-\alpha\right|^a-\left|\alpha\right|^a\ge c_\star-1+\alpha$. Then we split the second term as
\begin{align}
   & \frac{1}{\left|\rho\right|^{s-2\beta-2}}\int_{\alpha_c^\prime}^{O(1)}\frac{dy}{\left[c_\star-\left|1-\alpha\right|^a-\left|\alpha\right|^a+\alpha^2+\frac{1}{\left|\rho\right|^2}\right]^{s/2-\beta}}\notag\\
   \sim& \frac{1}{\left|\rho\right|^{s-2\beta-2}}\int_0^{1/2}\frac{d\alpha}{\left[c_\star-\left|1-\alpha\right|^a-\left|\alpha\right|^a+\alpha^2+\frac{1}{\left|\rho\right|^2}\right]^{s/2-\beta}}\notag\\
   =& \frac{1}{\left|\rho\right|^{s-2\beta-2}}\left(\int_0^{\delta}+\int_\delta^{1/2}\right)\frac{d\alpha}{\left[c_\star-\left|1-\alpha\right|^a-\left|\alpha\right|^a+\alpha^2+\frac{1}{\left|\rho\right|^2}\right]^{s/2-\beta}}\label{I2split}
\end{align}
For the first part, we control it by
\begin{align*}
    \frac{1}{\left|\rho\right|^{s-2\beta-2}}\int_0^{\delta}\frac{d\alpha}{\left[c_\star-1+\alpha+\alpha^2+\frac{1}{\left|\rho\right|^2}\right]^{s/2-\beta}}\lesssim\frac{1}{\left|\rho\right|^{s-2\beta-2}}\int_0^{\delta}\frac{d\alpha}{\left[(c_\star-1)^2+\alpha+\frac{1}{\left|\rho\right|^2}\right]^{s/2-\beta}},
\end{align*}
which will finally be bounded by (\ref{3.23}). For the second part, we can control it trivially by using 
$$\displaystyle{c_\star-\left|1-\alpha\right|^a-\left|\alpha\right|^a+\alpha^2+\frac{1}{\left|\rho\right|^2}\gtrsim \delta^2\gtrsim1},$$
due to the fact that $\delta$ is independent of $c_\star$. Hence this part is bounded by $\displaystyle{\frac{1}{|\rho|^{s-2\beta-2}}}$, which can be absorbed in (\ref{3.23}).\par
\noindent\underline{Case 3. $2^{1-a}\le c_\star\le 1$}\par
This case can be handled in the same way as the previous case. By (\ref{3.22}), we can rewrite
\begin{align*}
    I_{21}&\sim \frac{1}{|\rho|^{s-2\beta-2}}\int_{\alpha_c}^{\alpha_c+\delta}\frac{d\alpha}{\left[\alpha-\alpha_c+|\alpha|^2+\frac{1}{|\rho|^2}\right]^{s/2-\beta}}\\
    &\qquad\qquad\qquad\qquad\qquad+\frac{1}{|\rho|^{s-2\beta-2}}\int_{\alpha_c+\delta}^{\frac{1}{2}}\frac{d\alpha}{\left[c_\star-\left|1-\alpha\right|^a-\left|\alpha\right|^a+\alpha^2+\frac{1}{\left|\rho\right|^2}\right]^{s/2-\beta}}\\
    &=\frac{1}{|\rho|^{s-2\beta-2}}\int_0^{\delta}\frac{dy}{\left[y+|y+\alpha_c|^2+\frac{1}{|\rho|^2}\right]^{s/2-\beta}}\\
    &\qquad\qquad\qquad\qquad\qquad+\frac{1}{|\rho|^{s-2\beta-2}}\int_{\delta}^{\frac{1}{2}-\alpha_c}\frac{dy}{\left[c_\star-\left|1-\alpha\right|^a-\left|\alpha\right|^a+\alpha^2+\frac{1}{\left|\rho\right|^2}\right]^{s/2-\beta}}
\end{align*}
where $0<\delta\ll 1$ is a small positive number independent of $c_\star$.\par
Now, for the first term, we estimate it as
\begin{align*}
    &\frac{1}{|\rho|^{s-2\beta-2}}\int_0^{\delta}\frac{dy}{\left[y+|y+\alpha_c|^2+\frac{1}{|\rho|^2}\right]^{s/2-\beta}}<\frac{1}{|\rho|^{s-2\beta-2}}\int_0^{\delta}\frac{dy}{\left[y+|\alpha_c|^2+\frac{1}{|\rho|^2}\right]^{s/2-\beta}}.
\end{align*}
Then the right side can be controlled by (\ref{3.23}). On the other hand, the second term can be bounded by $\displaystyle{\frac{1}{|\rho|^{s-2\beta-2}}}$ following the same argument as in Step 2, which can again be absorbed in (\ref{3.23}).\par
Combining the estimates from Steps 1--3, we complete the proof.
\end{proof}
We are now ready to estimate $\mathcal J_l$ for $l=1,2,3$. For simplicity, from now on we omit the supremum over $p$ in the formula for $\mathcal J_l$. We start with $\mathcal J_1$. It remains to consider three cases: low-high, high-low, and high-high. The first two cases are straightforward, whereas the last one requires a more delicate argument. We begin with the low-high case.\par
In the low-high case, we assume $|p|\lesssim1$ and $1\ll |p_1|=r$. Noting that $|\rho|\gtrsim 1$, we use the following crude estimate given by Lemma 3.3 and Lemma 3.4:
\begin{align}
I_1+I_2\lesssim\begin{cases}
    1\qquad&\mbox{, if }s>2\beta+2\\
    \log\left|\rho\right|\qquad&\mbox{, if }s=2\beta+2\\
    |\rho|^{2\beta+2-s}\qquad&\mbox{, if }s<2\beta+2
\end{cases}.\label{3.25}
\end{align}
We also note that $|\rho|\le |p|+r\lesssim r$.\par
If $s>2\beta+2$, then we have
$$\mathcal{J}_{1,\mbox{low},\mbox{high}}\lesssim\frac{1}{|p|^{-4\beta}}\int_1^{+\infty}\frac{dr}{r^{s-2\beta-3+a}}\sim |p|^{4\beta}\lesssim 1,$$
where we used $s-2\beta-3+a>1$ due to $s>2\beta+2$.\par
If $s=2\beta+2$, we control $\log|\rho|\lesssim|\rho|^\eps$ for some very small $\eps>0$. Then we have
$$\mathcal{J}_{1,\mbox{low},\mbox{high}}\lesssim\frac{1}{|p|^{-4\beta}}\int_1^{+\infty}\frac{dr}{r^{s-2\beta-3+a-\eps}}=|p|^{4\beta}\int_1^{+\infty}\frac{dr}{r^{a-1-\eps}}\sim |p|^{4\beta}\lesssim 1,$$
where we used $a>2$ in the second last step. Indeed, if $\beta=0$, then we can derive this as follows:
\begin{align}
    s>4\beta+3-\frac{a}{2}\qquad\Rightarrow\qquad 2\beta+2>4\beta+3-\frac{a}{2}\qquad\Rightarrow\qquad a>4\beta+2=2;\label{3.26}
\end{align}
whereas if $\beta<0$, then we have
\begin{align}
    a\ge 4\beta+2>2.\label{3.27}
\end{align}
Together with (\ref{3.26}) and (\ref{3.27}), we conclude that we must have $a>2$ here.\par
If $s<2\beta+2$, then we have
$$\mathcal{J}_{1,\mbox{low},\mbox{high}}\lesssim\frac{1}{|p|^{-4\beta}}\int_1^{+\infty}\frac{r^{2\beta+2-s}}{r^{s-2\beta-3+a}}dr=|p|^{4\beta}\int_1^{+\infty}\frac{dr}{r^{2s-4\beta+a-5}}\sim |p|^{4\beta}\lesssim 1,$$
where we used $2s-4\beta+a-6>0$ in the penultimate step. Indeed, this is due to the fact that
\begin{align}
\underline{2s-4\beta+a-6=0\mbox{ is equivalent to both }\beta=0\mbox{ and }s=4\beta+3-\frac{a}{2}}.\label{3.28}
\end{align}\par
Thus, we complete the proof in the low-high case that $\displaystyle{\mathcal{J}_{1,\mbox{low},\mbox{high}}\lesssim 1}$.\par
Next, we move to the high-low case, so we assume $|p|\gg 1$ and $|p_1|=r\lesssim 1$. Note that $|\rho|\lesssim|p|$. We will still use (\ref{3.25}) to estimate $\displaystyle{\mathcal{J}_{1,\mbox{high},\mbox{low}}}$ in this high-low case.\par
If $s>2\beta+2$, then we have
$$\mathcal{J}_{1,\mbox{high},\mbox{low}}\lesssim\frac{1}{|p|^{a-4\beta-1}}\int_0^1 r^2\,dr\sim\frac{1}{|p|^{a-4\beta-1}}\lesssim 1.$$\par
If $s=2\beta+2$, then we have
$$\mathcal{J}_{1,\mbox{high},\mbox{low}}\lesssim\frac{1}{|p|^{a-4\beta-1}}\log|p|\int_0^1 r^2\,dr\sim\frac{1}{|p|^{a-4\beta-1-\eps}}\lesssim 1,$$
where we used $a\ge 4\beta+2$, which follows from (\ref{3.27}).\par
If $s<2\beta-2$, then we have
$$\mathcal{J}_{1,\mbox{high},\mbox{low}}\lesssim\frac{|p|^{2\beta+2-s}}{|p|^{a-4\beta-1}}\int_0^1 r^2\,dr\sim\frac{1}{|p|^{s-6\beta+a-3}}\lesssim 1.$$\par
Thus, we complete the proof in the high-low case that $\displaystyle{\mathcal{J}_{1,\mbox{high},\mbox{low}}\lesssim 1}$.\par
Finally, we move to the high-high case in which we assume $|p|\gg 1$ and $|p_1|=r\gg 1$. We first consider the contribution of $I_1$.\par
If $r\le |p|$, then $|\rho|\leq |p|+r\lesssim |p|$. On the other hand, since $c_*\lesssim1$, we have $\displaystyle{\frac{|p|^a}{|\rho|^a}\lesssim 1}$ and hence $|p|\lesssim|\rho|$. Therefore, $|p|\sim|\rho|$ and Lemma 3.3 tells us 
\begin{align}
I_1\lesssim\frac{1}{|\rho|^{s-2\beta-2}}\sim\frac{1}{|p|^{s-2\beta-2}}.\label{3.29}
\end{align}
Substituting this estimate into $\mathcal J_1$, we obtain
$$\mathcal{J}_{1,\mbox{high},\mbox{high}}\lesssim\frac{1}{|p|^{a-4\beta-1}}\frac{1}{|p|^{s-2\beta-2}}\int_1^{|p|} \frac{dr}{r^{s-2\beta-2}}\lesssim\begin{cases}
    \displaystyle{\frac{1}{|p|^{s-6\beta+a-3}}}\ \ &\mbox{, if }s>2\beta+3\\
    \displaystyle{\frac{1}{|p|^{s-6\beta+a-3-\eps}}}\ \ &\mbox{, if }s=2\beta+3\\
    \displaystyle{\frac{1}{|p|^{2s-8\beta+a-6}}}\ \ &\mbox{, if }s<2\beta+3\\
\end{cases}\ \lesssim 1,$$
where we absorbed the logarithmic loss as $\displaystyle{\log|p|\lesssim|p|^{\eps}}$ for some $0<\eps\ll 1$.\par
On the other hand, if $r\ge |p|$, then $|\rho|\leq |p|+r\lesssim r$. By the same argument as above, we obtain $r\lesssim |\rho|$, and hence $|\rho|\sim r$. Therefore,
$$I_1\lesssim\frac{1}{|\rho|^{s-2\beta-2}}\sim\frac{1}{r^{s-2\beta-2}}.$$
Using this bound in the definition of $\mathcal J_1$, we obtain
\begin{align}
\mathcal{J}_{1,\mbox{high},\mbox{high}}\lesssim\frac{1}{|p|^{-4\beta}}\int_{|p|}^{+\infty}\frac{dr}{r^{s-2\beta-3}\,r^{s-2\beta-2+a}}=\frac{1}{|p|^{-4\beta}}\int_{|p|}^{+\infty}\frac{dr}{r^{2s-4\beta+a-5}}\lesssim\frac{1}{|p|^{2s-8\beta+a-6}}\lesssim 1,\label{3.30}
\end{align}
where (\ref{3.28}) gives the integrability of the $r$-integral.\par
It remains to estimate the contribution of $I_2$. If $1+\eps\le c_\star\lesssim 1$, then Lemma 3.4 still gives us
$$I_2\lesssim\frac{1}{|\rho|^{s-2\beta-2}}\sim\frac{1}{|p|^{s-2\beta-2}}.$$
Thus, the corresponding contribution can be controlled by repeating the argument following (\ref{3.29}). Now it suffices to only consider the case $2^{1-a}\le c_\star\le 1+\eps$. We now divide the argument into several cases, depending on the size of $s$, the size of $|\rho|$, and the relation between $r$ and $|p|$. From now on, we set $\theta$ be the angle between $p$ and $p_1$.\par
\noindent\underline{Case 1. $s>2\beta+2$, $|\rho|\gg 1$ and $r\le|p|$.}\par
In this case, we must have $|\rho|\sim |p|$. Indeed, $|\rho|\ll |p|$ would imply $|p|\approx r\gg|\rho|$ and thus $c_\star\approx |p|^a/|\rho|^a\gg1$, contradicting $c_\star\lesssim1$. Conversely, $|\rho|\gg |p|$ is impossible since $|\rho|\leq |p|+r\lesssim |p|$. Moreover, it suffices to consider the case $0<c_\star-1\ll1$, since the other case $2^{1-a}<c_\star< 1$ can be treated in the same way using Corollary 6.5 instead of Lemma 6.4.\par
First, if $\displaystyle{0<c_\star-1<\frac{1}{|\rho|}}$, then in view of Lemma 6.4, it's impossible have $\theta\approx 0$. However, if $\theta\approx\pi$, then Lemma 6.4 gives
$$c_\star-1\sim \widetilde{r}=\frac{r}{|p|}\lesssim\frac{1}{|p|}\qquad\Rightarrow\qquad r\lesssim 1,$$
which contradicts the assumption $r\gg 1$. As a result, we must have $\sin\theta\gtrsim 1$ and then by Lemma 6.4 we obtain 
$$\Delta\theta\sim\frac{c_\star-1}{\widetilde{r}}<\frac{1}{|\rho|}\frac{|p|}{r}\sim\frac{1}{r}.$$
Hence the angular integration in the $p_1$-variable gives a gain of order $1/r$. Splitting the $p_1$-integration into radial and angular parts, we obtain
$$\mathcal{J}_{1,\mbox{high},\mbox{high}}\lesssim\frac{1}{|p|^{a-4\beta-1}}\int_1^{|p|}\frac{dr}{r^{s-2\beta-1}}\sim\frac{1}{|p|^{a-4\beta-1}}\lesssim 1,$$
where we used $I_2\lesssim 1$ in this setting.\par
Next, if $\displaystyle{\frac{1}{|\rho|}<c_\star-1\ll\frac{r}{|p|}}$, then by Lemma 3.4 we have
$$I_2\lesssim\frac{1}{|\rho|^{s-2\beta-2}}\frac{1}{(c_\star-1)^{s-2\beta-2}}.$$
By the same argument, we conclude $\sin\theta\gtrsim 1$ as well. We then perform a dyadic decomposition as $c_\star-1\sim 2^t$, where $t$ is such that $\displaystyle{\frac{1}{|\rho|}<2^t\ll \frac{r}{|p|}}$. Therefore, we get 
$$\Delta\theta\sim\frac{c_\star-1}{\widetilde{r}}=\frac{2^t|p|}{r}.$$
As before, we decompose the $p_1$-integration into its radial and angular parts. Then the angular integration of $p_1$-variable including $I_2$ is equal to
\begin{align*}
    &\sum_{\substack{t\\1/|p|<2^t\ll r/|p|}}2^t\cdot\frac{|p|}{r}\cdot\frac{1}{(2^t)^{s-2\beta-2}}\cdot\frac{1}{|p|^{s-2\beta-2}}=\frac{1}{|p|^{s-2\beta-3}}\cdot\frac{1}{r}\sum_{\substack{t\\1/|p|<2^t\ll r/|p|}}\frac{1}{(2^t)^{s-2\beta-3}}\\[6pt]
    \sim&\begin{cases}
        \displaystyle{\frac{1}{r}}\qquad&\mbox{, if }s>2\beta+3\\[8pt]
        \displaystyle{\frac{\log r}{r}}\qquad&\mbox{, if }s=2\beta+3\\[6pt]
        \displaystyle{\frac{1}{r^{s-2\beta-2}}}\qquad&\mbox{, if }2\beta+2<s<2\beta+3\\
    \end{cases}.
\end{align*}
In addition, the radial integration of $p_1$-variable is equal to
\begin{align*}
    \begin{cases}
        \displaystyle{\int_1^{|p|}\frac{dr}{r^{s-2\beta-1}}}\qquad&\mbox{, if }s>2\beta+3\\[10pt]
        \displaystyle{\int_1^{|p|}\frac{\log r}{r^2}\,dr}\qquad&\mbox{, if }s=2\beta+3\\[10pt]
        \displaystyle{\int_1^{|p|}\frac{dr}{r^{2s-4\beta-4}}}\qquad&\mbox{, if }2\beta+2<s<2\beta+3\\
    \end{cases}=\begin{cases}
        1\qquad&\mbox{, if }s>2\beta+2.5\\[6pt]
        \log|p|\qquad&\mbox{, if }s=2\beta+2.5\\[6pt]
        \displaystyle{\frac{1}{|p|^{2s-4\beta-5}}}\qquad&\mbox{, if }2\beta+2<s<2\beta+2.5\\
    \end{cases}.
\end{align*}
Combining the above estimates, we obtain
\begin{align*}
    \mathcal{J}_{1,\mbox{high},\mbox{high}}&\lesssim\frac{1}{|p|^{a-4\beta-1}}\times\begin{cases}
        1\ &\mbox{, if }s>2\beta+2.5\\[6pt]
        \log|p|\ &\mbox{, if }s=2\beta+2.5\\[6pt]
        \displaystyle{\frac{1}{|p|^{2s-4\beta-5}}}\ &\mbox{, if }2\beta+2<s<2\beta+2.5
    \end{cases}\\
    &\lesssim\begin{cases}
        \displaystyle{\frac{1}{|p|^{a-4\beta-1}}}\qquad&\mbox{, if }s>2\beta+2.5\\[8pt]
        \displaystyle{\frac{1}{|p|^{a-4\beta-1-\eps}}}\qquad&\mbox{, if }s=2\beta+2.5\\[8pt]
        \displaystyle{\frac{1}{|p|^{2s-8\beta+a-6}}}\qquad&\mbox{, if }2\beta+2<s<2\beta+2.5
    \end{cases},
\end{align*}
If $s=2\beta+2.5$, then in fact we must have $\displaystyle{s>4\beta+3-\frac{a}{2}}$. Indeed, otherwise we would be in the endpoint assumption $\displaystyle{s=4\beta+3-\frac{a}{2}}$ and $\displaystyle{0<\beta<\frac{a-1}{4}}$. This gives
$$2\beta+2.5=4\beta+3-\frac{a}{2}\qquad\Rightarrow\qquad\beta=\frac{a-1}{4},$$
which contradicts $\displaystyle{0<\beta<\frac{a-1}{4}}$. Therefore $s=2\beta+2.5$ implies $\displaystyle{s>4\beta+3-\frac{a}{2}}$ and consequently $a-4\beta-1>0$. As a result, we finally conclude $\displaystyle{\mathcal{J}_{1,\mbox{high},\mbox{high}}\lesssim 1}$. \par
Finally, if $\displaystyle{\frac{r}{|p|}\lesssim c_\star-1\ll 1}$, then by Lemma 3.4 we have
$$I_2\lesssim \frac{1}{|\rho|^{s-2\beta-2}}\frac{|p|^{s-2\beta-2}}{r^{s-2\beta-2}}\sim\frac{1}{r^{s-2\beta-2}}.$$
It follows that
$$\mathcal{J}_{1,\mbox{high},\mbox{high}}\lesssim\frac{1}{|p|^{a-4\beta-1}}\int_1^{|p|}\frac{dr}{r^{2s-4\beta-4}}\lesssim\begin{cases}
        \displaystyle{\frac{1}{|p|^{a-4\beta-1}}}\qquad&\mbox{, if }s>2\beta+2.5\\[8pt]
        \displaystyle{\frac{1}{|p|^{a-4\beta-1-\eps}}}\qquad&\mbox{, if }s=2\beta+2.5\\[8pt]
        \displaystyle{\frac{1}{|p|^{2s-8\beta+a-6}}}\qquad&\mbox{, if }2\beta+2<s<2\beta+2.5
    \end{cases}\lesssim 1,$$
where the final inequality follows from the same discussion of the powers of $|p|$ as above.

\noindent\underline{Case 2. $s>2\beta+2$, $|\rho|\gg 1$ and $r\ge|p|$.}\par
In this case, the same argument as in Case 1 shows that $|\rho|\gtrsim |p|$. If $|\rho|\gg |p|$, then $r\approx|\rho|\gg |p|$; whereas if $|\rho|\sim |p|$, then $r\sim |p|\sim|\rho|$. Thus, it suffices to consider the case $r\approx|\rho|\gg |p|$, since the other case $r\sim |p|\sim|\rho|$ can be treated in exactly the same way. As before, we assume $0<c_\star-1\ll 1$.\par
First, if $\displaystyle{0<c_\star-1<\frac{1}{|\rho|}}$, then it's impossible to have $\theta\approx 0$ by Corollary 6.5. However, if $\theta\approx\pi$, then Corollary 6.5 gives
$$c_\star-1\sim \overline{r}=\frac{|p|}{r}<\frac{1}{|\rho|}\approx\frac{1}{r}\qquad\Rightarrow\qquad |p|\lesssim 1,$$
which contradicts the assumption $|p|\gg 1$. As a result, we must have $\sin\theta\gtrsim 1$ and then by Corollary 6.5 we obtain 
$$\Delta\theta\sim\frac{c_\star-1}{\overline{r}}<\frac{1}{|\rho|}\frac{r}{|p|}\sim\frac{1}{|p|}.$$
Hence the angular integration in the $p_1$-variable gives a gain of order $1/|p|$. Splitting the $p_1$-integration into radial and angular parts, we obtain
$$\mathcal{J}_{1,\mbox{high},\mbox{high}}\lesssim\frac{1}{|p|^{1-4\beta}}\int_{|p|}^{+\infty}\frac{dr}{r^{s-2\beta-3+a}}\sim\frac{1}{|p|^{s-6\beta+a-3}}\lesssim 1,$$
where in order to ensure the integrability of $r$-integration, we used $s>2\beta+4-a$ as a result of the assumption $s>2\beta+2$.\par
Next, if $\displaystyle{\frac{1}{|\rho|}<c_\star-1\ll \frac{|p|}{r}}$, then by Lemma 3.4 we still have
$$I_2\lesssim\frac{1}{|\rho|^{s-2\beta-2}}\frac{1}{(c_\star-1)^{s-2\beta-2}}.$$
The argument is similar to that of Case 1, so we only indicate the necessary modifications. We first conclude $\sin\theta\gtrsim 1$  and then perform a dyadic decomposition as $c_\star-1\sim 2^t$, where $t$ is such that $\displaystyle{\frac{1}{|\rho|}<2^t\ll \frac{|p|}{r}}$. Therefore, we get 
$$\Delta\theta\sim\frac{c_\star-1}{\overline{r}}\sim\frac{2^t\,r}{|p|}.$$
As before, we decompose the $p_1$-integration into its radial and angular parts. Then the angular integration of $p_1$-variable including $I_2$ is equal to
\begin{align*}
    &\sum_{\substack{t\\1/|p|<2^t\ll |p|/r}}2^t\cdot\frac{r}{|p|}\cdot\frac{1}{(2^t)^{s-2\beta-2}}\cdot\frac{1}{|\rho|^{s-2\beta-2}}
    \sim&\begin{cases}
        \displaystyle{\frac{1}{|p|}}\qquad&\mbox{, if }s>2\beta+3\\[10pt]
        \displaystyle{\frac{\log |p|}{|p|}}\qquad&\mbox{, if }s=2\beta+3\\[10pt]
        \displaystyle{\frac{1}{|p^{s-2\beta-2}}}\qquad&\mbox{, if }2\beta+2<s<2\beta+3\\
    \end{cases},
\end{align*}
where we used $r\approx |\rho|$. In addition, since $s>2\beta+4-a$, the radial integration of $p_1$-variable is equal to
$$\int_{|p|}^{+\infty}\frac{dr}{r^{s-2\beta-3+a}}=\frac{1}{|p|^{s-2\beta+a-4}}.$$
Combining the above estimates, we obtain
\begin{align*}
    \mathcal{J}_{1,\mbox{high},\mbox{high}}\lesssim\begin{cases}
        \displaystyle{\frac{1}{|p|^{-4\beta+1}}\frac{1}{|p|^{s-2\beta+a-4}}=\frac{1}{|p|^{s-6\beta+a-3}}}\qquad&\mbox{, if }s>2\beta+3\\[10pt]
        \displaystyle{\frac{1}{|p|^{-4\beta+1}}\frac{\log|p|}{|p|^{s-2\beta+a-4}}=\frac{1}{|p|^{s-6\beta+a-3-\eps}}}\qquad&\mbox{, if }s=2\beta+3\\[10pt]
        \displaystyle{\frac{1}{|p|^{-4\beta}}\frac{1}{|p|^{s-2\beta-2}}\frac{1}{|p|^{s-2\beta+a-4}}=\frac{1}{|p|^{2s-8\beta+a-6}}}\qquad&\mbox{, if }2\beta+2<s<2\beta+3
    \end{cases}.
\end{align*}
Then the same discussion of the powers of $|p|$ as in Case 1 yields that $\displaystyle{\mathcal{J}_{1,\mbox{high},\mbox{high}}\lesssim 1}$.\par
Finally, if $\displaystyle{\frac{|p|}{r}\lesssim c_\star-1\ll 1}$, then by Lemma 3.4 we have
$$I_2\lesssim \frac{1}{|\rho|^{s-2\beta-2}}\frac{r^{s-2\beta-2}}{|p|^{s-2\beta-2}}\sim\frac{1}{|\rho|^{s-2\beta-2}}.$$
It follows that
$$\mathcal{J}_{1,\mbox{high},\mbox{high}}\lesssim\frac{1}{|p|^{-4\beta}}\frac{1}{|p|^{s-2\beta-2}}\int_{|p|}^{+\infty}\frac{dr}{r^{s-2\beta+a-3}}\lesssim\frac{1}{|p|^{2s-8\beta+a-6}}\lesssim 1,$$
where the final inequality follows from the same discussion of the powers of $|p|$ as above.

\noindent\underline{Case 3. $s>2\beta+2$ and $|\rho|\lesssim 1$.}\par
In this case, we have $I_2\lesssim 1$.\par
If $r\le |p|$, then
$$\mathcal{J}_{1,\mbox{high},\mbox{high}}\lesssim\frac{|\rho|}{|p|^{a-4\beta}}\int_1^{|p|}\frac{dr}{r^{s-2\beta-2}}\lesssim\begin{cases}
    \displaystyle{\frac{1}{|p|^{a-4\beta}}}\qquad&\mbox{, if }s>2\beta+3\\[8pt]
    \displaystyle{\frac{1}{|p|^{a-4\beta-\eps}}}\qquad&\mbox{, if }s=2\beta+3\\[8pt]
    \displaystyle{\frac{1}{|p|^{s-6\beta+a-3}}}\qquad&\mbox{, if }2\beta+2<s<2\beta+3
\end{cases}\lesssim 1;$$
whereas if $r\ge |p|$, then
$$\mathcal{J}_{1,\mbox{high},\mbox{high}}\lesssim\frac{|\rho|}{|p|^{-4\beta}}\int_{|p|}^{+\infty}\frac{dr}{r^{s-2\beta+a-2}}\lesssim\frac{1}{|p|^{s-6\beta+a-3}}\lesssim 1.$$

\noindent\underline{Case 4. $s=2\beta+2$, $|\rho|\gg 1$ and $r\le|p|$.}\par
In this case, we must have $|\rho|\sim |p|$ as before. Again we assume $0<c_\star-1\ll 1$. \par
If $\displaystyle{0<c_\star-1<\frac{1}{|\rho|}}$, then by Lemma 6.4 we necessarily have $\sin\theta\gtrsim 1$ and $\Delta\theta\lesssim 1/r$. Therefore the angular integration in the $p_1$-variable gives a gain of order $1/r$. Note that Lemma 3.4 gives
$$I_2\lesssim\frac{1}{|\rho|^{s-2\beta-2}}\log\left(\frac{1}{\left(c_\star-1\right)^2+\frac{1}{|\rho|^2}}\right)\sim \log |\rho|.$$
Consequently, we obtain
$$\mathcal{J}_{1,\mbox{high},\mbox{high}}\lesssim\frac{1}{|p|^{a-4\beta-1}}\log|p|\int_1^{|p|}\frac{dr}{r}\sim\frac{1}{|p|^{a-4\beta-1-\eps}}\lesssim 1,$$
where we used the discussion of the powers of $|p|$ in the Case 1 above.\par
If $\displaystyle{\frac{1}{|\rho|}<c_\star-1\ll\frac{r}{|p|}}$, then we have $\sin\theta\gtrsim 1$, $\Delta\theta\sim(c_\star-1)/\widetilde{r}\sim \left(2^t|p|\right)/r$ and 
$$I_2\lesssim\log\left(1+\frac{1}{(c_\star-1)^2}\right)\lesssim\frac{1}{(c_\star-1)^\eps},$$
where we performed a dyadic decomposition $c_\star-1\sim 2^t$ as before. Therefore the angular integration of $p_1$-variable including $I_2$ is equal to
\begin{align*}
    \sum_{\substack{t\\1/|p|<2^t\ll r/|p|}}2^t\cdot\frac{|p|}{r}\cdot\frac{1}{(2^t)^\eps}\sim\frac{|p|^{\eps}}{r^\eps}
\end{align*}
and the radial integration of $p_1$-variable is equal to
$$\int_1^{|p|}\frac{dr}{r^{s-2\beta-2+\eps}}\sim|p|^{1-\eps}.$$
Combining the above estimates, we obtain
\begin{align*}
    \mathcal{J}_{1,\mbox{high},\mbox{high}}\lesssim\frac{1}{|p|^{a-4\beta+1}}\,|p|^{1-\eps}\,|p|^{\eps}=\frac{1}{|p|^{a-4\beta-2}}\lesssim 1,
\end{align*}
where we used $a\ge 4\beta+2$ which follows from our assumption $s=2\beta+2$.\par
If $\displaystyle{\frac{r}{|p|}\lesssim c_\star-1\ll 1}$, then we have
$$I_2\lesssim\frac{1}{(c_\star-1)^\eps}\lesssim\frac{|p|^\eps}{r^\eps}.$$
As a result, we conclude
$$\mathcal{J}_{1,\mbox{high},\mbox{high}}\lesssim\frac{|p|^\eps}{|p|^{a-4\beta-1}}\int_1^{|p|}\frac{dr}{r^{\eps}}\lesssim\frac{1}{|p|^{a-4\beta-2}}\lesssim 1.$$

\noindent\underline{Case 5. $s=2\beta+2$, $|\rho|\gg 1$ and $r\ge|p|$.}\par
In this case, we must have $|\rho|\gtrsim |p|$. However, as before, we only consider the case $r\approx |\rho|\gg |p|$ and assume $0<c_\star-1\ll 1$. \par
If $\displaystyle{0<c_\star-1<\frac{1}{|\rho|}}$, then Lemma 6.4 implies that we necessarily have $\sin\theta\gtrsim 1$ and $\Delta\theta\lesssim 1/|p|$. Therefore the angular integration in the $p_1$-variable gives a gain of order $1/r$. Additionally, Lemma 3.4 gives
$$I_2\lesssim\log|\rho|\lesssim|\rho|^\eps.$$
Consequently, we obtain
$$\mathcal{J}_{1,\mbox{high},\mbox{high}}\lesssim\frac{1}{|p|^{-4\beta+1}}\int_{|p|}^{+\infty}\frac{dr}{r^{s-2\beta-3+a-\eps}}=\frac{1}{|p|^{1-4\beta}}\int_{|p|}^{+\infty}\frac{dr}{r^{a-1-\eps}}\sim\frac{1}{|p|^{a-4\beta-1-\eps}}\lesssim 1,$$
where we used $a>2$ to ensure the integrability of $r$-variable. In fact, if $\beta=0$, then we must have
$$s>4\beta+3-\frac{a}{2}\qquad\iff\qquad 2\beta+2>4\beta+3-\frac{a}{2}\qquad\iff\qquad a>4\beta+2>2.$$
On the other hand, if $\beta>0$, then we correspondingly have $a\ge 4\beta+2>2$. All in all, we must have $a>2$ in this case $s=2\beta+2$.\par
If $\displaystyle{\frac{1}{|\rho|}<c_\star-1\ll\frac{|p|}{r}}$, then by performing the dyadic decomposition $c_\star-1\sim 2^t$ as before, we have $\sin\theta\gtrsim 1$, $\Delta\theta\sim(c_\star-1)/\overline{r}\sim \left(2^t r\right)/|p|$ and 
$$I_2\lesssim\frac{1}{(c_\star-1)^\eps}.$$
Therefore the angular integration of $p_1$-variable including $I_2$ is equal to
\begin{align*}
    \sum_{\substack{t\\1/|p|<2^t\ll |p|/r}}2^t\cdot\frac{r}{|p|}\cdot\frac{1}{(2^t)^\eps}\sim\frac{r^\eps}{|p|^{\eps}}
\end{align*}
and the radial integration of $p_1$-variable is equal to
$$\int_{|p|}^{+\infty}\frac{dr}{r^{a-\eps-1}}\sim\frac{1}{|p|^{a-2-\eps}}.$$
Together with the above estimates, we obtain
\begin{align*}
    \mathcal{J}_{1,\mbox{high},\mbox{high}}\lesssim\frac{1}{|p|^{-4\beta}}\frac{1}{|p|^\eps}\frac{1}{|p|^{a-2-\eps}}=\frac{1}{|p|^{a-4\beta-2}}\lesssim 1.
\end{align*}\par
If $\displaystyle{\frac{|p|}{r}\lesssim c_\star-1\ll 1}$, then we have
$$I_2\lesssim\frac{1}{(c_\star-1)^\eps}\lesssim\frac{r^\eps}{|p|^\eps},$$
which leads to 
$$\mathcal{J}_{1,\mbox{high},\mbox{high}}\lesssim\frac{1}{|p|^{-4\beta+\eps}}\int_{|p|}^{+\infty}\frac{dr}{r^{a-\eps-1}}\sim\frac{1}{|p|^{a-4\beta-2}}\lesssim 1.$$

\noindent\underline{Case 6. $s=2\beta+2$ and $|\rho|\lesssim 1$.}\par
In this case, note that $\displaystyle{(c_\star-1)^2+\frac{1}{|\rho|^2}\sim\frac{1}{|\rho|^2}}$. Therefore, by Lemma 3.4, we see that
$$I_2\lesssim\log\left(1+\frac{1}{\left(c_\star-1\right)^2+\frac{1}{|\rho|^2}}\right)\sim \log (1+|\rho|^2)\lesssim 1.$$\par
If $r\le |p|$, then
$$\mathcal{J}_{1,\mbox{high},\mbox{high}}\lesssim\frac{1}{|p|^{a-4\beta}}\int_1^{|p|}dr\sim\frac{1}{|p|^{a-4\beta-1}}\lesssim 1;$$
whereas if $r\ge |p|$, then
$$\mathcal{J}_{1,\mbox{high},\mbox{high}}\lesssim\frac{1}{|p|^{-4\beta}}\int_{|p|}^{+\infty}\frac{dr}{r^a}\sim\frac{1}{|p|^{a-4\beta-1}}\lesssim 1,$$
where we again used $a>2$ to ensure the $r$-integrability which can be argued in the same way as in Case 5.

\noindent\underline{Case 7. $s<2\beta+2$, $|\rho|\gg 1$ and $r\le|p|$.}\par
In this case, we directly apply Lemma 3.4 to get that $I_2\lesssim |\rho|^{2\beta+2-s}$, which leads to
$$\mathcal{J}_{1,\mbox{high},\mbox{high}}\lesssim\frac{1}{|p|^{a-4\beta}}|\rho|^{2\beta+3-s}\int_1^{|p|}\frac{dr}{r^{s-2\beta-2}}\lesssim\frac{1}{|p|^{a-4\beta}}|p|^{2\beta+3-s}\frac{1}{|p|^{s-2\beta-3}}\lesssim\frac{1}{|p|^{a-4\beta}}\lesssim 1.$$

\noindent\underline{Case 8. $s<2\beta+2$, $|\rho|\gg 1$ and $r\ge|p|$.}\par
In this case, Lemma 3.4 also yields $I_2\lesssim |\rho|^{2\beta+2-s}$, which gives
$$\mathcal{J}_{1,\mbox{high},\mbox{high}}\lesssim\frac{1}{|p|^{-4\beta}}|\rho|^{2\beta+3-s}\int_{|p|}^{+\infty}\frac{dr}{r^{s-2\beta+a-2}}\lesssim\frac{1}{|p|^{-4\beta}}\int_{|p|}^{+\infty}\frac{dr}{r^{2s-4\beta+a-5}}\lesssim\frac{1}{|p|^{2s-8\beta+a-6}}\lesssim 1,$$
where we used $2s-4\beta+a-6>0$ to make sure that $r$-variable is integrable. In fact, this follows from the fact that $2s-4\beta+a-6=0$ holds if and only if both $\beta=0$ and $\displaystyle{s=4\beta+3-\frac{a}{2}}$, which contradicts the assumption (H2).

\noindent\underline{Case 9. $s<2\beta+2$ and $|\rho|\lesssim 1$.}\par
In this case, Lemma 3.4 gives $I_2\lesssim 1$.\par
If $r\le |p|$, then
$$\mathcal{J}_{1,\mbox{high},\mbox{high}}\lesssim\frac{|\rho|}{|p|^{a-4\beta}}\int_1^{|p|}\frac{dr}{r^{s-2\beta-2}}\lesssim\frac{1}{|p|^{a-4\beta}}\frac{1}{|p|^{s-2\beta-3}}=\frac{1}{|p|^{s-6\beta+a-3}}\lesssim 1;$$
whereas if $r\ge |p|$, then we only need to repeat the same argument in Case 3 to obtain $\mathcal{J}_{1,\mbox{high},\mbox{high}}\lesssim 1$.\par
Combining (\ref{3.30}) with the estimates from Cases 1--9, we complete the proof of $\mathcal{J}_{1,\mbox{high},\mbox{high}}\lesssim 1$.\par
Now, it remains to consider $\mathcal{J}_2$ and $\mathcal{J}_3$. For $\mathcal{J}_2$, we perform the same parametrization procedure as before and (\ref{para}) becomes
$$\mathcal{J}_2\lesssim\sup_p\int_{\mathbb{R}^3}\left|p\right|^{4\beta}\int_0^{+\infty}\int_0^{2\pi}\frac{\left|z\right|}{\ang{z}^{s-2\beta}\ang{\rho-z}^{s-2\beta}\left|\rho-z\right|^{a-2}\left|\rho\right|}\,d\theta d\left|z\right| dp_1.$$
Now, we have exactly two possibilities: (a) $\ang{p_1}\le\ang{p_3}$ or (b) $\ang{p_1}\le\ang{p_2}$. If $\ang{p_1}\le\ang{p_3}$, then following the same procedure as before, we can rewrite 
\begin{align*}
    \mathcal{J}_2&\lesssim\sup_p\int_{\mathbb{R}^3}\frac{\left|p\right|^{4\beta}}{\ang{p_1}^{s-2\beta}}\int_0^{+\infty}\int_0^{2\pi}\frac{\left|z\right|}{\ang{z}^{s-2\beta}\left|\rho-z\right|^{a-2}\left|\rho\right|}\,d\theta d\left|z\right| dp_1.\\
    &\lesssim\sup_p\left|p\right|^{4\beta}\int_{\mathbb{R}^3}\frac{\left|\rho\right|}{\ang{p_1}^{s-2\beta}\left(\left|p\right|^a+\left|p_1\right|^a\right)}\left(\int_{\substack{\alpha\le 1/2\\h_\alpha\ge 0}}\frac{h_\alpha+\left|\rho\right|^2\left(1-\alpha\right)^2}{\ang{h_\alpha+\left|\rho\right|^2\alpha^2}^{s/2-\beta}}d\alpha+\right.\\
    &\qquad\qquad\qquad\qquad\qquad\qquad\qquad\qquad\qquad\qquad\left.\int_{\substack{\alpha> 1/2\\h_\alpha\ge 0}}\frac{h_\alpha+\left|\rho\right|^2\alpha^2}{\ang{h_\alpha+\left|\rho\right|^2\alpha^2}^{s/2-\beta}}\, d\alpha\right)\, dp_1\\
    &\triangleq\sup_p\left|p\right|^{4\beta}\int_{\mathbb{R}^3}\frac{\left|\rho\right|}{\ang{p_1}^{s-2\beta}\left(\left|p\right|^a+\left|p_1\right|^a\right)}\left(I_A+I_B\right)dp_1.    
\end{align*}
Since $h_\alpha=h_{1-\alpha}$, by the change of variable $\alpha\rightarrow-\alpha$ we have
$$I_B=\int_{\substack{\alpha\le 1/2\\h_\alpha\ge 0}}\frac{h_\alpha+\left|\rho\right|^2(1-\alpha)^2}{\ang{h_\alpha+\left|\rho\right|^2(1-\alpha)^2}^{s/2-\beta}}\, d\alpha \le \int_{\substack{\alpha\le 1/2\\h_\alpha\ge 0}}\frac{h_\alpha+\left|\rho\right|^2(1-\alpha)^2}{\ang{h_\alpha+\left|\rho\right|^2\alpha^2}^{s/2-\beta}}\, d\alpha=I_A,$$
Therefore, this case reduces to the estimate already obtained for $\mathcal J_1$. On the other hand, if $\ang{p_1}\le\ang{p_2}$, then the same parametrization gives 
\begin{align*}
    \mathcal{J}_2&\lesssim\sup_p\int_{\mathbb{R}^3}\frac{\left|p\right|^{4\beta}}{\ang{p_1}^{s-2\beta}}\int_0^{+\infty}\int_0^{2\pi}\frac{\left|z\right|}{\ang{\rho-z}^{s-2\beta}\left|\rho-z\right|^{a-2}\left|\rho\right|}\,d\theta d\left|z\right| dp_1.\\
    &\lesssim\sup_p\left|p\right|^{4\beta}\int_{\mathbb{R}^3}\frac{\left|\rho\right|}{\ang{p_1}^{s-2\beta}\left(\left|p\right|^a+\left|p_1\right|^a\right)}\left(\int_{\substack{\alpha\le 1/2\\h_\alpha\ge 0}}\frac{h_\alpha+\left|\rho\right|^2\left(1-\alpha\right)^2}{\ang{h_\alpha+\left|\rho\right|^2(1-\alpha)^2}^{s/2-\beta}}d\alpha+\right.\\
    &\qquad\qquad\qquad\qquad\qquad\qquad\qquad\qquad\qquad\qquad\left.\int_{\substack{\alpha> 1/2\\h_\alpha\ge 0}}\frac{h_\alpha+\left|\rho\right|^2\alpha^2}{\ang{h_\alpha+\left|\rho\right|^2(1-\alpha)^2}^{s/2-\beta}}\, d\alpha\right)\, dp_1\\
    &\triangleq\sup_p\left|p\right|^{4\beta}\int_{\mathbb{R}^3}\frac{\left|\rho\right|}{\ang{p_1}^{s-2\beta}\left(\left|p\right|^a+\left|p_1\right|^a\right)}\left(\widetilde{I}_A+\widetilde{I}_B\right)dp_1.    
\end{align*}
Using the same trick, it's easy to see that
$$\widetilde{I}_B=\int_{\substack{\alpha\le 1/2\\h_\alpha\ge 0}}\frac{h_\alpha+\left|\rho\right|^2(1-\alpha)^2}{\ang{h_\alpha+\left|\rho\right|^2\alpha^2}^{s/2-\beta}}\, d\alpha=I_A$$
and trivially $\widetilde{I}_A\le I_A$ which again reduces to the $\mathcal{J}_1$ case. Thus, the $\mathcal{J}_2$ case is reduced to the $\mathcal{J}_1$ case.\par
Finally, the term $\mathcal J_3$ has the same structure as $\mathcal J_1$ after the estimate $\langle p\rangle\lesssim \langle p_3\rangle$ is used. Hence, the proof for $\mathcal J_1$ carries over without any change.\par
To summarize, in this subsection we have established the following \begin{enumerate}[label=(\roman*),itemsep=0pt]
    \item $\mathcal{J}_{l,\mbox{low},\mbox{high}}\lesssim 1$ for $l=1,2,3$;\\
    \item $\mathcal{J}_{l,\mbox{high},\mbox{low}}\lesssim 1$ for $l=1,2,3$;\\
    \item If $c_\star\lesssim 1$, then $\mathcal{J}_{l,\mbox{high},\mbox{high}}\lesssim 1$ for $l=1,2,3$.
\end{enumerate}

\subsection{The Non-cancellation Case: $1< a< 2$}
\ \par
In this case, we use the identity
$$\frac{1}{A^{a/2-1}+B^{a/2-1}}\sim \min\left(A,B\right)^{1-a/2},$$
where $A,B>0$ and $1\le a<2$. Therefore, in view of (\ref{3.7}) we have
\begin{align*}
    \mathcal{J}_1&\lesssim\sup_p\int_{\mathbb{R}^3}\frac{\left|p\right|^{4\beta}\left|\rho\right|}{\ang{p_1}^{s-2\beta}}\int\frac{\left[\min\left(h_\alpha+\left|\rho\right|^2\left(1-\alpha\right)^2,h_\alpha+\left|\rho\right|^2\alpha^2\right)\right]^{1-a/2}}{\ang{z}^{s-2\beta}}\,d\alpha dp_1\\
    &\lesssim \sup_p\int_{\mathbb{R}^3}\frac{\left|p\right|^{4\beta}\left|\rho\right|}{\ang{p_1}^{s-2\beta}}\left(\int_{\substack{\alpha\le 1/2\\h_\alpha\ge 0}}\frac{\left[h_\alpha+\left|\rho\right|^2\alpha^2\right]^{1-a/2}}{\ang{z}^{s-2\beta}}d\alpha+\int_{\substack{\alpha> 1/2\\h_\alpha\ge 0}}\frac{\left[h_\alpha+\left|\rho\right|^2(1-\alpha)^2\right]^{1-a/2}}{\ang{z}^{s-2\beta}}\, d\alpha\right)\, dp_1\\
    &\sim\sup_p\int_{\mathbb{R}^3}\frac{\left|p\right|^{4\beta}\left|\rho\right|}{\ang{p_1}^{s-2\beta}}\left(\int_{\substack{\alpha\le 1/2\\h_\alpha\ge 0}}\frac{\left[h_\alpha+\left|\rho\right|^2\alpha^2\right]^{1-a/2}}{\ang{h_\alpha+\left|\rho\right|^2\alpha^2}^{s/2-\beta}}d\alpha+\right.\\
    &\qquad\qquad\qquad\qquad\qquad\qquad\qquad\qquad\qquad\qquad\left.\int_{\substack{\alpha> 1/2\\h_\alpha\ge 0}}\frac{\left[h_\alpha+\left|\rho\right|^2(1-\alpha)^2\right]^{1-a/2}}{\ang{h_\alpha+\left|\rho\right|^2\alpha^2}^{s/2-\beta}}\, d\alpha\right)\, dp_1,
\end{align*}
where again $\left|z\right|=\sqrt{r_\alpha^2+\alpha^2\left|\rho\right|^2}=\sqrt{h_\alpha+\alpha^2\left|\rho\right|^2}$ is used in the last step.\par

If $h_\alpha^{a/2}\gtrsim \left|\rho\right|^a\max\left(\left|1-\alpha\right|^a,\left|\alpha\right|^a\right)$, then as before in view of (\ref{3.7}), we see that
\begin{align*}
    \mathcal{J}_1&\lesssim\sup_p\int_{\mathbb{R}^3}\frac{\left|p\right|^{4\beta}\left|\rho\right|}{\ang{p_1}^{s-2\beta}}\left(\int_{\substack{\alpha\le 1/2\\\left|\alpha\right|\lesssim c_\star^{1/a}}}\frac{\left[\left(\left|p\right|^a+\left|p_1\right|^a\right)^{2/a}+\left|\rho\right|^2\alpha^2\right]^{1-a/2}}{\ang{\left(\left|p\right|^a+\left|p_1\right|^a\right)^{2/a}+\left|\rho\right|^2\alpha^2}^{s/2-\beta}}d\alpha\right.\\
    &\left.\qquad\qquad\qquad\qquad\qquad+\int_{\substack{\alpha> 1/2\\\left|\alpha\right|\lesssim c_\star^{1/a}}}\frac{\left[\left(\left|p\right|^a+\left|p_1\right|^a\right)^{2/a}+\left|\rho\right|^2\left(1-\alpha\right)^2\right]^{1-a/2}}{\ang{\left(\left|p\right|^a+\left|p_1\right|^a\right)^{2/a}+\left|\rho\right|^2\alpha^2}^{s/2-\beta}}\, d\alpha\right)\, dp_1\\
    &\triangleq \sup_p\int_{\mathbb{R}^3}\frac{\left|p\right|^{4\beta}\left|\rho\right|}{\ang{p_1}^{s-2\beta}}\left(I_{11}+I_{12}\right)\, dp_1.
\end{align*}
where we used $h_\alpha\sim \left(\left|p\right|^a+\left|p_1\right|^a\right)^{2/a}$ as before.\par
If $h_\alpha^{a/2}\ll \left|\rho\right|^a\max\left(\left|1-\alpha\right|^a,\left|\alpha\right|^a\right)$, then since $1\le a<2$, we have 
$$\displaystyle{h_\alpha^{(a-2)/2}=\left(h_\alpha^{a/2}\right)^\frac{a-2}{a}\gtrsim \left|\rho\right|^{a-2}\max\left(\left|1-\alpha\right|^{a-2},\left|\alpha\right|^{a-2}\right)}.$$ 
Denote $G(\alpha,h_\alpha)$ as in (\ref{3.10})
and we see that
\begin{align*}
\partial_{h_\alpha}G(\alpha,h_\alpha)&=\frac{a}{2}\left(h_\alpha+\left(1-\alpha\right)^2\left|\rho\right|^2\right)^{\frac{a}{2}-1}+\frac{a}{2}\left(h_\alpha+\alpha^2\left|\rho\right|^2\right)^{\frac{a}{2}-1}\\
&\sim\max\left(h_\alpha^{(a-2)/2},\left|\rho\right|^{a-2}\left|1-\alpha\right|^{a-2},\left|\rho\right|^{a-2}\left|\alpha\right|^{a-2}\right)\sim h_\alpha^{(a-2)/2}.
\end{align*}
Therefore, by mean value theorem, for some $\xi\in(0,h_\alpha)$, we have
\begin{align*}
h_\alpha\sim\frac{G(\alpha,h_\alpha)-G(\alpha,0)}{\partial_{h_\alpha}G(\alpha,\xi)}\sim \frac{C_{p,p_1}-\left|\rho\right|^a\left|1-\alpha\right|^a-\left|\rho\right|^a\left|\alpha\right|^a}{h_\alpha^{(a-2)/2}},
\end{align*}
which implies that
$$h_\alpha \sim \left[C_{p,p_1}-\left|\rho\right|^a\left|1-\alpha\right|^a-\left|\rho\right|^a\left|\alpha\right|^a\right]^{2/a}=\left|\rho\right|^2\left[c_\star-\left|1-\alpha\right|^a-\left|\alpha\right|^a\right]^{2/a}.$$
Thus, in view of (\ref{3.7}), we can write
\begin{align*}
    \mathcal{J}_1&\lesssim\sup_p\int_{\mathbb{R}^3}\frac{\left|p\right|^{4\beta}\left|\rho\right|}{\ang{p_1}^{s-2\beta}}\left(\int_{\substack{\alpha\le 1/2\\h_\alpha\ge 0}}\frac{\left[\left|\rho\right|^2\left[c_\star-\left|1-\alpha\right|^a-\left|\alpha\right|^a\right]^{2/a}+\left|\rho\right|^2\alpha^2\right]^{1-a/2}}{\ang{\left|\rho\right|^2\left[c_\star-\left|1-\alpha\right|^a-\left|\alpha\right|^a\right]^{2/a}+\left|\rho\right|^2\alpha^2}^{s/2-\beta}}d\alpha\right.\\
    &\qquad\qquad\qquad\qquad\qquad\qquad\left.+\int_{\substack{\alpha> 1/2\\h_\alpha\ge 0}}\frac{\left[\left|\rho\right|^2\left[c_\star-\left|1-\alpha\right|^a-\left|\alpha\right|^a\right]^{2/a}+\left|\rho\right|^2(1-\alpha)^2\right]^{1-a/2}}{\ang{\left|\rho\right|^2\left[c_\star-\left|1-\alpha\right|^a-\left|\alpha\right|^a\right]^{2/a}+\left|\rho\right|^2\alpha^2}^{s/2-\beta}}\, d\alpha\right)\, dp_1,\\
    &\triangleq\sup_p\int_{\mathbb{R}^3}\frac{\left|p\right|^{4\beta}\left|\rho\right|}{\ang{p_1}^{s-2\beta}}\left(I_{21}+I_{22}\right)dp_1,
\end{align*}
where 
\begin{align}
    I_{21}&\triangleq\int_{\substack{\alpha\le 1/2\\h_\alpha\ge 0}}\frac{\left[\left|\rho\right|^2\left[c_\star-\left|1-\alpha\right|^a-\left|\alpha\right|^a\right]^{2/a}+\left|\rho\right|^2\alpha^2\right]^{1-a/2}}{\ang{\left|\rho\right|^2\left[c_\star-\left|1-\alpha\right|^a-\left|\alpha\right|^a\right]^{2/a}+\left|\rho\right|^2\alpha^2}^{s/2-\beta}}d\alpha\notag\\
    &=\frac{1}{\left|\rho\right|^{s-2\beta+a-2}}\int_{\substack{\alpha\le 1/2\\h_\alpha\ge 0}}\frac{\left[\left[c_\star-\left|1-\alpha\right|^a-\left|\alpha\right|^a\right]^{2/a}+\alpha^2\right]^{1-a/2}}{\left[\left[c_\star-\left|1-\alpha\right|^a-\left|\alpha\right|^a\right]^{2/a}+\alpha^2+\frac{1}{|\rho|^2}\right]^{s/2-\beta}}d\alpha\label{I212}
\end{align}
and $I_{22}$ can be written analogously. Again as in Chapter 1, We define $I_{21,\mbox{low},\mbox{high}}$, $I_{21,\mbox{high},\mbox{low}}$, $I_{21,\mbox{high},\mbox{high}}$, etc., accordingly.\par
Next two lemmas give the estimate of $I_{1}$ and $I_{2}$. Before proceeding, we again first note that $I_{12}\leq I_{11}$ and $I_{22}\leq I_{21}$.
\begin{lemma}
    Assume $1\le a<2$. Then $\displaystyle{I_1\sim I_{11}\lesssim\frac{1}{\left|\rho\right|^{s-2\beta+a-2}}}$.
\end{lemma}
\begin{proof}
    The proof is exactly the same as Lemma 3.3, so we omit the proof here.
\end{proof}

\begin{lemma}
    Assume $1< a<2$. Pick a small $\eps$ such that $0<\eps\ll 1$. Then we have the following estimates of $I_{2}$:
    \begin{enumerate}[label=(\roman*),itemsep=6pt]
        \item When $1+\eps\le c_\star\lesssim 1$, $\displaystyle{I_2\sim I_{21}\lesssim\frac{1}{|\rho|^{s-2\beta+a-2}}}$;
        \item When $2^{1-a}\le c_\star<1+\eps$, 
        $\displaystyle{I_2\sim I_{21}\lesssim\frac{1}{|\rho|^{s-2\beta+a-2}}\frac{1}{\left[\left(c_\star-1\right)^2+\frac{1}{|\rho|^2}\right]^{s/2-\beta-1}}}$;
        \item When $c_\star<2^{1-a}$, $I_2=I_{21}=0$.
    \end{enumerate}
\end{lemma}
\begin{proof}
The proof is similar as the one of Lemma 3.4. Again we split the proof into three cases according to the size of $c_\star$.\par
\noindent\underline{Case 1. $1+\eps\le c_\star\lesssim1$}\par
In this case, as in Lemma 3.4, we rewrite (\ref{I212}) as
\begin{align*}
    I_{21}=&\frac{1}{\left|\rho\right|^{s-2\beta+a-2}}\int_0^{O(1)}\frac{\left[\left[c_\star-\left|1+y\right|^a-\left|y\right|^a\right]^{2/a}+y^2\right]^{1-a/2}}{\left[\left[c_\star-\left|1+y\right|^a-\left|y\right|^a\right]^{2/a}+y^2+\frac{1}{|\rho|^2}\right]^{s/2-\beta}}dy\\
    &\qquad\qquad\qquad\qquad+\frac{1}{\left|\rho\right|^{s-2\beta+a-2}}\int_0^{1/2}\frac{\left[\left[c_\star-\left|1-\alpha\right|^a-\left|\alpha\right|^a\right]^{2/a}+\alpha^2\right]^{1-a/2}}{\left[\left[c_\star-\left|1-\alpha\right|^a-\left|\alpha\right|^a\right]^{2/a}+\alpha^2+\frac{1}{|\rho|^2}\right]^{s/2-\beta}}d\alpha.
\end{align*}\par
Then, for the first term, we observe that the numerator satisfies
$$\left[c_\star-\left|1+y\right|^a-\left|y\right|^a\right]^{2/a}+y^2\lesssim 1$$
and meanwhile the denominator satisfies
$$\left[c_\star-\left|1+y\right|^a-\left|y\right|^a\right]^{2/a}+y^2+\frac{1}{|\rho|^2}\gtrsim 1$$
by a same argument in Lemma 3.4. On the other hand, the second term can be handled exactly the same as before. Therefore, we obtain $\displaystyle{I_{21}\lesssim\frac{1}{|\rho|^{s-2\beta+a-2}}}$.\par
\noindent\underline{Case 2. $1\le c_\star\le 1+\eps$}\par
In this case, we use the same notation as in Lemma 3.4. As before, we perform the change of variable $\alpha\leftrightarrow y=\alpha-\alpha_c$ and we can rewrite
\begin{align*}
    I_{21}\lesssim&\frac{1}{\left|\rho\right|^{s-2\beta+a-2}}\int_0^{\left(\alpha_c^\prime\right)^a}\frac{dy}{\left[y^{2/a}+\left(y-\alpha_c^\prime\right)^2+\frac{1}{\left|\rho\right|^2}\right]^{s/2-\beta+a/2-1}}\\
    &\qquad\qquad\qquad\qquad+\frac{1}{\left|\rho\right|^{s-2\beta+a-2}}\int_{\left(\alpha_c^\prime\right)^a}^{\alpha_c^\prime}\frac{dy}{\left[y^{2/a}+\left(y-\alpha_c^\prime\right)^2+\frac{1}{\left|\rho\right|^2}\right]^{s/2-\beta+a/2-1}}\\
    &\qquad\qquad\qquad\qquad+\frac{1}{\left|\rho\right|^{s-2\beta+a-2}}\int_{\alpha_c^\prime}^{O(1)}\frac{dy}{\left[y^{2/a}+\left(y-\alpha_c^\prime\right)^2+\frac{1}{\left|\rho\right|^2}\right]^{s/2-\beta+a/2-1}},
\end{align*}
where we used (\ref{3.22}), 
$$\left[y^{2/a}+\left(y-\alpha_c^\prime\right)^2\right]^{1-a/2}\lesssim\frac{1}{\left[y^{2/a}+\left(y-\alpha_c^\prime\right)^2+\frac{1}{|\rho|^2}\right]^{a/2-1}}.$$
and
$$c_\star-|1-\alpha|^a-|\alpha|^a\gtrsim c_\star -1+\alpha.$$
\par
For the first term, since $a>1$, we must have $y\ll \alpha_c^\prime$, which implies that $\left(y-\alpha_c^\prime\right)^2\sim\left(\alpha_c^\prime\right)^2$. Hence, we get
$$y^{2/a}+\left(y-\alpha_c^\prime\right)^2+\frac{1}{|\rho|^2}\gtrsim \left(\alpha_c^\prime\right)^2+\frac{1}{|\rho|^2}.$$
Using this observation, this first term can be estimated as follows:
\begin{align}
    &\frac{1}{\left|\rho\right|^{s-2\beta+a-2}}\int_0^{\left(\alpha_c^\prime\right)^a}\frac{dy}{\left[y^{2/a}+\left(y-\alpha_c^\prime\right)^2+\frac{1}{\left|\rho\right|^2}\right]^{s/2-\beta+a/2-1}}\notag\\
    \lesssim&\frac{1}{\left|\rho\right|^{s-2\beta+a-2}}\frac{\left(\alpha_c^\prime\right)^a}{\left[\left(\alpha_c^\prime\right)^2+\frac{1}{\left|\rho\right|^2}\right]^{s/2-\beta+a/2-1}}\lesssim\frac{1}{\left|\rho\right|^{s-2\beta+a-2}}\frac{\left(\alpha_c^\prime+\frac{1}{|\rho|}\right)^a}{\left[\alpha_c^\prime+\frac{1}{\left|\rho\right|}\right]^{s-2\beta+a-2}}\notag\\
    \sim&\frac{1}{\left|\rho\right|^{s-2\beta+a-2}}\frac{1}{\left[\left(c_\star-1\right)^2+\frac{1}{|\rho|^2}\right]^{s/2-\beta-1}}\label{I213}.
\end{align}\par
For the second term, we use 
$$y^{2/a}+\left(y-\alpha_c^\prime\right)^2+\frac{1}{|\rho|^2}\gtrsim y^{2/a}+\frac{1}{|\rho|^2}$$
to estimate as follows:
\begin{align}
    &\frac{1}{\left|\rho\right|^{s-2\beta+a-2}}\int_{\left(\alpha_c^\prime\right)^a}^{\alpha_c^\prime}\frac{dy}{\left[y^{2/a}+\left(y-\alpha_c^\prime\right)^2+\frac{1}{\left|\rho\right|^2}\right]^{s/2-\beta+a/2-1}}\notag\\
    \lesssim&\frac{1}{\left|\rho\right|^{s-2\beta+a-2}}\int_{\left(\alpha_c^\prime\right)^2}^{+\infty}\frac{\widetilde{y}^{a/2-1}\,d\widetilde{y}}{\left[\widetilde{y}+\frac{1}{\left|\rho\right|^2}\right]^{s/2-\beta+a/2-1}}\notag\\
    =&\frac{1}{\left|\rho\right|^{s-2\beta+a-2}}\left(\int_{\left(\alpha_c^\prime\right)^2}^{\left(\alpha_c^\prime\right)^2+\frac{1}{|\rho|^2}}+\int_{\left(\alpha_c^\prime\right)^2+\frac{1}{|\rho|^2}}^{+\infty}\right)\frac{\widetilde{y}^{a/2-1}\,d\widetilde{y}}{\left[\widetilde{y}+\frac{1}{\left|\rho\right|^2}\right]^{s/2-\beta+a/2-1}}\label{I214}\\
    \lesssim&\frac{1}{\left|\rho\right|^{s-2\beta+a-2}}\frac{1}{\left[\left(\alpha_c^\prime\right)^2+\frac{1}{|\rho|^2}\right]^{s/2-\beta-1}}\int_{\left(\alpha_c^\prime\right)^2}^{\left(\alpha_c^\prime\right)^2+\frac{1}{|\rho|^2}}\frac{d\widetilde{y}}{\widetilde{y}^{1-a/2}}+\frac{1}{\left|\rho\right|^{s-2\beta+a-2}}\int_{\left(\alpha_c^\prime\right)^2+\frac{1}{|\rho|^2}}^{+\infty}\frac{d\widetilde{y}}{\widetilde{y}^{s/2-\beta}}\notag\\
    \lesssim&\frac{1}{\left|\rho\right|^{s-2\beta+a-2}}\frac{1}{\left[\left(c_\star-1\right)^2+\frac{1}{|\rho|^2}\right]^{s/2-\beta-1}},\label{I215}
\end{align}
where we used the change of variable $\widetilde{y}=y^{2/a}$ and notice that $\displaystyle{\frac{a}{2}-1<0}$ and $s>2\beta+2$.\par
For the third term, we similarly split the region into $0\leq \alpha\leq \delta$ and $\delta\leq \alpha\leq 1/2$, as in (\ref{I2split}). Since the latter contribution is harmless, it suffices to consider $0\leq \alpha\leq \delta$. This means $\alpha_c^\prime\le y\le \alpha_c^\prime+\delta$ and we can estimate as follows:
\begin{align}
    &\frac{1}{\left|\rho\right|^{s-2\beta+a-2}}\int_{\alpha_c^\prime}^{\alpha_c^\prime+\delta}\frac{dy}{\left[y^{2/a}+\left(y-\alpha_c^\prime\right)^2+\frac{1}{\left|\rho\right|^2}\right]^{s/2-\beta+a/2-1}}\notag\\
    \lesssim&\frac{1}{\left|\rho\right|^{s-2\beta+a-2}}\int_{\alpha_c^\prime}^{\alpha_c^\prime+\delta}\frac{dy}{\left[y^{2/a}+\frac{1}{\left|\rho\right|^2}\right]^{s/2-\beta+a/2-1}}\lesssim\frac{1}{\left|\rho\right|^{s-2\beta+a-2}}\int_{\left(\alpha_c^\prime\right)^2}^{+\infty}\frac{\widetilde{y}^{a/2-1}\,d\widetilde{y}}{\left[\widetilde{y}+\frac{1}{\left|\rho\right|^2}\right]^{s/2-\beta+a/2-1}},\label{I216}
\end{align}
where we used $\widetilde{y}=y^{2/a}$ and $\left(\alpha_c^\prime\right)^{2/a}\gtrsim \left(\alpha_c^\prime\right)^2$. Then, using the trick in (\ref{I214}), (\ref{I216}) can be bounded by (\ref{I215}).\par
Together with (\ref{I213}), (\ref{I215}) and (\ref{I216}), we finally conclude
$$\displaystyle{I_{21}\lesssim\frac{1}{|\rho|^{s-2\beta+a-2}}\frac{1}{\left[\left(c_\star-1\right)^2+\frac{1}{|\rho|^2}\right]^{s/2-\beta-1}}}.$$\par
\noindent\underline{Case 3. $2^{1-a}\le c_\star\le 1$}\par
In this case, for the same reason as before, it suffices to consider the region $\alpha_c\le \alpha\le\alpha_c+\delta$. Using (\ref{3.22}) and the fact that $\displaystyle{1-\frac{a}{2}>0}$, we compute $I_{21}$ as
\begin{align}
    I_{21}&\sim \frac{1}{\left|\rho\right|^{s-2\beta+a-2}}\int_{\alpha_c}^{\alpha_c+\delta}\frac{d\alpha}{\left[\left(\alpha-\alpha_c\right)^{2/a}+\alpha^2+\frac{1}{|\rho|^2}\right]^{s/2-\beta+a/2-1}}\notag\\
    &=\frac{1}{\left|\rho\right|^{s-2\beta+a-2}}\int_0^\delta\frac{dy}{\left[y^{2/a}+\left(y+\alpha_c\right)^2+\frac{1}{|\rho|^2}\right]^{s/2-\beta+a/2-1}}\notag\\
    &=\frac{1}{\left|\rho\right|^{s-2\beta+a-2}}\int_0^{\delta^{2/a}}\frac{\widetilde{y}^{2/a-1}\,d\widetilde{y}}{\left[\widetilde{y}+\left(\alpha_c\right)^2+\frac{1}{|\rho|^2}\right]^{s/2-\beta+a/2-1}}.\notag
\end{align}
Then splitting $\displaystyle{\int_0^{\delta^{2/a}}\lesssim\int_0^{\left(\alpha_c\right)^2+\frac{1}{|\rho|^2}}+\int_{\left(\alpha_c\right)^2+\frac{1}{|\rho|^2}}^{+\infty}}$ as in (\ref{I214}), we can again obtain
$$\displaystyle{I_{21}\lesssim\frac{1}{|\rho|^{s-2\beta+a-2}}\frac{1}{\left[\left(c_\star-1\right)^2+\frac{1}{|\rho|^2}\right]^{s/2-\beta-1}}}.$$\par
Combining the estimates from Steps 1--3, we complete the proof.   
\end{proof}
\begin{remark}
    The same estimate as in Lemma 3.6 also holds in the endpoint case $a=1$ with $\beta=0$. However, the preceding proof does not apply directly, since (\ref{3.22}) fails when $a=1$. Fortunately, in this endpoint case, $I_{21}$ has a more explicit structure and can be estimated directly. See Lemma 3.8 below.
\end{remark}
We next estimate $\mathcal J_l$ for $l=1,2,3$ in the regime $1<a<2$. As before, we suppress the supremum over $p$ in the expression for $\mathcal J_l$. We start with $\mathcal J_1$. In contrast with the case $2<a<5$, only two configurations need to be considered here: the high-low and high-high cases. The former is relatively straightforward, whereas the latter requires a more delicate argument. We begin with the high-low case.\par
In the high-low case, we assume $|p|\gg1$ and $|p_1|=r\lesssim 1$. Noting that $|\rho|\gtrsim 1$, we use the following crude estimate given by Lemma 3.6:
$$I_2\lesssim\frac{1}{|\rho|^a}.$$
Then, it follows that 
$$\mathcal{J}_{1,\mbox{high},\mbox{low}}\lesssim\frac{1}{|p|^{a-4\beta-1}}\int_0^1 r^2 dr\lesssim\frac{1}{|p|^{a-4\beta-1}},$$
where we used $|\rho|\lesssim |p|$, $a-4\beta-1\ge 0$ and $\ang{r}^{s-2\beta}\sim 1$.\par
Next, we move to the high-high case in which we assume $|p|\gg 1$ and $|p_1|=r\gg 1$. As before, we first consider the contribution of $I_1$.\par
If $r\le |p|$, then we get $|p|\sim|\rho|$ as before and Lemma 3.5 tells us 
\begin{align}
I_1\lesssim\frac{1}{|\rho|^{s-2\beta+a-2}}\sim\frac{1}{|p|^{s-2\beta+a-2}}.\label{3.36}
\end{align}
Note that the radial integration of $p_1$-variable is equal to
\begin{align*}
    \int_1^{|p|}\frac{dr}{r^{s-2\beta-2}}\sim\begin{cases}
        1\qquad&\mbox{, if }s>2\beta+3\\[6pt]
        \log|p|\qquad&\mbox{, if }s=2\beta+3\\[6pt]
        \displaystyle{\frac{1}{|p|^{s-2\beta-3}}}\qquad&\mbox{, if }s<2\beta+3\\
    \end{cases}
\end{align*}
Substituting these estimates into $\mathcal J_1$, we obtain
$$\mathcal{J}_{1,\mbox{high},\mbox{high}}\lesssim\begin{cases}
    \displaystyle{\frac{1}{|p|^{-4\beta-1}}\frac{1}{|p|^{s-2\beta+a-2}}=\frac{1}{|p|^{s-6\beta+a-3}}}\ \ &\mbox{, if }s>2\beta+3\\[8pt]
    \displaystyle{\frac{\log|p|}{|p|^{s-6\beta+a-3}}=\frac{1}{|p|^{a-4\beta-\eps}}}\ \ &\mbox{, if }s=2\beta+3\\[8pt]
    \displaystyle{\frac{1}{|p|^{-4\beta-1}}\frac{1}{|p|^{s-2\beta+a-2}}\frac{1}{|p|^{s-2\beta-3}}=\frac{1}{|p|^{2s-8\beta+a-6}}}\ \ &\mbox{, if }s<2\beta+3\\[8pt]
\end{cases}\ \lesssim 1.$$
\par
On the other hand, if $r\ge |p|$, then $|\rho|\leq |p|+r\lesssim r$. By the same argument as above, we have $|\rho|\sim r$. Therefore,
$$I_1\lesssim\frac{1}{|\rho|^{s-2\beta+a-2}}\sim\frac{1}{r^{s-2\beta+a-2}}.$$
Estimating $\mathcal J_1$ gives exactly the same expression as in (\ref{3.30}), and hence 
\begin{align}
\mathcal{J}_{1,\mbox{high},\mbox{high}}\lesssim 1.\label{3.37}
\end{align}\par
It remains to estimate the contribution of $I_2$. If $1+\eps\le c_\star\lesssim 1$, then Lemma 3.6 still yields
$$I_2\lesssim\frac{1}{|\rho|^{s-2\beta+a-2}}\sim\frac{1}{|p|^{s-2\beta+a-2}}.$$
Thus, the corresponding contribution can be controlled by repeating the argument following (\ref{3.36}). Now it remains to only consider the case $2^{1-a}\le c_\star\le 1+\eps$. Since $|p|,r\gg 1$ and $c_\star\le 1+\eps$, we must have $|\rho|\gg 1$ as well. Moreover. in view of Remark 2.4, we only need to concern the case $|\rho|\gg 1$ and $r\le|p|$. From now on, we set $\theta$ be the angle between $p$ and $p_1$ as before.\par
By a same argument as before, we necessarily have $|\rho|\sim |p|$ again. Without loss of generality, we assume $0<c_\star-1\ll1$ as before.\par
First, if $\displaystyle{0<c_\star-1<\frac{1}{|\rho|}}$, then we must have $\sin\theta\gtrsim 1$ and $\Delta\theta\lesssim 1/r$. Hence the angular integration in the $p_1$-variable gives a gain of order $1/r$. Splitting the $p_1$-integration into radial and angular parts, we obtain
$$\mathcal{J}_{1,\mbox{high},\mbox{high}}\lesssim\frac{1}{|p|^{-4\beta-1}}\frac{1}{|p|^a}\int_1^{|p|}\frac{dr}{r^{s-2\beta-1}}\sim\frac{1}{|p|^{a-4\beta-1}}\lesssim 1,$$
where by Lemma 3.6 we used $I_2\lesssim 1/|\rho|^a$ in this setting.\par
Next, if $\displaystyle{\frac{1}{|\rho|}<c_\star-1\ll\frac{r}{|p|}}$, then we have $\sin\theta\gtrsim 1$, $\Delta\theta\sim(c_\star-1)/\widetilde{r}\sim \left(2^t|p|\right)/r$ and by Lemma 3.6
$$I_2\lesssim\frac{1}{|\rho|^{s-2\beta+a-2}}\frac{1}{(c_\star-1)^{s-2\beta-2}},$$
where we performed a dyadic decomposition $c_\star-1\sim 2^t$ as before. Therefore the angular integration of $p_1$-variable including $I_2$ is equal to
\begin{align}
    \sum_{\substack{t\\1/|p|<2^t\ll r/|p|}}2^t\cdot\frac{|p|}{r}\cdot\frac{1}{(2^t)^{s-2\beta-2}}\cdot\frac{1}{|p|^{s-2\beta+a-2}}\sim\begin{cases}
        \displaystyle{\frac{1}{|p|^a\,r}}\qquad&\mbox{, if }s>2\beta+3\\[8pt]
        \displaystyle{\frac{1}{|p|^{a}}\frac{\log r}{r}}\qquad&\mbox{, if }s=2\beta+3\\[8pt]
        \displaystyle{\frac{1}{|p|^a\,r^{s-2\beta-2}}}\qquad&\mbox{, if }s<2\beta+3
    \end{cases}\label{3.38}
\end{align}
and the radial integration of $p_1$-variable is equal to
\begin{align*}
    \begin{cases}
        \displaystyle{\int_1^{|p|}\frac{dr}{r^{s-2\beta-1}}}\ \ &\mbox{, if }s>2\beta+3\\[8pt]
        \displaystyle{\int_1^{|p|}\frac{\log r}{r^{2}}}\ \ &\mbox{, if }s=2\beta+3\\[8pt]
        \displaystyle{\int_1^{|p|}\frac{dr}{r^{2s-4\beta-4}}}\ \ &\mbox{, if }s<2\beta+3
    \end{cases}\,\lesssim\begin{cases}
        1\ \ &\mbox{, if }s>2\beta+2.5\\[6pt]
        \log |p|\ \ &\mbox{, if }s=2\beta+2.5\\[6pt]
        \displaystyle{\frac{1}{|p|^{2s-4\beta-5}}}\ \ &\mbox{, if }s<2\beta+2.5
    \end{cases}
\end{align*}
Combining the above estimates, we obtain
\begin{align*}
    \mathcal{J}_{1,\mbox{high},\mbox{high}}\lesssim\begin{cases}
    \displaystyle{\frac{1}{|p|^{-4\beta-1}}\frac{1}{|p|^a}=\frac{1}{|p|^{a-4\beta-1}}}\ \ &\mbox{, if }s>2\beta+2.5\\[8pt]
    \displaystyle{\frac{1}{|p|^{-4\beta-1}}\frac{\log |p|}{|p|^a}\lesssim\frac{1}{|p|^{a-4\beta-1-\eps}}}\ \ &\mbox{, if }s=2\beta+2.5\\[8pt]
    \displaystyle{\frac{1}{|p|^{-4\beta-1}}\frac{1}{|p|^a}\frac{1}{|p|^{2s-4\beta-5}}=\frac{1}{|p|^{2s-8\beta+a-6}}}\ \ &\mbox{, if }s<2\beta+2.5
    \end{cases}\,\lesssim 1,
\end{align*}
where in the second line we used $a\ge 4\beta+5$ which follows from our assumption $s=2\beta+2.5$.\par
If $\displaystyle{\frac{r}{|p|}\lesssim c_\star-1\ll 1}$, then we have
$$I_2\lesssim\frac{1}{|\rho|^{s-2\beta+a-2}}\frac{|p|^{s-2\beta-2}}{r^{s-2\beta-2}}\sim\frac{1}{|p|^a\,r^{s-2\beta-2}}.$$
As a result, using the argument following (\ref{3.38}), we still conclude
$$\mathcal{J}_{1,\mbox{high},\mbox{high}}\lesssim 1.$$\par
Now, it remains to consider $\mathcal J_2$ and $\mathcal J_3$. However, the estimates for $\mathcal J_2$ and $\mathcal J_3$ in the regime $1\leq a<2$ are essentially the same as those in the case $2<a<5$. After the same parametrization, the expressions are slightly different, but the same change of variables in $\alpha$ reduces the $\mathcal J_2$ terms to the corresponding $\mathcal J_1$ terms. Moreover, after using $\langle p\rangle\lesssim \langle p_3\rangle$, the term $\mathcal J_3$ has the same structure as $\mathcal J_1$. Thus the proofs carry over verbatim, and we omit the repeated details.\par
Finally, we conclude this subsection by recording the estimates obtained above:
\begin{enumerate}[label=(\roman*),itemsep=0pt]
    \item $\mathcal{J}_{l,\mbox{high},\mbox{low}}\lesssim 1$ for $l=1,2,3$;\\
    \item If $c_\star\lesssim 1$ and $r\le |p|$, then $\mathcal{J}_{l,\mbox{high},\mbox{high}}\lesssim 1$ for $l=1,2,3$.
\end{enumerate}

\subsection{The Non-cancellation Case: $a=1$}
\ \par
In this case, we must have $\beta=0$ and $\displaystyle{s>\frac{5}{2}}$. The corresponding $I_{11},I_{12},I_{21},I_{22}$, etc. are actually the same as those in the case $1<a<2$ in Section 3.3. To be more specific, one can compute
\begin{align*}
    I_2&=\frac{1}{\left|\rho\right|^{s-1}}\left[\int_{(1-c_\star)/2}^0 \frac{\sqrt{\left(c_\star+2\alpha-1\right)^2+\left|\alpha\right|^2}}{\left[\left(c_\star+2\alpha-1\right)^2+\left|\alpha\right|^2+\frac{1}{\left|\rho\right|^2}\right]^{s/2}}d\alpha\right.\\
    &\left.+\int_0^1 \frac{\sqrt{\left(c_\star-1\right)^2+\left|\alpha\right|^2}}{\left[\left(c_\star-1\right)^2+\left|\alpha\right|^2+\frac{1}{\left|\rho\right|^2}\right]^{s/2}}d\alpha+\int^{(c_\star+1)/2}_1 \frac{\sqrt{\left(c_\star-2\alpha+1\right)^2+\left|\alpha\right|^2}}{\left[\left(c_\star-2\alpha+1\right)^2+\left|\alpha\right|^2+\frac{1}{\left|\rho\right|^2}\right]^{s/2}}d\alpha\right]\\
    &=\frac{1}{\left|\rho\right|^{s-1}}\int_{(1-c_\star)/2}^{1/2}(\ldots)\,d\alpha+\frac{1}{\left|\rho\right|^{s-1}}\int_{1/2}^{(c_\star+1)/2}(\ldots)\,d\alpha\triangleq I_{21}+I_{22}.
\end{align*}\par
The next result is parallel to Lemma 3.7.
\begin{lemma}
    In the case of $a=1$, we have
    $$I_2\sim I_{21}\lesssim\begin{cases}
        0\qquad&\mbox{, if }c_\star<1\\
        \displaystyle{\frac{1}{|\rho|}}&\mbox{, if }\displaystyle{0\le c_\star-1<\frac{1}{|\rho|}}\\
        \displaystyle{\frac{1}{|\rho|^{s-1}}\cdot\frac{1}{|c_\star-1|^{s-2}}}\qquad&\mbox{, if }\displaystyle{c_\star-1\ge\frac{1}{|\rho|}}
    \end{cases}.$$
\end{lemma}
\begin{proof}
First of all, it's obvious that if $c_\star <1$ then $I_{21}=0$. Therefore we assume that $c_\star\ge 1$. We write $I_{21}$ as
\begin{align*}
    I_{21}&=\frac{1}{|\rho|^{s-1}}\int_{(1-c_\star)/2}^0 \frac{\sqrt{\left(c_\star+2\alpha-1\right)^2+\left|\alpha\right|^2}}{\left[\left(c_\star+2\alpha-1\right)^2+\left|\alpha\right|^2+\frac{1}{\left|\rho\right|^2}\right]^{s/2}}d\alpha+\frac{1}{|\rho|^{s-1}}\int_0^{1/2} \frac{\sqrt{\left(c_\star-1\right)^2+\left|\alpha\right|^2}}{\left[\left(c_\star-1\right)^2+\left|\alpha\right|^2+\frac{1}{\left|\rho\right|^2}\right]^{s/2}}d\alpha.
\end{align*}\par
For the first term, we let $y=-\alpha$. Then it can be written as
$$\frac{1}{|\rho|^{s-1}}\int_0^{(c_\star-1)/2} \frac{\sqrt{\left(c_\star-1-2y\right)^2+y^2}}{\left[\left(c_\star-1-2y\right)^2+y^2+\frac{1}{\left|\rho\right|^2}\right]^{s/2}}dy.$$
Next, we observe that
$$\left(c_\star-1-2y\right)^2+y^2\sim\left(c_\star-1\right)^2.$$
Indeed, if $\displaystyle{0<y\le \frac{c_\star-1}{3}}$, then LHS  $\sim \left(c_\star-1-2y\right)^2\sim \left(c_\star-1\right)^2\sim $ RHS; while if $\displaystyle{\frac{c_\star-1}{3}\le y\le \frac{c_\star-1}{2}}$, the LHS $\sim y^2\sim \left(c_\star-1\right)^2\sim$ RHS. Hence the claim follows. Thus, the first term is equivalent to 
\begin{align*}
    &\frac{1}{|\rho|^{s-1}}\int_0^{(c_\star-1)/2} \frac{c_\star-1}{\left[\left(c_\star-1\right)^2+\frac{1}{\left|\rho\right|^2}\right]^{s/2}}dy=\frac{1}{|\rho|^{s-1}}\frac{\left(c_\star-1\right)^2}{\left(c_\star-1+\frac{1}{|\rho|}\right)^s}\\
    \sim&\begin{cases}
        \displaystyle{\frac{1}{|\rho|}}&\mbox{, if }0\le c_\star-1<\frac{1}{|\rho|}\\
        \displaystyle{\frac{1}{|\rho|^{s-1}}\cdot\frac{1}{|c_\star-1|^{s-2}}}\qquad&\mbox{, if }\displaystyle{c_\star-1\ge\frac{1}{|\rho|}}
    \end{cases},
\end{align*}
which agrees the statement of the lemma.\par
For the second term, we split it into two regions: $\displaystyle{0\le\alpha\lesssim c_\star-1+\frac{1}{|\rho|}}$ and $\displaystyle{\alpha\gtrsim c_\star-1+\frac{1}{|\rho|}}$. On the first region, the denominator is comparable to
$$\left[(c_\star-1)^2+\frac{1}{|\rho|^2}\right]^{s/2}\sim \left[(c_\star-1)+\frac{1}{|\rho|}\right]^s,$$ 
and the numerator is bounded by
$$\sqrt{(c_\star-1)^2+\frac{1}{|\rho|^2}}\sim(c_\star-1)+\frac{1}{|\rho|}.$$
The length of this integration interval is $\displaystyle{c_\star-1+\frac{1}{|\rho|}}$. Therefore, adding the factor $\displaystyle{\frac{1}{|\rho|^{s-1}}}$ in the beginning, this part contributes at most $\displaystyle{\frac{1}{|\rho|^{s-1}}\cdot\frac{1}{\left[(c_\star-1)+\frac{1}{|\rho|}\right]^{s-2}}}$, which is exactly the result stated in the lemma. On the second region, the denominator is comparable to $\sim|\alpha|^s$ and the numerator is comparable to $\sim\alpha$. Thus, the contribution on this region is about
$$\frac{1}{|\rho|^{s-1}}\int_{c_\star-1+\frac{1}{|\rho|}}^{1/2}\frac{d\alpha}{\alpha^{s-1}}\sim \frac{1}{|\rho|^{s-1}}\frac{1}{\left(c_\star-1+\frac{1}{|\rho|}\right)^{s-2}},$$
which again matches the bound stated in the lemma.   
\end{proof}
We now turn to the estimates for $\mathcal J_l$, $l=1,2,3$. As explained in the last paragraph of Section 3.3, the terms $\mathcal J_2$ and $\mathcal J_3$ can be reduced to the corresponding $\mathcal J_1$ estimate. The same argument applies here, and hence we only estimate $\mathcal J_1$ below. The strategy is similar to the one used in the regime $1<a<2$, but is much simpler in this case. As before, we suppress the supremum over $p$ in the expression for $\mathcal J_l$. Again, it remains to consider the high-low and high-high cases. We begin with the high-low case.\par
In the high-low case, we assume $|p|\gg1$ and $|p_1|=r\lesssim 1$. Noting that $|\rho|\gtrsim 1$, we use the following crude estimate given by Lemma 3.8:
$$I_2\lesssim\frac{1}{|\rho|}.$$
As a result, $$\mathcal{J}_{1,\mbox{high},\mbox{low}}\lesssim|\rho|\frac{1}{|\rho|}\int_0^1 r^2 dr\lesssim 1,$$
where we used $\ang{r}^{s-2\beta}\sim 1$.\par
Next, we move to the high-high case in which we assume $|p|\gg 1$ and $|p_1|=r\gg 1$. The estimates for the $I_1$ contribution in the high-high case and for the $I_2$ contribution in the region $1+\varepsilon\leq c_*\lesssim1$ are identical to those in the case $1<a<2$ in the Section 3.3. Therefore, we omit the details. The remaining contribution comes from the region $1\le c_\star\leq1+\varepsilon$. Hence we may also assume $|\rho|\gg 1$. By Remark 2.4, we further assume $r\le |p|$. Then it's easy to see that one must have $|\rho|\sim |p|$. We now estimate $I_2$ in this regime.\par
First, suppose that $\displaystyle{0<c_\star-1<\frac1{|\rho|}}$. Then Lemma 6.7 implies that the region $\varepsilon\leq\theta\leq\pi$ is impossible. Indeed, otherwise Lemma 6.7 would give 
$$c_\star-1\sim\frac{r}{|p|}\gg\frac{1}{|p|}\sim\frac{1}{|\rho|}$$ 
contradicting $\displaystyle{c_\star-1\sim\frac{1}{|\rho|}}$. Therefore, we must have $0\le\theta\le\eps$ and 
$$(\Delta\theta)^2\sim\frac{(c_\star-1)|p|}{r}\lesssim\frac{1}{r}.$$
Splitting the $p_1$-integration into radial and angular parts, we obtain
$$\mathcal{J}_{1,\mbox{high},\mbox{high}}\lesssim\int_1^{|p|}\frac{dr}{r^{s-1}}\sim\lesssim 1,$$
where by Lemma 3.8 we used $I_2\lesssim 1/|\rho|$.\par
Next, suppose that $\displaystyle{\frac{1}{|\rho|}<c_\star-1\ll\frac{r}{|p|}}$. We apply Lemma 3.8 to get that
$$I_2\lesssim\frac{1}{|\rho|^{s-1}}\frac{1}{(c_\star-1)^{s-2}}.$$
Now we perform the dyadic decomposition $c_\star-1\sim 2^t$ as before. If $\eps\le\theta\le\pi$, then by Lemma 6.7 we have
$$c_\star-1\sim \frac{r}{|p|}\sim 2^t,$$
which leads to $r\sim 2^t|p|$. Thus, we obtain
\begin{align*}
\mathcal{J}_{1,\mbox{high},\mbox{high}}&\lesssim\sum_{\substack{t\\1/|p|<2^t\ll r/|p|}} \frac{1}{|p|^{s-2}}\,\frac{1}{(2^t)^{s-2}}\int_{r\sim 2^t|p|}\frac{dr}{r^{s-2}}\\
&\sim\frac{1}{|p|^{s-2}}\sum_{\substack{t\\1/|p|<2^t\ll r/|p|}}\frac{1}{(2^t)^{s-2}}\frac{1}{(2^t|p|)^{s-3}}\sim\frac{1}{|p|^{2s-5}}\sum_{\substack{t\\1/|p|<2^t\ll r/|p|}}\frac{1}{(2^t|p|)^{2s-5}}\sim 1,
\end{align*}
where we used $2s-5>0$. On the other hand, if $0\le\theta\le\pi$, then Lemma 6.7 gives
$$(\Delta\theta)^2\sim\frac{2^t|p|}{r},$$
which leads to 
\begin{align*}
\mathcal{J}_{1,\mbox{high},\mbox{high}}&\lesssim\int_1^{|p|} \sum_{\substack{t\\1/|p|<2^t\ll r/|p|}} \frac{1}{|p|^{s-2}}\,\frac{2^t|p|}{(2^t)^{s-2}}\frac{dr}{r^{s-1}}=\frac{1}{|p|^{s-3}}\int_1^{|p|}\frac{1}{r^{s-1}}  \sum_{\substack{t\\1/|p|<2^t\ll r/|p|}} \frac{1}{(2^t)^{s-3}}dr\\[10pt]
&\sim\begin{cases}
    \displaystyle{\int_1^{|p|}\frac{dr}{r^{s-1}}}\qquad&\mbox{, if }s>3\\[8pt]
    \displaystyle{\int_1^{|p|}\frac{\log r}{r^2}dr}\qquad&\mbox{, if }s=3\\[8pt]
    \displaystyle{\int_1^{|p|}\frac{dr}{r^{2s-4}}}\qquad&\mbox{, if }2.5<s<3\\
\end{cases}\lesssim 1.
\end{align*}\par
Finally suppose that $\displaystyle{\frac{r}{|p|}\lesssim c_\star-1\ll 1}$, then using Lemma 6.7 we see that
$$I_2\lesssim\frac{1}{|\rho|^{s-1}}\frac{|p|^{s-2}}{r^{s-2}}\sim\frac{1}{|p|\,r^{s-2}}.$$
Therefore, it follows that
$$\mathcal{J}_{1,\mbox{high},\mbox{high}}\lesssim\int_1^{|p|}\frac{dr}{r^{2s-4}}\sim 1.$$\par
Finally, the estimates proved in this subsection can be summarized as follows:
\begin{enumerate}[label=(\roman*),itemsep=0pt]
    \item $\mathcal{J}_{l,\mbox{high},\mbox{low}}\lesssim 1$ for $l=1,2,3$;\\
    \item If $c_\star\lesssim 1$ and $r\le |p|$, then $\mathcal{J}_{l,\mbox{high},\mbox{high}}\lesssim 1$ for $l=1,2,3$.
\end{enumerate}

\subsection{Conclusion}
\ \par
Summarizing all estimates proved above, we have now established the key proposition needed for the proof of local well-posedness.

\begin{prop}
    Under the assumption (H1) and (H2), the operator $\mathcal{T}$ is bounded from $\left(L^{\infty}_s\right)^3$ to $L^{\infty}_s$.
\end{prop}

\section{Proof of Local Wellposedness}
To be more precisely, we first introduce the notion of a strong solution.
\begin{definition}
    Let $T>0$ and $f_0\in L^{\infty}_s$. We say that $f\in C\left([0,T];L^\infty_s\right)$ is a solution of (\ref{1.3}) on $\left[0,T\right]$ with initidal datum $f_0$, if
    \begin{align}
        f(t)=f_0+\int_0^t \mathcal{T}[f](s)\,ds,\qquad t\in\left[0,T\right]
    \end{align}
\end{definition}
Now, we are ready to state and prove our local well-posedness result.
\begin{theorem}
    Assume either (H1) or (H2) holds. The initial value problem (\ref{1.3}) is locally well-posed in $L^\infty_s$. More precisely, for any $R>0$ there exists $T\gtrsim_s R^{-2}$ such that for any initial data $f_0\in L^\infty_s$ with $\left\|f_0\right\|_{L^\infty_s}\le R$, there is a unique strong solution $f$ of the (\ref{1.3}). Furthermore, $\left\|f(t)\right\|_{L^\infty_s}\le 2R$ for any $t\in\left[0,T\right]$ and the map $f_0\mapsto f$ is continuous from $L^\infty_s$ to $C^1([0,T];L^\infty_s)$.
\end{theorem}

\begin{proof}
Let $T:=A_s^{-1}R^{-2}$ for a sufficiently large constant $A_s$. We define the approximating sequence
\begin{align*}
f^0(t):=f_0, 
\qquad 
f^{n+1}(t):=f_0+\int_0^t \mathcal T[f^n(\tau)]\,d\tau ,
\end{align*}
on the interval $[0,T]$. Using Proposition 3.9, it follows easily by induction that
$f^n\in C^1([0,T]:L_s^\infty)$ and
\begin{align*}
\sup_{t\in[0,T]}\|f^n(t)\|_{L_s^\infty}\leq 2R .
\end{align*}
Using Proposition 3.9 again, it follows that the sequence $f^n$ is Cauchy in
$C([0,T]:L_s^\infty)$, thus convergent to a function
$f\in C([0,T]:L_s^\infty)$ satisfying
\begin{align*}
f(0)=f_0,\qquad 
f(t)=f_0+\int_0^t \mathcal T[f(\tau)]\,d\tau,\qquad
\sup_{t\in[0,T]}\|f(t)\|_{L_s^\infty}\leq 2R .
\end{align*}
In particular, $\partial_t f=\mathcal T[f]$, and hence
$f\in C^1([0,T]:L_s^\infty)$. Uniqueness and continuity of the flow map
$f_0\mapsto f$ follow again from the contraction principle.
\end{proof}

\begin{remark}
\ \par
\begin{enumerate}[label=(\roman*),itemsep=6pt]
    \item In the above theorem, the solution actually belongs to $C^1([0,T];L_s^\infty)$.
    \item Following the argument in Chapter 5 of \cite{GermainIonescuTran2020}, one can also use the Euler scheme to prove the positivity of the solution. More precisely, if $s$ is sufficiently large and $f_0\geq 0$, then $f(t)$ is non-negative for any $t\in[0,T]$. We note that this argument alone does not allow us to reach the critical threshold $s$, since the proof in this paper relies on a cancellation effect. However, for larger values of $s$, such a cancellation is no longer necessary. Therefore, the same argument yields the positivity result.
\end{enumerate}
\end{remark}

\section{Proof of Ill-posedness}
In the previous chapter, we proved local well-posedness in the weighted space $L_s^\infty$ under the assumption (H1) and (H2). We now show that this condition is essentially sharp.\par
To prove ill-posedness, the key point is to identify the worst contribution of the nonlinear term. In terms of the variables $p,p_1,p_2,p_3$, the high-low-low-high interaction turns out to be (one of) the most singular configurations. Testing the nonlinear operator on this configuration shows why the boundedness estimate fails below the critical threshold, and leads to the following ill-posedness result.\par
We are now ready to state and prove our ill-posedness result.

\begin{theorem}
    Suppose that we are in one of the following scenarios:\begin{enumerate}[label=(\roman*),itemsep=6pt]
        \item $\displaystyle{\beta>\frac{a-1}{4}}$;
        \item $\displaystyle{s<4\beta+3-\frac{a}{2}}$;
        \item $\displaystyle{s=4\beta+3-\frac{a}{2}}$, $\displaystyle{\beta=\frac{a-1}{4}}$.
    \end{enumerate}
    Then there exists some initial datum $f_0\in L^\infty_s$ such that no strong solution $f\in C([0,T];L^\infty_s)$ with $f(0)=f_0$ can be constructed.
\end{theorem}
Before we proceed to the proof of the ill-posedness result, we make three preliminary comments.\begin{enumerate}[label=\roman*),itemsep=6pt]
    \item When $\beta=0$, it is not possible to define the operators $\mathcal T_j$ for general input functions whose tails decay like $|p|^{-(3-a/2)}$. Therefore, we only consider the other endpoint case 
    $$\displaystyle{s=4\beta+3-\frac{a}{2}}\qquad\mbox{ and }\qquad\displaystyle{\beta=\frac{a-1}{4}}.$$
    \item In the construction below, we choose the initial data so that the output point $p_0$ is not in the support of the initial data. Then, at this output point, the loss terms vanish, and only the gain term $\mathcal{T}_1$ remains. This is the reason why we place the initial mass in several carefully chosen directions: it allows us to separate the gain contribution from the loss contribution.
    \item At the critical threshold, one patch is not enough to produce blow-up. Therefore, we need to put many patches together to create a logarithmic loss. However, this creates a new problem. If the high-frequency parts of two different high-low-low-high patches are too close, then they may interact with each other, and the clean separation above may fail. To avoid this, we choose the output frequencies to be very far apart, for example $|p_0^{(j)}|\sim 2^{2^j}$. With this choice, most patches do not interfere with each other. Even if a few bad interactions remain, the good patches are still enough to produce the required logarithmic divergence.
\end{enumerate}\par
\begin{proof}[Proof of Theorem 5.1]
Recall that
\begin{align*}
\mathcal{T}[f](p)=\left|p\right|^{2\beta}\int \left(\left|p_1\right|^{2\beta}\left|p_2\right|^{2\beta}\left|p_3\right|^{2\beta}\right)\left(f_1 f_2 f_3+f f_2 f_3-f f_1 f_2-f f_1 f_3\right)\\
\ \ \ \ \ \ \ \ \ \ \ \ \ \ \ \times\delta(p+p_1-p_2-p_3) \delta\left(\left|p\right|^a+\left|p_1\right|^a-\left|p_2\right|^a-\left|p_3\right|^a\right)\,dp_1 dp_2 dp_3
\end{align*}\\
\underline{{Step 1. Notations and Preliminary}}\par
Fix $0<\varepsilon\ll 1$ such that $\varepsilon$ is small enough as we wish. Then, we further denote $e_1\triangleq(1,0,0)$, $e_2\triangleq(0,1,0)$ and for any $q\in\mathbb{R}^3$, $q\triangleq(q_{\spar},q_\perp)$, where $q_{\spar}\in\mathbb{R}, q_\perp\in\mathbb{R}^2$. Namely, the first component of $q$ is $q_{\spar}$ and the last two components of $q$ are contained in $q_\perp$. Denote
\begin{enumerate}[label=(\roman*)]
    \item For $i\in\mathbb{Z}^+$, $A_i\triangleq\left\{2^k: k\in\mathbb{N}, 1\le 2^k< \varepsilon 2^{2^i}, \displaystyle{\frac{2^k}{2^{2^m}}\notin\left[\frac{1}{4},4\right]} \mbox{ for }m\gg 1\right\};$
    \item For $i\in\mathbb{Z}^+$, $\displaystyle{E_{\mbox{low},i}\triangleq\bigcup_{2^k\in A_i}\left\{q:q\in\mathbb{R}^3,\left|q\right|\sim 2^k\right\}}$ \\
 and $\displaystyle{E_{\mbox{high},i}\triangleq\bigcup_{2^k\in A_i}\left\{q:q\in\mathbb{R}^3,\left|q-2^{2^i}e_1\right|\sim 2^k, \left|q_\perp\right|\sim 2^k\right\}}$;
    \item $E_i\triangleq E_{\mbox{low},i}\cup E_{\mbox{high},i}$ and $\displaystyle{E=\bigcup_{i=1}^{+\infty} E_i}$.
\end{enumerate}
Define a nonnegative initial datum 
$$\displaystyle{f_0(p)=\langle p\rangle^{-s}\chi_E(p)}.$$ 
From now on, we will focus on a sequence of output variables $p=p^{(j)}\triangleq 2^{2^j}e_1$ for $j\gg 1$. First, we notice that
\begin{align}
f_0(p^{(j)})=0\mbox{ for all }j>j_0\mbox{, where }j_0\gg 1.\label{5.1}
\end{align}
This is because (i) for any $i$, $p^{(j)}\notin E_{\mbox{low},i}$ and (ii) for any $i$, $p^{(j)}\notin E_{\mbox{high},i}$. In terms of (i), if not, then there exists $i_0$ such that $p^{(j)}\in E_{\mbox{low},i_0}$. This means that there exist $k_0\ll 2^{i_0}$ such that $\left|p^{(j)}\right|\sim 2^{k_0}$ and $\displaystyle{\frac{2^k}{2^{2^m}}\notin\left[\frac{1}{4},4\right]}$ for all $m$. However, this contradicts $\left|p^{(j)}\right|\sim 2^{2^j}$. In terms of (ii), if not, then there exists $1\le k_0\ll 2^{i_0}$ such that $\left|p^{(j)}-2^{2^i}e_1\right|\sim 2^{k_0}$, $\left|p^{(j)}_\perp\right|\sim 2^{k_0}$ and $\displaystyle{\frac{2^k}{2^{2^m}}\notin\left[\frac{1}{4},4\right]}$ for all $m$. However, $p^{(j)}_\perp=0$ contradicts $\left|p^{(j)}_\perp\right|\sim 2^{k_0}$. Moreover, from the definition of $f_0$, it is trivial that $\left\|f_0\right\|_{L^\infty_s}\lesssim 1$.\par
Next, we fix $0<c\ll 1$ as a small constant. Denote $z\triangleq p_2-p_1=p-p_3$ and
$$B_k=\left\{(p_1,p_2):\left|p_{1,\spar}\right|\le c\cdot 2^k, \left|p_{1,\perp}-2^k e_2\right|\le c\cdot 2^k, \left|p_{2,\perp}-2\cdot 2^k e_2\right|\le c\cdot 2^k, \left|z_{\spar}\right|\le 3\cdot 2^k\right\}.$$
Therefore, it's easy to see that $\left|z_\perp-2^ke_2\right|\le 2^k$, $\left|p_2\right|\sim 2^k$ and
\begin{align}
B_k\subset\left\{(p_1,p_2):\left|p_1\right|\approx 2^k, \left|p_2\right|\sim 2^k\right\}.\label{5.2}
\end{align}
As before, we write
$$\Phi(p_1,p_2)=\left|p^{(j)}\right|^a+\left|p_1\right|^a-\left|p_2\right|^a-\left|p^{(j)}+p_1-p_2\right|^a.$$
We now perform the change of variables $(p_1,p_2)\longleftrightarrow(p_1, z_{\spar}, z_\perp)$. By abusing the notation, we also write $\Phi=\Phi(p_1,z_{\spar},z_\perp)$. We will prove that
\begin{align}
\mbox{the volume of }B_k \cap \left\{\Phi=0\right\}\sim 2^{5k}.\label{5.3} 
\end{align}
To prove this, it suffices to show that
\begin{align*}
    &B_k \cap \left\{\Phi=0\right\}\\
    =&\left\{(p_1,p_{2,\perp}):\left|p_{1,\spar}\right|\le c\cdot 2^k, \left|p_{1,\perp}-2^k e_2\right|\le c\cdot 2^k, \left|p_{2,\perp}-2\cdot 2^k e_2\right|\le c\cdot 2^k, z_{\spar}=z^\star_{\spar}(p_1,z_\perp)\right\}, 
\end{align*}
where $z_{\spar}=z^\star_{\spar}(p_1,z_\perp)$ is the unique solution of $\Phi(p_1,z_{\spar},z_\perp)=0$ on $B_k$ and $\left|z_{\spar}\right|\le 3\cdot 2^k$. Now we divide into two cases: $a>1$ and $a=1$.\par
First, when $a>1$, we compute
$$\partial_{z_{\spar}}\Phi=-a\left|p_1+z_{\spar}e_1+z_\perp\right|^{a-2}(p_{1,\spar}+z_{\spar})+a\left|2^{2^j}e_1-z_{\spar}e_1-z_\perp\right|^{a-2}(2^{2^j}-z_{\spar}).$$
Since on $B_k$ we have that (i)\\

$\begin{cases}
    \left|p_1+z_{\spar}e_1+z_\perp\right|\sim 2^k\\
    \left|p_{1,\spar}+z_{\spar}\right|\lesssim 2^k
\end{cases}$
$\Longrightarrow \ \ \ \ \ \left|\left|p_1+z_{\spar}e_1+z_\perp\right|^{a-2}(p_{1,\spar}+z_{\spar})\right|\lesssim \left(2^k\right)^{a-1}$\\[4pt]
and (ii)\\[4pt]
$\begin{cases}
    \left|2^{2^j}e_1-z_{\spar}e_1-z_\perp\right|\sim 2^{2^j}\\
    \left|2^{2^j}-z_{\spar}\right|\sim 2^{2^j}
\end{cases}$
$\Longrightarrow \ \ \ \ \ \left|\left|2^{2^j}e_1-z_{\spar}e_1-z_\perp\right|^{a-2}(2^{2^j}-z_{\spar})\right|\sim \left(2^{2^j}\right)^{a-1}, $\\[4pt]
we conclude that 
\begin{align}
    \left|\partial_{z_{\spar}}\Phi\right|\sim \left(2^{2^j}\right)^{a-1}>0 \mbox{ for all }\left|z_{\spar}\right|\le 3\cdot 2^k\label{5.4}
\end{align}
On the other hand, we compute
$$\Phi\big|_{z_{\spar}=0}=\left(2^{2^j}\right)^a+\left|p_1\right|^a-\left|p_1+z_\perp\right|^a-\left|2^{2^j}e_1-z_\perp\right|^a.$$
We notice that (i)
\begin{align*}
    \left|2^{2^j}e_1-z_\perp\right|^a&=\left(2^{2^j}\right)^a \left(1+\frac{\left|z_\perp\right|^2}{\left(2^{2^j}\right)^2}\right)^{a/2}=\left(2^{2^j}\right)^a\left(1+\frac{a}{2}\frac{\left|z_\perp\right|^2}{\left(2^{2^j}\right)^2}+\ldots\right)\\
    &=\left(2^{2^j}\right)^a+O\left(\left(2^{2^j}\right)^{a-2}\cdot 2^{2k_2}\right),
\end{align*}
and (ii)
$$\Bigl|\left|p_1\right|^a-\left|p_1+z_\perp\right|^a\Bigr|\lesssim 2^{ak},$$
which implies that
\begin{align}
    \Bigl|\Phi \big|_{z_{\spar}=0}\Bigr|\lesssim 2^{ak}+\left(2^{2^j}\right)^{a-2}\cdot 2^{2k}=\left(2^{2^j}\right)^{a-1}\left[\frac{2^{ak}}{\left(2^{2^j}\right)^{a-1}}+\frac{2^{2k}}{2^{2^j}}\right].\label{5.5}
\end{align}
Combing (\ref{5.4}) and (\ref{5.5}), we conclude that there exists a unique solution $z^\star_{\spar}\in [-C_1 R,C_1 R]$ such that $\Phi(z^\star_{\spar})=0$. After shrinking $\varepsilon$ in the definition of $f_0$ above if necessary, we have $C_1\cdot R\ll 2^k$, which of course implies that $\left|z_{\spar}\right|\le 3\cdot 2^k$.\par
Second, when $a=1$, we notice that (i)
$$\left|p_1\right|=2^k+O(c\cdot 2^k),$$
(ii)\begin{align*}
    \begin{cases}
    p_{2,\spar}=p_{1,\spar}+z_{\spar}=z_{\spar}+O(c\cdot 2^k)\\
    p_{2,\perp}=2\cdot 2^k+O(c\cdot 2^k)
\end{cases} \Longrightarrow \left|p_2\right|&=\sqrt{\left(z_{\spar}+O(c\cdot 2^k)\right)^2+\left(2\cdot 2^k+O(c\cdot 2^k)\right)^2}\\
&=\sqrt{z_{\spar}^2+4\cdot 2^{2k}}+O(c\cdot 2^k),
\end{align*}
and (iii)
\begin{align*}
    \begin{cases}
    p_{3,\spar}=p^{(j)}_{\spar}-z_{\spar}=2^{2^j}-z_{\spar}\\
    p_{3,\perp}=p^{(j)}_{\perp}-z_{\perp}=-z_\perp
\end{cases} \Longrightarrow \left|p_3\right|&=\sqrt{\left(2^{2^j}-z_{\spar}\right)^2+\left|z_\perp\right|^2}\\
&=\left(2^{2^j}-z_{\spar}\right)\left[1+\frac{\left|z_\perp\right|^2}{\left(2^{2^j}-z_{\spar}\right)^2}+\ldots\right]\\
&=\left(2^{2^j}-z_{\spar}\right)+O\left(2^{2k-2^j}\right).
\end{align*}
This gives us that 
\begin{align*}
    \Phi(z_{\spar})&=2^{2^j}+\left|p_1\right|-\left|p_2\right|-\left|p_3\right|\\
    &=2^{2^j}+2^k-\sqrt{z_{\spar}^2+4\cdot 2^{2k}}-\left(2^{2^j}-z_{\spar}\right)+O\left(c\cdot 2^k\right)+O\left(2^{2k-2^j}\right)\\
    &=z_{\spar}+2^k-\sqrt{z_{\spar}^2+4\cdot 2^{2k}}+O\left(c\cdot 2^k\right)+O\left(2^{2k-2^j}\right).
\end{align*}
Therefore, we get that
$$\Phi(2^k)=2\cdot 2^k-\sqrt{5\cdot 2^{2k}}+O\left(c\cdot 2^k\right)+O\left(2^{2k-2^j}\right)=\left(2-\sqrt{5}\right)2^k+O\left(c\cdot 2^k\right)+O\left(2^{2k-2^j}\right)<0$$
and
$$\Phi(2\cdot2^k)=3\cdot 2^k-\sqrt{8\cdot 2^{2k}}+O\left(c\cdot 2^k\right)+O\left(2^{2k-2^j}\right)=\left(3-\sqrt{8}\right)2^k+O\left(c\cdot 2^k\right)+O\left(2^{2k-2^j}\right)>0,$$
which implies that there exists some $z^\star_{\spar}\in\left[2^k,2\cdot 2^k\right]$ such that $\Phi(z^\star_{\spar})=0$. Furthermore, we see that
\begin{align}
\displaystyle{\partial_{z_{\spar}}\Phi=\frac{2^{2^j}-z_{\spar}}{\left|\left(2^{2^j}-z_{\spar}\right)e_1-z_\perp\right|}-\frac{p_{1,\spar}+z_{\spar}}{\left|p_1+z_{\spar}e_1+z_\perp\right|}=\frac{p_{3,\spar}}{\left|p_3\right|}-\frac{p_{2,\spar}}{\left|p_2\right|}\gtrsim 1}\ \ \ \ \ \ \ \mbox{on }z_{\spar}\in\left[2^k,2\cdot 2^k\right]\label{5.6}
\end{align}
due to $\displaystyle{\frac{p_{3,\spar}}{\left|p_3\right|}\approx 1}$ and $\displaystyle{\left|\frac{p_{2,\spar}}{\bigl|p_2\bigr|}\right|<\frac{3+c}{\sqrt{\left(3-c\right)^2+\left(2-c\right)^2}}<0.85}$. This tells us that the above $z^\star_{\spar}$ is actually the unique solution of $\Phi=0$. These finish the proof of the volume of $B_k \cap \left\{\Phi=0\right\}\sim 2^{5k}$.\par
Finally, we note that for any $0\le k\ll 2^j$ with $2^k\in A_j$, if $(p_1,p_2)\in B_k$, then we have the following two facts: (i) For $l=1,2$, we have $\displaystyle{\left|p_l\right|\sim 2^k \Longrightarrow p_l\in E_{\mbox{low},j} \Longrightarrow f_0(p_l)=\frac{1}{2^{sk}}}$; (ii) $\left|z\right|\le\left|p_1\right|+\left|p_2\right|\lesssim 2^k,\left|z\right|\ge\left|z_l\right|\sim 2^k \Longrightarrow \left|z\right|\sim 2^k \Longrightarrow \left|p_3-2^{2^j}e_1\right|=\left|z\right|\sim 2^k$; also since $\left|z_\perp\right|\sim 2^k$ and $2^k\in A_j$, we get that $\displaystyle{p_3\in E_{\mbox{high},j} \Longrightarrow f_0(p_3)=\frac{1}{\left(2^{2^j}\right)^s}}$. Namely, $\displaystyle{f_0(p^{(j)}+p_1-p_2)=\frac{1}{\left(2^{2^j}\right)^s}}$. Therefore, when $k$ satisfies $0\le k\ll 2^j$ and $2^k\in A_j$, we obtain
\begin{align*}
    I_k\left[f_0\right]\left(p^{(j)}\right)&\triangleq\left|p^{(j)}\right|^{2\beta}\int_{B_k} \left|p_1\right|^{2\beta}\left|p_2\right|^{2\beta}\left|p_3\right|^{2\beta} f_0(p_1) f_0(p_2) f_0(p_3)\\
    &\ \ \ \ \ \ \ \ \ \ \ \ \ \ \ \ \ \ \ \times\delta(p+p_1-p_2-p_3) \delta\left(\left|p\right|^a+\left|p_1\right|^a-\left|p_2\right|^a-\left|p_3\right|^a\right)\,dp_1 dp_2 dp_3\\
    &\sim \left(2^{2^j}\right)^{2\beta}\left(2^k\right)^{2\beta}\left(2^k\right)^{2\beta}\left(2^{2^j}\right)^{2\beta}\frac{1}{2^{sk}}\frac{1}{2^{sk}}\frac{1}{\left(2^{2^j}\right)^s}\int_{B_k}\delta\left(\left|p\right|^a+\left|p_1\right|^a-\left|p_2\right|^a-\left|p_3\right|^a\right)\,dp_1 dp_2 \\
    &\sim \left(2^{2^j}\right)^{4\beta-s}\left(2^k\right)^{4\beta-s}\left(2^k\right)^{4\beta-s}\left(2^{2^j}\right)^{1-a}2^{5k}=\left(2^{2^j}\right)^{4\beta-s+1-a}\cdot \left(2^k\right)^{4\beta-2s+5},
\end{align*}
where the second last step follows from coarea formula and (\ref{5.3}), (\ref{5.4}) and (\ref{5.6}). Recalling (\ref{5.1}) and (\ref{5.2}) and simplifying $\mathcal{T}\left[f_0\right]\left(p^{(j)}\right)$, we obtain
\begin{align}
    \mathcal{T}\left[f_0\right]\left(p^{(j)}\right)&=\left|p^{(j)}\right|^{2\beta}\int_{B_k} \left|p_1\right|^{2\beta}\left|p_2\right|^{2\beta}\left|p_3\right|^{2\beta} f_0(p_1) f_0(p_2) f_0(p_3)\notag\\
    &\ \ \ \ \ \ \ \ \ \ \ \ \ \ \ \ \ \ \ \times\delta(p+p_1-p_2-p_3) \delta\left(\left|p\right|^a+\left|p_1\right|^a-\left|p_2\right|^a-\left|p_3\right|^a\right)\,dp_1 dp_2 dp_3\notag\\
    &\gtrsim \sum_{\substack{k\\0\le k\ll 2^j}}I_k\left[f_0\right]\left(p^{(j)}\right)\ge\sum_{\substack{k\\0\le k\ll 2^j\\2^k\in A_j}}I_k\left[f_0\right]\left(p^{(j)}\right)\label{5.7}
\end{align}
\noindent\underline{{Step 2. Prove that $\mathcal{T}:L^\infty_s\rightarrow L^\infty_s$ is not bounded.}}\par
We now use the function $f_0$ constructed in Step~1 to prove the unboundedness of $\mathcal{T}$. We divide the proof into the three scenarios stated in the theorem.\\
\underline{Case 1.} $\displaystyle{\beta>\frac{a-1}{4}}$\par
In this case, by picking $k=0$ in (\ref{5.7}), we get $\displaystyle{\mathcal{T}\left[f_0\right]\left(p^{(j)}\right)\gtrsim \left(2^{2^j}\right)^{4\beta-s+1-a}}$. Namely, 
$$\displaystyle{\langle p^{(j)}\rangle^s\mathcal{T}\left[f_0\right]\left(p^{(j)}\right)\gtrsim \left(2^{2^j}\right)^{4\beta+1-a}}.$$
Letting $j\rightarrow +\infty$ and recalling $4\beta+1-a>0$, we obtain $\left\|\mathcal{T}\left[f_0\right]\right\|_{L^\infty_s}=+\infty$.\\
\underline{Case 2.} $\displaystyle{s<4\beta+3-\frac{a}{2}}$\par
In this case, we first pick $k_0$ such that $k_0\sim \varepsilon\cdot 2^j$ and $k_0\le \varepsilon\cdot 2^j$. Then we see that
$$\begin{cases}
    \displaystyle{\frac{2^{k_0}}{2^{2^m}}\ll\frac{2^{k_0}}{2^{2^j}}\sim 2^\varepsilon\ll \frac{1}{4}}\mbox{ for all }m>j\\[14pt]
    \displaystyle{\frac{2^{k_0}}{2^{2^m}}\gg\frac{2^{k_0}}{2^{2^j}}\sim 2^\varepsilon\ll 4}\mbox{ for all }m<j
\end{cases}$$
by picking $j_0\gg 1$ large enough and noticing that $0<\varepsilon\ll 1$ is very small. Therefore, $\displaystyle{\frac{2^{k_0}}{2^{2^j}}\notin\left[\frac{1}{4},4\right]}$, which implies that $k_0\in A_j$. Next, by picking $k=k_0$ in (\ref{5.7}), we get
$$\displaystyle{\mathcal{T}\left[f_0\right]\left(p^{(j)}\right)\gtrsim \left(2^{2^j}\right)^{4\beta-s+1-a}\cdot\left(2^{k_0}\right)^{4\beta-2s+5}}\sim \left(2^{2^j}\right)^{8\beta-3s+6-a}.$$
Namely,
$$\displaystyle{\langle p^{(j)}\rangle^s\mathcal{T}\left[f_0\right]\left(p^{(j)}\right)\gtrsim\left(2^{2^j}\right)^{8\beta-2s+6-a}}.$$
Letting $j\rightarrow +\infty$ and recalling $8\beta-2s+6-a>0$, we obtain $\left\|\mathcal{T}\left[f_0\right]\right\|_{L^\infty_s}=+\infty$.\\
\underline{Case 3.} $\displaystyle{s=4\beta+3-\frac{a}{2}}$ and $\displaystyle{4\beta+1-a=0}$\par
This critical case is more delicate, since each individual dyadic patch contributes only $1$, which is not divergent. Therefore, we sum over all such dyadic patches to obtain a logarithmic divergence. \par
First, it's easy to verify that in this case we have $4\beta-2s+5=0$ and $4\beta-s+1-a=-s$. Next, by (\ref{5.7}), we get
\begin{align*}
    \mathcal{T}\left[f_0\right]\left(p^{(j)}\right)&\gtrsim \sum_{\substack{k\\0\le k\ll 2^j\\2^k\in A_j}}\left(2^{2^j}\right)^{4\beta-s+1-a}\cdot \left(2^k\right)^{4\beta-2s+5}\\
    &=\sum_{\substack{k\\0\le k\ll 2^j\\2^k\in A_j}}\left(2^{2^j}\right)^{-s}=\left(2^{2^j}\right)^{-s}\cdot \#A_j\sim \left(2^{2^j}\right)^{-s}\cdot 2^j.
\end{align*}
Namely,
$$\displaystyle{\langle p^{(j)}\rangle^s\mathcal{T}\left[f_0\right]\left(p^{(j)}\right)\gtrsim2^j}.$$
Letting $j\rightarrow +\infty$, we immediately obtain $\left\|\mathcal{T}\left[f_0\right]\right\|_{L^\infty_s}=+\infty$.\par
To sum up, we have proved that $\mathcal{T}:L^\infty_s\rightarrow L^\infty_s$ is not a bounded operator.

\noindent\underline{{Step 3. Prove that no strong solution
$f\in C([0,T];L^\infty_s)$ with initial datum $f_0$ can exist.}}\par
We will prove by contradiction. If not, then there exists $f\in C([0,T];L^\infty_s)$ such that (i) $f$ is a solution of (\ref{1.3}) and (ii) $f(0)=f_0$. This means that there exists $0<\eta\ll 1$ such that for all $t\in\left[0,\eta\right]$, $\left\|f(t)-f(0)\right\|_{L^\infty_s}\le \eta$, which is equivalent to say that for all $t\in\left[0,\eta\right]$ and $p\in\mathbb{R}^3$, $\left|f(t,p)-f_0(p)\right|\le \eta\langle p\rangle^{-s}$. Therefore, for some small constant $c>0$ we have $f_0\left(p^{(j)}\right)=0$ and $\displaystyle{\left|f\left(t,p^{(j)}\right)\right|\le c\eta\frac{1}{\left(2^{2^j}\right)^s}}$ for all $j>j_0\gg 1$. Now, we fix $j$, when $0\le k \ll 2^j$, $2^k\in A_j$ and $(p_1,p_2)\in B_k$, we observe the following:
$$\displaystyle{f_0(p_l)\sim \frac{1}{2^{sk}},\ \ \ \ \ \ \ \left|f_0(p_l)-\frac{1}{2^{sk}}\right|\le c\eta\frac{1}{2^{sk}}\ \ \ \ \ \mbox{ for }l=1,2,}$$
and
$$\displaystyle{f_0(p_3)\sim \frac{1}{\left(2^{2^j}\right)^s},\ \ \ \ \ \ \ \left|f_0(p_3)-\frac{1}{\left(2^{2^j}\right)^s}\right|\le c\eta\frac{1}{\left(2^{2^j}\right)^s}.}$$
Therefore, when $0\le t\le \eta$, thanks to the smallness of $0\le \eta\ll 1$, we have
\begin{align}
    \displaystyle{f(p_1)f(p_2)f(p_3)\ge \left(1-c\eta\right)^3\frac{1}{2^{2ks}\cdot \left(2^{2^j}\right)^s}\approx \frac{1}{2^{2ks}\cdot \left(2^{2^j}\right)^s}}\label{5.8}
\end{align}
and
\begin{align}
    \left|f(p^{(j)})\right|\left[f(p_2)f(p_3)+f(p_1)f(p_2)+f(p_1)f(p_3)\right]&\le \frac{c\eta\left(1+c\eta\right)^2}{\left(2^{2^j}\right)^{2s}\cdot 2^{sk}}+\frac{c\eta\left(1+c\eta\right)^2}{\left(2^{2^j}\right)^{s}\cdot 2^{2sk}}\notag\\
    &\le\frac{2c\eta\left(1+c\eta\right)^2}{\left(2^{2^j}\right)^{s}\cdot 2^{2sk}}\ll \frac{1}{2^{2ks}\cdot \left(2^{2^j}\right)^s},\label{5.9}
\end{align}
which will still give us 
\begin{align}
    &f(p_1)f(p_2)f(p_3)+f(p^{(j)})f(p_2)f(p_3)-f(p^{(j)})f(p_1)f(p_2)-f(p^{(j)})f(p_1)f(p_3)\notag\\
    \ge &f(p_1)f(p_2)f(p_3)-\left|f(p^{(j)})\right|\left[f(p_2)f(p_3)+f(p_1)f(p_2)+f(p_1)f(p_3)\right]\sim \frac{1}{2^{2ks}\cdot \left(2^{2^j}\right)^s}, \label{5.10}
\end{align}
where we abbreviate $f(t,\cdot)$ by $f(\cdot)$ for simplicity. Then, we just need to apply the argument in Step 1 and we can conclude
$$\mathcal{T}\left[f\right]\left(t,p^{(j)}\right)\gtrsim a_j\ \ \ \ \ \ \ \ \ \mbox{for }0\le t\le \eta,$$
where 
$$a_j\triangleq\begin{cases}
    \left(2^{2^j}\right)^{4\beta-s+1-a}\ \ \ \ \ \ \ \mbox{, in Case 1;}\\
    \left(2^{2^j}\right)^{8\beta-3s+6-a}\ \ \ \ \ \ \mbox{, in Case 2;}\\
    \left(2^{2^j}\right)^{-s}\cdot 2^j\ \ \ \ \ \ \ \ \ \ \ \mbox{, in Case 3.}\\
\end{cases}$$
By Duhamel, we have
$$f(t,p^{(j)})=f_0(p^{(j)})+\int^t_0\mathcal{T}\left[f\right]\left(s,p^{(j)}\right)\,ds\ge\int^t_0\mathcal{T}\left[f\right]\left(s,p^{(j)}\right)\,ds\gtrsim t\cdot a_j\ \ \ \ \ \mbox{whenever }0\le t\le \eta,$$
which implies
$$\left\|f(t)\right\|_{L^\infty_s}\ge \langle p^{(j)}\rangle^sf\left(t,p^{(j)}\right)\gtrsim t\cdot a_j\cdot\left(2^{2^j}\right)^s.$$
Letting \(j\to\infty\), we obtain $a_j\left(2^{2^j}\right)^s\to\infty$. Therefore, $\left\|f(t)\right\|_{L^\infty_s}=\infty$ for all $0\le t\le \eta$, which means that no strong solution
$f\in C([0,T];L^\infty_s)$ with this initial datum $f_0$ can exist.
\end{proof}
\begin{remark}
In the work of \cite{AmpatzoglouLeger2025}, the authors proved a special ill-posedness result for the case $a=2$ and $\beta>1/4$. Since their initial data are radial, they cannot directly avoid the interaction between the gain term and the loss term. Instead, they introduce an additional cosine factor into the initial data in order to cancel this interaction. This is also a valid approach, and can be viewed as an alternative method.
\end{remark}

\section{Auxiliary Results}
In this section, we collect several auxiliary lemmas. The first part concerns spherical difference estimates, which will be used in the cancellation argument in the proof of local well-posedness. The second part consists of some elementary geometric estimates for the resonance function. We begin with the first part.

\subsection{Spherical Difference Estimates}
\ \par
Let $Y_{l m}$ be spherical harmonics on $\mathbb{S}^2$. Then
\begin{align*}
    \Delta_{\mathbb{S}^2}Y_{l m}
    =
    -l(l+1)Y_{l m},
\end{align*}
where $l =0,1,2,\dots$ and $m=-l,\dots,l$. Denote by
$\{P_l(x)\}_{l\in\mathbb N}$ the Legendre polynomials and by
$\{P_l^{(\alpha,\beta)}(x)\}_{l\in\mathbb N,\alpha,\beta>-1}$ the Jacobi polynomials.

For $\mu\in[-1,1]$, define the family of operators $T_\mu$ by
\begin{align*}
    (T_\mu f)(\omega)
    \triangleq
    \int_{\mathbb{S}^2} f(\sigma)
    \delta(\sigma\cdot\omega-\mu)\,\dd\sigma,
    \qquad \omega\in\mathbb{S}^2 .
\end{align*}

\begin{lemma}[Funk--Hecke Formula]
For every $\mu\in[-1,1]$, every $l=0,1,2,\dots$, and every
$m=-l,\dots,l$, we have
\begin{align*}
    T_\mu Y_{l m}= \lambda_l(\mu)Y_{l m},
\end{align*}
where $\lambda_l(\mu)=2\pi P_l(\mu)$.
\end{lemma}

\begin{proof}
This follows from the Funk--Hecke formula; see for instance Theorem 6 in
M\"uller, \emph{Spherical Harmonics}, Lecture Notes in Mathematics, vol.~17. \cite{Muller1966}\par
Briefly speaking, this is because $T_\mu$ is a zonal convolution operator on $\mathbb{S}^2$.
\end{proof}
\vspace{1em}

\begin{lemma}[Legendre Polynomial Difference Estimate]
\ \par
Assume $\displaystyle{0<\theta_-,\theta_+,\theta\leq \frac{\pi}{2}}$. Let $\displaystyle{\theta_0\triangleq\frac{\theta_++\theta_-}{2},\ \Delta\theta\triangleq\left|\theta_+-\theta_-\right|}$. \vspace{0.4em} Moreover, assume $\Delta\theta\ll \theta_0$. Then for every
$l=0,1,2,\dots$, we have
\begin{align}
    |P_l(\cos\theta)|
    &\lesssim\min\left(1,\frac{1}{(l\theta)^{1/2}}\right),
    \label{eq:Legendre-pointwise}\\
    |P_l(\cos\theta_+)-P_l(\cos\theta_-)|
    &\lesssim\min\left(l\Delta\theta,1\right)\min\left(1,\frac{1}{(l\theta_0)^{1/2}}\right).
    \label{eq:Legendre-difference}
\end{align}
\end{lemma}

\begin{proof}
Estimate (\ref{eq:Legendre-pointwise}) is the standard pointwise bound for Legendre polynomials. See Theorem 7.3.4 in \cite{Szego1975} for reference.\par
To prove (\ref{eq:Legendre-difference}), We first prove the derivative estimate
\begin{align}
    \left|\frac{\dd}{\dd\theta}P_l(\cos\theta)\right|\lesssim l\min\left(1,\frac{1}{(l\theta)^{1/2}}\right),
    \qquad 0<\theta\leq \frac{\pi}{2}.
    \label{eq:derivative-bound}
\end{align}
Recall the following identity about the derivative of Legendre polynomial
\begin{align*}
    \frac{\dd}{\dd x}P_l(x)=\frac{l+1}{2}P_{l-1}^{(1,1)}(x).
\end{align*}
Hence we get
\begin{align*}
    \frac{\dd}{\dd\theta}P_l(\cos\theta)=-\sin\theta\,P_l'(\cos\theta)=-\frac{l+1}{2}\sin\theta\,P_{l-1}^{(1,1)}(\cos\theta).
\end{align*}

When $l\theta\gtrsim 1$, the Jacobi polynomial estimate (See Theorem 7.32.2 in \cite{Szego1975}) gives
\begin{align*}
    |P_{l-1}^{(1,1)}(\cos\theta)|\lesssim    \theta^{-3/2}l^{-1/2}.
\end{align*}
Therefore we obtain
\begin{align*}
    \left|\frac{\dd}{\dd\theta}P_l(\cos\theta)\right|\lesssim l\sin\theta\,\theta^{-3/2}l^{-1/2}  \lesssim l^{1/2}\theta^{-1/2}  =
    l(l\theta)^{-1/2}=l\min\left(1,\frac{1}{(l\theta)^{1/2}}\right).
\end{align*}

When $l\theta\ll1$, we can estimate it directly. First, we notice that $\left|P_{l-1}^{(1,1)}(x)\right|$ attains its maximum at either $x=1$ of $x=-1$ due to $\displaystyle{\alpha=\beta=1>-\frac{1}{2}}$ (See \cite{Szego1975}). Without loss of generality, we may assume
\begin{align*}
    |P_{l-1}^{(1,1)}(x)|\leq P_{l-1}^{(1,1)}(1)=\binom{l}{l-1}=l,
\end{align*}
which gives that $\displaystyle{\left|P^\prime_l(x)\right|\le \frac{l+1}{2}\cdot l=\frac{l(l+1)}{2}}$. Therefore, we get
$$\left|\frac{\dd}{\dd\theta}P_l(\cos\theta)\right|\lesssim \sin \theta\left|P^\prime_l(\cos\theta)\right|\le \theta\frac{l(l+1)}{2}\lesssim \theta l^2\lesssim l=l\min\left(1,\frac{1}{(l\theta)^{1/2}}\right).$$
Thus \eqref{eq:derivative-bound} follows.

We now prove \eqref{eq:Legendre-difference}. If $l\Delta\theta\leq 1$,
then by the mean value theorem and \eqref{eq:derivative-bound},
\begin{align*}
    |P_l(\cos\theta_+)-P_l(\cos\theta_-)|\lesssim l\Delta\theta\min\left\{1,\frac{1}{(l\theta_0)^{1/2}}\right\},
\end{align*}
where we used $\theta\approx\theta_0$ on the interval
$[\theta_-,\theta_+]$, since $\Delta\theta\ll\theta_0$.

If $l\Delta\theta>1$, then by \eqref{eq:Legendre-pointwise},
\begin{align*}
    |P_l(\cos\theta_+)-P_l(\cos\theta_-)|\leq\left|P_l(\cos\theta_+)\right|+|P_l(\cos\theta_-)|  \lesssim
    \min\left\{1,\frac{1}{(l\theta_0)^{1/2}}\right\}.
\end{align*}
Combining the two cases gives \eqref{eq:Legendre-difference}.
\end{proof}
\vspace{1em}

\begin{lemma}[Spherical Averaging Difference Estimate]
\ \par
Let $\mu_+,\mu_-\in(-1,1)$, and set $\displaystyle{\mu_0\triangleq\frac{\mu_++\mu_-}{2}}$. Assume that $\displaystyle{\frac{|\mu_+-\mu_-|}{1-|\mu_0|}\leq \eps}$, where $0<\eps\ll 1$. For $f\in L^2(\mathbb{S}^2)$, define the difference operator $D$ as
\begin{align*}
    (Df)(\omega)\triangleq\int_{\mathbb{S}^2}f(\sigma)
    \bigl[
        \delta(\sigma\cdot\omega-\mu_+)
        -
        \delta(\sigma\cdot\omega-\mu_-)
    \bigr]
    \dd\sigma .
\end{align*}
Then
\begin{align}
    \|Df\|_{L^2(\mathbb{S}^2)}\lesssim\eps^{1/2}\|f\|_{L^2(\mathbb{S}^2)}.\label{operator norm}
\end{align}
\end{lemma}

\begin{proof}
By Lemma 6.1 (Funk--Hecke formula), we see that $D=T_{\mu_+}-T_{\mu_-}$, and for every spherical harmonic $Y_{l m}$,
\begin{align*}
    DY_{l m}=2\pi \bigl(P_l(\mu_+)-P_l(\mu_-)\bigr)Y_{l m}.
\end{align*}
Therefore, by Plancherel on $L^2(\mathbb{S}^2)$,
\begin{align*}
    \|D\|_{L^2(\mathbb{S}^2)\to L^2(\mathbb{S}^2)}=2\pi\sup_{l\geq 0}|P_l(\mu_+)-P_l(\mu_-)|.
\end{align*}

Write $\mu_\pm=\cos\theta_\pm, \ \theta_\pm\in(0,\pi)$, and set $\displaystyle{\theta_0\triangleq\frac{\theta_++\theta_-}{2}}$. We claim first that
\begin{align}
    1-|\mu_0|\sim\min\left(\theta_0^2,(\pi-\theta_0)^2\right).
    \label{eq:mu0-theta0}
\end{align}
It suffices to prove this for $0<\theta_0\leq \pi/2$; the case $\pi/2<\theta_0<\pi$ follows by replacing $\theta_\pm$ with $\pi-\theta_\pm$ and noticing that $\theta_\pm<\pi$. Put $\displaystyle{\delta\triangleq\frac{\theta_+-\theta_-}{2}}$. Then $\theta_\pm=\theta_0\pm\delta$. Since $\theta_\pm>0$, we have
$|\delta|<\theta_0$. Moreover,
$$1-\left|\mu_0\right|=1-\frac{\cos\theta_+ + \cos\theta_-}{2}=1-\cos\theta_0\cdot\cos\delta\ge 1-\cos\theta_0\sim\theta_0^2.$$
On the other hand, 
$$1-\left|\mu_0\right| =(1-\cos\theta_0)+\cos\theta_0(1-\cos\delta)\le1-\cos\theta_0+1-\cos\delta\sim \theta_0^2+\delta^2\sim\theta_0^2$$
due to $|\delta|<\theta_0$. This proves \eqref{eq:mu0-theta0}.

Next, again it suffices to prove (\ref{operator norm}) for $0<\theta_0\leq\pi/2$; the case $\pi/2<\theta_0<\pi$ follows by performing the change of variables: $\widetilde\theta_\pm\triangleq\pi-\theta_\pm, \ \widetilde\theta_0\triangleq\pi-\theta_0$, and noticing that 
$$P_l(\cos\theta_\pm)=P_l(-\cos\widetilde\theta_\pm)=(-1)^l P_l(\cos\widetilde\theta_\pm).$$
Now, we have $|\mu_+-\mu_-|=|\cos\theta_+-\cos\theta_-|\sim\theta_0|\theta_+-\theta_-|$ and $\left|\mu_+-\mu_-\right|\le\varepsilon\left(1-\mu_0\right)\sim\eps\theta_0^2$, which implies $\displaystyle{\frac{\Delta\theta}{\theta_0}\lesssim\eps\ll 1}$ .

Now we apply Lemma 6.2 to estimate $\left|P_l(\cos\theta_+)-P_l(\cos\theta_-)\right|$. First, when $l\Delta\theta\leq 1$, then (\ref{eq:Legendre-difference}) gives
\begin{align*}
    |P_l(\cos\theta_+)-P_l(\cos\theta_-)|\lesssim l\Delta\theta\min\left\{1,\frac{1}{(l\theta_0)^{1/2}}\right\}.
\end{align*}
If $l\theta_0\leq 1$, then
\begin{align*}
    |P_l(\cos\theta_+)-P_l(\cos\theta_-)|\lesssim l\Delta\theta\leq\frac{\Delta\theta}{\theta_0}\lesssim \left(\frac{\Delta\theta}{\theta_0}\right)^{1/2}.
\end{align*}
If $l\theta_0>1$, then
\begin{align*}
    |P_l(\cos\theta_+)-P_l(\cos\theta_-)|\lesssim \frac{l\Delta\theta}{(l\theta_0)^{1/2}}=\frac{l^{1/2}\Delta\theta}{\theta_0^{1/2}} \leq\left(\frac{\Delta\theta}{\theta_0}\right)^{1/2}\lesssim\eps^{1/2},
\end{align*}
where we used $l\Delta\theta\leq 1$. Next, when $l\Delta\theta>1$, then using (\ref{eq:Legendre-pointwise}) we obtain
\begin{align*}
    |P_l(\cos\theta_+)-P_l(\cos\theta_-)|\lesssim \left|P_l(\cos\theta_+)\right|+\left|P_l(\cos\theta_-)\right|\sim\left|P_l(\cos\theta_0)\right|\lesssim
    \frac{1}{(l\theta_0)^{1/2}}\le\left(\frac{\Delta\theta}{\theta_0}\right)^{1/2},
\end{align*}
where in the last inequality we used $l\Delta\theta>1$. Thus in all cases we obtain
\begin{align*}
    |P_l(\mu_+)-P_l(\mu_-)|\lesssim\left(\frac{\Delta\theta}{\theta_0}\right)^{1/2}\lesssim\eps^{1/2},
\end{align*}
which implies that
$\left\|D\right\|_{L^2(\mathbb{S}^2)\to L^2(\mathbb{S}^2)}\lesssim\eps^{1/2}$ and hence (\ref{operator norm}) follows.
\end{proof}

\subsection{Geometric Properties of the Resonance Function}
\ \par
We now turn to the second part, which concerns elementary geometric properties of the resonance function. Recall from (\ref{3.8}) that
$$c_\star=\frac{|p|^a+|p_1|^a}{|p+p_1|^a}=\frac{|p|^a+r^a}{|\rho|^a}.$$\par
\begin{lemma}
    Assume $1<a\le 5$, $0<c_\star-1\ll 1$ and $\displaystyle{0<\widetilde{r}=\frac{r}{|p|}<1}$. Denote by $\theta$ the angle between $p$ and $p_1$. Then the following statements hold:\begin{enumerate}[label=(\roman*),itemsep=6pt]
        \item The case $\theta\approx0$ is impossible;
        \item If $\theta\approx\pi$, then $c_\star-1\sim \widetilde{r}$;
        \item If $\sin\theta\gtrsim 1$ and $c_\star-1\lesssim\widetilde{r}$, then we have $\displaystyle{\Delta\theta\sim\frac{c_\star-1}{\widetilde{r}}}$. In other words, the length of the admissible $\theta$-interval is of order $\displaystyle{\frac{c_\star-1}{\widetilde{r}}}$.
    \end{enumerate}
\end{lemma}
\begin{proof}
(i) If $\theta\approx 0$, then 
$$c_\star=\frac{|p|^a+|p_1|^a}{|p+p_1|^a}\approx\frac{1+\widetilde{r}^a}{\left(1+\widetilde{r}\right)^a}>1,$$
which contradicts $0<c_\star-1\ll 1$.\par
(ii) If $\theta\approx\pi$, then we have $\displaystyle{c_\star\approx\frac{1+\widetilde{r}^a}{\left(1-\widetilde{r}\right)^a}}$. Since $0<c_\star-1\ll 1$, we must have $0\le\widetilde{r}\ll 1$. Now we rewrite 
$$c_\star=\frac{|p|^a+|p_1|^a}{|p+p_1|^a}=\frac{1+\widetilde{r}^a}{\left(1+\widetilde{r}^2+2\widetilde{r}\cos\theta\right)^{a/2}}.$$
By performing Taylor expansion, we see that
\begin{align*}
    \left(1+\widetilde{r}^2+2\widetilde{r}\cos\theta\right)^{-a/2}&=1-\frac{a}{2}\left(\widetilde{r}^2+2\widetilde{r}\cos\theta\right)+\frac{a(a+2)}{2}\left(\widetilde{r}^2+2\widetilde{r}\cos\theta\right)^2+O\left(\widetilde{r}^3\right)\\
    &=1-a\widetilde{r}\cos\theta-\frac{a}{2}\widetilde{r}^2+\frac{a(a+2)}{2}\widetilde{r}^2\cos^2\theta+O\left(\widetilde{r}^3\right).
\end{align*}
As a result, we get
$$c_\star=\left(1+\widetilde{r}^a\right)\left[1-a\widetilde{r}\cos\theta+\left(\frac{a(a+2)}{2}\cos^2\theta-\frac{a}{2}\right)\widetilde{r}^2+O\left(\widetilde{r}^3\right)\right],$$
which yields that
$$c_\star-1=\widetilde{r}^a-a\widetilde{r}\cos\theta+\widetilde{r}^2\left(\frac{a(a+2)}{2}\cos^2\theta-\frac{a}{2}\right)+O\left(\widetilde{r}^{a+1}+\widetilde{r}^3\right).$$
Since $a>1$ and $\cos\theta\approx -1$, we conclude that $c_\star-1\sim \widetilde{r}$.\par
(iii) Finally, consider the case $\sin\theta\gtrsim1$. We set $\delta\triangleq c_\star-1$, with $0<\delta\ll1$, and denote by $\theta_0$ the solution when $c_\star=1$. That is,
$$\cos\theta_0=\frac{\left(1+\widetilde{r}^a\right)^{2/a}-1-\widetilde{r}^2}{2\widetilde{r}}.$$
It follows that
\begin{align*}
    &1+\widetilde{r}^2+2\widetilde{r}\cos\theta=\left(1+\widetilde{r}^a\right)^{2/a}\left(1+\delta\right)^{-2/a}\\
    \Rightarrow\qquad&2\widetilde{r}\left(\cos\theta-\cos\theta_0\right)=\left(1+\widetilde{r}^a\right)^{2/a}\left[\left(1+\delta\right)^{-2/a}-1\right]\\
    \Rightarrow\qquad&\cos\theta-\cos\theta_0\sim\frac{\left(1+\widetilde{r}^a\right)^{2/a}}{\widetilde{r}}\delta.
\end{align*}
Since $\sin\theta\gtrsim 1$, we have $\cos\theta_0-\cos\theta\sim\theta-\theta_0$. As a result, we obtain
$$0<\theta-\theta_0\sim\frac{\left(1+\widetilde{r}^a\right)^{2/a}}{\widetilde{r}}(c_\star-1)\sim \frac{c_\star-1}{\widetilde{r}},$$
where we used $\left(1+\widetilde{r}^a\right)^{2/a}\sim 1$ due to $0<\widetilde{r}<1$.
\end{proof}

\begin{cor}
     Assume $1<a\le 5$, $0<1-c_\star\ll 1$ and $\displaystyle{0<\widetilde{r}=\frac{r}{|p|}<1}$. We Still denote by $\theta$ the angle between $p$ and $p_1$. Then the following statements hold:\begin{enumerate}[label=(\roman*),itemsep=6pt]
        \item If $\theta\approx 0$, then $1-c_\star\sim \widetilde{r}$;
        \item The case $\theta\approx\pi$ is impossible;
        \item If $\sin\theta\gtrsim 1$ and $1-c_\star\lesssim\widetilde{r}$, then we have $\displaystyle{\Delta\theta\sim\frac{1-c_\star}{\widetilde{r}}}$. In other words, the length of the admissible $\theta$-interval is of order $\displaystyle{\frac{1-c_\star}{\widetilde{r}}}$.
    \end{enumerate}
\end{cor}
\begin{proof}
    The proof is identical to that of Lemma 6.4.
\end{proof}

\begin{remark}
When $\widetilde r>1$, set
$$\overline{r}=\frac{|p|}{r}=\frac{1}{\widetilde{r}}.$$
Then $c_\star(\widetilde{r})=c_\star(\overline{r})$. Hence, after exchanging the roles of $\widetilde r$ and $\overline r$, we can obtain the corresponding results parallel to Lemma 6.4 and Corollary 6.5.
\end{remark}

\begin{lemma}
    Assume $a=1$, $0<c_\star-1\ll 1$ and $\displaystyle{0<\widetilde{r}=\frac{r}{|p|}<1}$. As before, we still\vspace{0.5em} denote by $\theta$ the angle between $p$ and $p_1$. Then there exists $0<\eps\ll 1$ such that the following statements hold:\begin{enumerate}[label=(\roman*),itemsep=6pt]
        \item If $\eps\le\theta<\pi$, then $c_\star-1\sim \widetilde{r}$;
        \item If $0<\theta<\eps$, then we have $\displaystyle{\Delta\theta\sim\sqrt{\frac{c_\star-1}{\widetilde{r}}}}$. In other words, the length of the admissible $\theta$-interval is of order $\displaystyle{\sqrt{\frac{c_\star-1}{\widetilde{r}}}}$.
    \end{enumerate}
\end{lemma}
\begin{proof}
Recall that in the case of $a=1$, we have
$$c_\star=\frac{|p|+|p_1|}{|p+p_1|}=\frac{1+\widetilde{r}}{\sqrt{1+\widetilde{r}^2+2\widetilde{r}\cos\theta}}=\frac{1+\widetilde{r}}{\sqrt{\left(1+\widetilde{r}\right)^2-2\widetilde{r}(1-\cos\theta)}}.$$
This yields that
\begin{align}
    c_\star-1&=\frac{1+\widetilde{r}-\sqrt{\left(1+\widetilde{r}\right)^2-2\widetilde{r}(1-\cos\theta)}}{\sqrt{\left(1+\widetilde{r}\right)^2-2\widetilde{r}(1-\cos\theta)}}\notag\\
    &=\frac{2\widetilde{r}(1-\cos\theta)}{\sqrt{\left(1+\widetilde{r}\right)^2-2\widetilde{r}(1-\cos\theta)}\cdot\left[1+\widetilde{r}+\sqrt{\left(1+\widetilde{r}\right)^2-2\widetilde{r}(1-\cos\theta)}\right]}\notag\\
    &\sim \widetilde{r}(1-\cos\theta),\label{6.6}
\end{align}
where in the last step we used that 
$$1\sim 1+\widetilde{r}\approx\sqrt{(1+\widetilde{r})^2-2\widetilde{r}(1-\cos\theta)}$$
due to the fact that
$$c_\star=\frac{1+\widetilde{r}}{\sqrt{\left(1+\widetilde{r}\right)^2-2\widetilde{r}(1-\cos\theta)}}\approx 1.$$\par
Now, if $\eps\le\theta<\pi$, then $1-\cos\theta\sim 1$. As a result, (\ref{6.6}) gives 
$$c_\star-1\sim \widetilde{r}.$$
On the other hand, if $0<\theta<\eps$, then $1-\cos\theta\sim\theta^2$. In view of (\ref{6.6}), we obtain $c_\star-1\sim\widetilde{r}\theta^2$, which is equivalent to 
$$\Delta\theta\sim\sqrt{\frac{c_\star-1}{\widetilde{r}}}.$$
\end{proof}

\begin{remark}
  \ \par  
    \begin{enumerate}[label=(\roman*),itemsep=6pt]
        \item When $a=1$, the condition $c_\star<1$ cannot occur. Therefore, in the endpoint case $a=1$, it suffices to consider the regime $c_\star\geq 1$.
        \item When $\widetilde r>1$, set 
        $$\overline{r}=\frac{|p|}{r}=\frac{1}{\widetilde{r}}$$
        as before. Since $c_\star(\widetilde r)=c_\star(\overline r)$, the corresponding results follow from the case $\widetilde r<1$ by exchanging the roles of $\widetilde r$ and $\overline r$.
    \end{enumerate}
\end{remark}

\section*{Acknowledgements}

The author would like to thank Yu Deng for suggesting this problem and his comments on an earlier version of this manuscript.

\bibliographystyle{plain}
\bibliography{references}

\end{document}